\pdfoutput=1
\documentclass[12pt]{article}

\usepackage[T1]{fontenc}
\usepackage[utf8]{inputenc}
\usepackage[english]{babel}

\usepackage[letterpaper,margin=1.05in,bottom=1.65in]{geometry}
\usepackage[onehalfspacing]{setspace}
\usepackage[hidelinks]{hyperref}
\usepackage[format=plain,font=small,labelfont=bf,textfont=up]{caption}

\usepackage{amsmath}
\usepackage{amssymb}
\usepackage{amsthm}
\usepackage{mathtools}
\mathtoolsset{showonlyrefs,showmanualtags}

\usepackage{bbm}

\usepackage{booktabs}
\usepackage{graphicx}
\usepackage{enumitem}
\usepackage{natbib}

\newcommand{\simulationpath}[1]{#1}
\newcommand{\figurepath}[1]{#1}

\newcommand{\SUTVA}{\textsc{sutva}}
\newcommand{\RMSE}{\textsc{rmse}}

\newcommand{\Reals}{\mathbb{R}}

\newcommand{\bfU}{\mathbf{U}}

\newcommand{\z}{\mathbf{z}}
\newcommand{\Z}{\mathbf{Z}}
\newcommand{\suppZ}{\mathcal{Z}}

\newcommand{\Zint}{\widetilde{\Z}}
\newcommand{\suppZint}{\widetilde{\suppZ}}

\newcommand{\icout}{c}
\newcommand{\icoutmom}[1]{C_{#1}}
\newcommand{\icoutavg}{\icoutmom{1}}
\newcommand{\icoutrms}{\icoutmom{2}}
\newcommand{\icoutmsq}{\icoutrms^2}
\newcommand{\icoutmax}{\icoutmom{\infty}}

\newcommand{\rotic}[2]{\rotatebox[origin=c]{180}{$#1\icout$}}
\newcommand{\rotatedic}{{\mathpalette\rotic\relax}}
\newcommand{\icin}{\rotatedic}

\newcommand{\di}{d}
\newcommand{\davg}{\di_{\normalfont\textsc{avg}}}
\newcommand{\drms}{\di_{\normalfont\textsc{rms}}}
\newcommand{\dmax}{\di_{\normalfont\textsc{max}}}

\newcommand{\eigmax}{\lambda_{1}}

\newcommand{\teconst}{k_\tau}

\newcommand{\ei}{e}
\newcommand{\eavg}{\ei_{\normalfont\textsc{avg}}}
\newcommand{\Rsum}{R_{\normalfont\textsc{sum}}}

\newcommand{\mxint}{\alpha_{\normalfont\textsc{int}}}
\newcommand{\mxext}{\alpha_{\normalfont\textsc{ext}}}

\DeclareMathOperator*{\balop}{bal}

\DeclarePairedDelimiterXPP\bal[1]{\balop}{\lparen}{\rparen}{}{#1}

\newcommand{\tmu}{\mu_z}
\newcommand{\cmu}{\breve{\mu}_z}

\newcommand{\ATE}{\textsc{ate}}
\newcommand{\eATE}{\tau_{\normalfont\ATE}}

\newcommand{\EATE}{\textsc{eate}}
\newcommand{\eEATE}{\tau_{\normalfont\EATE}}

\newcommand{\eMSTE}{\tau_{\normalfont\textsc{msq}}}

\newcommand{\eEATEp}{\tau_{\normalfont \textsc{p}\text{-}\EATE}}
\newcommand{\eEATEq}{\tau_{\normalfont \textsc{q}\text{-}\EATE}}

\newcommand{\ADSE}{\textsc{adse}}
\newcommand{\eADSE}{\tau_{\normalfont\ADSE}}

\newcommand{\LATE}{\textsc{late}}

\newcommand{\sote}{\xi}

\newcommand{\hmu}{\hat{\mu}_z}
\newcommand{\hn}{\hat{n}_z}

\newcommand{\HT}{\textsc{ht}}
\newcommand{\eHT}{\hat{\tau}_{\normalfont\HT}}

\newcommand{\HA}{\textsc{h\'a}}
\newcommand{\eHA}{\hat{\tau}_{\normalfont\HA}}

\newcommand{\eVarBer}{\VarEst{ber}{\eHT}}
\newcommand{\eVarAvg}{\VarEst{avg}{\eHT}}
\newcommand{\eVarMax}{\VarEst{max}{\eHT}}
\newcommand{\eVarSR}{\VarEst{sr}{\eHT}}

\theoremstyle{plain}
\newtheorem{corollary}{Corollary}
\newtheorem{proposition}{Proposition}
\newtheorem{lemma}{Lemma}

\newtheorem*{nonum-proposition}{Proposition}

\newtheorem{appcorollary}{Corollary}

\newtheorem{applemma}{Lemma}

\newtheorem{innerrefprop}{Proposition}
\newenvironment{refprop}[1]
{\renewcommand\theinnerrefprop{#1}\innerrefprop}
{\endinnerrefprop}

\newtheorem{innerreflem}{Lemma}
\newenvironment{reflem}[1]
{\renewcommand\theinnerreflem{#1}\innerreflem}
{\endinnerreflem}

\newtheorem{innerrefcoro}{Corollary}
\newenvironment{refcoro}[1]
{\renewcommand\theinnerrefcoro{#1}\innerrefcoro}
{\endinnerrefcoro}

\theoremstyle{definition}
\newtheorem{assumption}{Assumption}
\newtheorem{definition}{Definition}

\newtheorem{appdefinition}{Definition}

\theoremstyle{remark}
\newtheorem{remark}{Remark}

\newtheorem*{appremark}{Remark}

\newcommand{\bbone}{\mathbbm{1}}

\DeclareMathOperator{\Eop}{E}
\DeclareMathOperator{\Varop}{Var}
\DeclareMathOperator{\Covop}{Cov}

\DeclareMathOperator{\bigOop}{\mathcal{O}}
\DeclareMathOperator{\littleOop}{o}
\DeclareMathOperator{\bigOpop}{\mathcal{O}_p}
\DeclareMathOperator{\littleOpop}{o_p}
\DeclareMathOperator{\bigOmegaop}{\Omega}
\DeclareMathOperator{\littleOmegaop}{\omega}

\DeclareMathOperator*{\argmin}{arg\,min}

\newcommand{\darrow}{\overset{d}{\longrightarrow}}
\newcommand{\parrow}{\overset{p}{\longrightarrow}}

\makeatletter
\newcommand*{\@indep}[2]{\mathrel{\rlap{$#1#2$}\mkern2mu{#1#2}}}
\newcommand{\indep}{\protect\mathpalette{\protect\@indep}{\perp}}
\makeatother

\makeatletter
\newcommand*{\transpose}{{\mathpalette\@transpose{}}}
\newcommand*{\@transpose}[2]{\raisebox{\depth}{$\m@th#1\intercal$}}
\makeatother

\renewcommand{\complement}{\mathsf{c}}

\providecommand\given{}
\newcommand\givensymbol[1]{\nonscript\:#1\vert\allowbreak\nonscript\:\mathopen{}}

\DeclarePairedDelimiter\paren{\lparen}{\rparen}
\DeclarePairedDelimiter\bracket{\lbrack}{\rbrack}
\DeclarePairedDelimiter\braces{\{}{\}}
\DeclarePairedDelimiter\abs{\lvert}{\rvert}
\DeclarePairedDelimiter\norm{\lVert}{\rVert}
\DeclarePairedDelimiter\floor{\lfloor}{\rfloor}
\DeclarePairedDelimiter\ceil{\lceil}{\rceil}

\DeclarePairedDelimiterXPP\indicator[1]{\bbone}{\lbrack}{\rbrack}{}{#1}

\DeclarePairedDelimiterXPP\prob[1]{\Pr}{\lparen}{\rparen}{}{%
	\renewcommand\given{\givensymbol{\delimsize}}%
	#1}

	\DeclarePairedDelimiterXPP\probp[1]{{\Pr}_{P}}{\lparen}{\rparen}{}{%
	\renewcommand\given{\givensymbol{\delimsize}}%
	#1}

\DeclarePairedDelimiterXPP\probq[1]{{\Pr}_{Q}}{\lparen}{\rparen}{}{%
	\renewcommand\given{\givensymbol{\delimsize}}%
	#1}

\DeclarePairedDelimiterXPP\E[1]{\Eop}{\lbrack}{\rbrack}{}{%
	\renewcommand\given{\givensymbol{\delimsize}}%
	#1}

\DeclarePairedDelimiterXPP\Ep[1]{\Eop_{P}}{\lbrack}{\rbrack}{}{%
	\renewcommand\given{\givensymbol{\delimsize}}%
	#1}

\DeclarePairedDelimiterXPP\Eq[1]{\Eop_{Q}}{\lbrack}{\rbrack}{}{%
	\renewcommand\given{\givensymbol{\delimsize}}%
	#1}

\DeclarePairedDelimiterXPP\Ej[1]{\Eop_{J}}{\lbrack}{\rbrack}{}{%
	\renewcommand\given{\givensymbol{\delimsize}}%
	#1}

\DeclarePairedDelimiterXPP\Var[1]{\Varop}{\lparen}{\rparen}{}{%
	\renewcommand\given{\givensymbol{\delimsize}}%
	#1}

\DeclarePairedDelimiterXPP\VarEst[2]{\widehat{\Varop}_{\normalfont\textsc{#1}}}{\lparen}{\rparen}{}{%
	\renewcommand\given{\givensymbol{\delimsize}}%
	#2}

\DeclarePairedDelimiterXPP\Cov[1]{\Covop}{\lparen}{\rparen}{}{%
	\renewcommand\given{\givensymbol{\delimsize}}%
	#1}

\DeclarePairedDelimiterXPP\bigO[1]{\bigOop}{\lparen}{\rparen}{}{#1}
\DeclarePairedDelimiterXPP\littleO[1]{\littleOop}{\lparen}{\rparen}{}{#1}

\DeclarePairedDelimiterXPP\bigOp[1]{\bigOpop}{\lparen}{\rparen}{}{#1}
\DeclarePairedDelimiterXPP\littleOp[1]{\littleOpop}{\lparen}{\rparen}{}{#1}

\DeclarePairedDelimiterXPP\bigOmega[1]{\bigOmegaop}{\lparen}{\rparen}{}{#1}
\DeclarePairedDelimiterXPP\littleOmega[1]{\littleOmegaop}{\lparen}{\rparen}{}{#1}

\DeclarePairedDelimiterXPP\bigTheta[1]{\Theta}{\lparen}{\rparen}{}{#1}

\allowdisplaybreaks

\title{\textbf{Average treatment effects in the presence of unknown interference}}

\author{%
	Fredrik Sävje\thanks{Departments of Political Science and Statistics \& Data Science, Yale University.}
	\and
	Peter M.\ Aronow\thanks{Departments of Political Science, Biostatistics and Statistics \& Data Science, Yale University.}
	\and
	Michael G.\ Hudgens\thanks{Department of Biostatistics, University of North Carolina, Chapel Hill.}
}

\date{\today}

\newcommand{\acknowledgements}{We thank Alexander D'Amour, Matthew Blackwell, David Choi, Forrest Crawford, Peng Ding, Naoki Egami, Avi Feller, Owen Francis, Elizabeth Halloran, Lihua Lei, Cyrus Samii, Jasjeet Sekhon and Daniel Spielman for helpful suggestions and discussions. A previous version of this article was circulated under the title ``A folk theorem on interference in experiments.''}

\begin{document}

\ifdefined\acknowledgements%
\makeatletter%
\begin{NoHyper}\gdef\@thefnmark{}\@footnotetext{\hspace{-1em}\acknowledgements}\end{NoHyper}%
\makeatother%
\fi%

\maketitle

\bigskip
\begin{abstract}
\begin{singlespace}
\noindent We investigate large-sample properties of treatment effect estimators under unknown interference in randomized experiments.
The inferential target is a generalization of the average treatment effect estimand that marginalizes over potential spillover effects.
We show that estimators commonly used to estimate treatment effects under no interference are consistent for the generalized estimand for several common experimental designs under limited but otherwise arbitrary and unknown interference.
The rates of convergence depend on the rate at which the amount of interference grows and the degree to which it aligns with dependencies in treatment assignment.
Importantly for practitioners, the results imply that if one erroneously assumes that units do not interfere in a setting with limited, or even moderate, interference, standard estimators are nevertheless likely to be close to an average treatment effect if the sample is sufficiently large.
Conventional confidence statements may, however, not be accurate.

\end{singlespace}
\end{abstract}

\clearpage

\section{Introduction}

Investigators of causality routinely assume that subjects under study do not interfere with each other.
The no-interference assumption is so ingrained in the practice of causal inference that its application is often left implicit.
Yet, interference appears to be at the heart of the social and medical sciences.
Humans interact, and that is precisely the motivation for much of the research in these fields.
The assumption is common because investigators believe that their methods require it.
For example, in a recent textbook, \citet{Imbens2015} write: ``causal inference is generally impossible without no-interference assumptions.''
This sentiment provides the motivation for the current study.

	We investigate to what degree one can weaken the assumption of no interference and still draw credible inferences about causal relationships. We find that, indeed, causal inferences are impossible without assumptions about the interference, but the prevailing view severely exaggerates the issue. One can allow for moderate amounts of interference, and one can allow the subjects to interfere in unknown and largely arbitrary ways.

	Our focus is estimation of average treatment effects in randomized experiments. A random subset of a sample of units is assigned to some treatment, and the quantity of interest is the average effect of being assigned these treatments. The no-interference assumption in this context is the restriction that no unit's assignment affects other units' outcomes. We consider the setting where such \emph{spillover effects} exist, and in particular, when the form they may take is left unspecified.

The paper makes four main contributions.
We first introduce an estimand---the \emph{expected average treatment effect} or \EATE---that generalizes the conventional \emph{average treatment effect} (\ATE) to settings with interference.
The conventional estimand is not well-defined when units interfere because a unit's outcome may be affected by more than one treatment.
We resolve the issue by marginalizing the effects of interest over the assignment distribution of the incidental treatments.
That is, for a given reference assignment, one may ask how a particular unit's outcome is affected when only its own treatment is changed.
An unambiguous average treatment effect is defined by asking the same for each unit in the experiment and averaging the resulting unit-level effects.
While unambiguous, this average effect depends on which assignment was used as reference, and the result may be different if the exercise is repeated for another assignment.
To capture the typical treatment effect in the experiment, \EATE\ marginalizes the effects over all possible reference assignments.
The estimand is a generalization of \ATE\ in the sense that they coincide whenever the latter is well-defined.

The second contribution is to demonstrate that \EATE\ can be estimated consistently under weak restrictions on the interference and without structural knowledge thereof.
Focus is on the standard Horvitz-Thompson and H\'ajek estimators.
The results also apply to the widely used difference-in-means and ordinary least squares estimators, as they are special cases of the H\'ajek estimator.
The investigation starts with the Bernoulli and complete randomization experimental designs.
We show that the estimators are consistent for \EATE\ under the designs as long as the average amount of interference grows at a sufficiently slow rate (according to measures we define shortly).
Root-$n$ consistency is achieved whenever the average amount of interference is bounded.
The investigation then turns to the paired randomization design.
The design introduces perfectly correlated treatment assignments, and we show that this can make the estimators unstable even when the interference is limited.
The degree to which the dependencies introduced by the experimental design align with the interference structure must be restricted to achieve consistency.
Information about the interference structure beyond the aggregated restrictions is, however, still not needed.
The insights from the paired design extend to a more general setting, and similar restrictions result in consistency under arbitrary experimental designs.

The third contribution is to investigate the prospects of variance estimation.
We show that conventional variance estimators generally do not capture the loss of precision that may result from interference.
Confidence statements based on these estimators may therefore be misleading.
To address the concern, we construct three alternative estimators by inflating a conventional estimator with measures of the amount of interference similar to those used to show consistency.
The alternative estimators are shown to be asymptotically conservative under suitable regularity conditions.

The final contribution is to investigate to what degree \EATE\ generalizes to other designs.
The estimand marginalizes over the design actually used in the experiment, and a consequence is that the estimand may have taken a different value if another design were used.
We show that the estimand is informative of the effect for designs that are close to the one that was implemented under suitable regularity conditions.
When the amount of interference is limited, the estimands may converge.

	The findings are of theoretical interest as they shine light on the limits of causal inference under interference. They are also of practical interest. The results pertain to standard estimators under standard experimental designs. As such, they apply to many previous studies where interference might have been present, but where it was assumed not to be. Studies that mistakenly assume that units do not interfere might, therefore, not necessarily be invalidated. No-interference assumptions are, for example, common in experimental studies of voter mobilization \citep[see][and the references therein]{Green2004}. However, a growing body of evidence suggests that potential voters interact within households, neighborhoods and other social structures \citep{Nickerson2008,Aronow2012,Sinclair2012Detecting}. Interference is, in other words, likely present in this setting, and researchers have been left uncertain about the interpretation of existing findings. Our results provide a lens through which the evidence can be interpreted; the reported estimates capture expected average treatment effects.

\section{Related work}

Our investigation builds on a recent literature on causal inference under interference \citep[see][for a review]{Halloran2016}.
The no-interference assumption itself is due to \citet{Cox1958}.
The iteration that is most commonly used today was, however, formulated by \citet{Rubin1980} as a part of the \emph{stable unit treatment variation assumption}, or \SUTVA.
Early departures from this assumption were modes of analysis inspired by Fisher's exact randomization test \citep{Fisher1935}.
The approach employs sharp null hypotheses that stipulates the outcomes of all units under all assignments.
The most common such hypothesis is simply that treatment is inconsequential so the observed outcomes are constant over all assignments.
As this subsumes that both primary and spillover effects do not exist, the approach tests for the existence of both types of effects simultaneously.
The test has recently been adapted and extended to study interference specifically \citep[see, e.g.,][]{Rosenbaum2007,Aronow2012,Luo2012,Bowers2013,Choi2017,Athey2018Exact,Basse2019Randomization}.

Early methods for point estimation restricted the interference process through structural models and thereby presumed that interactions took a particular form \citep{Manski1993Identification}.
The structural approach has been extended to capture effects under weaker assumptions in a larger class of interference processes \citep{Lee2007,Graham2008,Bramoulle2009}.
Still, the approach has been criticized for being too restrictive \citep{Goldsmith2013,Angrist2014}.

A strand of the literature closer to the current study relaxes the structural assumptions.
Interference is allowed to take arbitrary forms as long as it is contained within known and disjoint groups of units.
The assumption is known as \emph{partial interference} \citep[see, e.g.,][]{Hudgens2008,Tchetgen2012,Liu2014,Rigdon2015,Kang2016,Liu2016,Basse2018Analyzing}.
While partial interference allows for some progress on its own, it is often coupled with \emph{stratified interference}.
The additional assumption stipulates that the only relevant aspect of the interference is the proportion of treated units in the groups.
The identities of the units are, in other words, inconsequential for the spillover effects.
Much like the structural approach, stratified interference restricts the form the interference can take.

More recent contributions have focused on relaxing the partial interference assumption.
Interference is not restricted to disjoint groups, and units are allowed to interfere along general structures such as social networks \citep[see, e.g.,][]{Manski2013,Toulis2013,Ugander2013,Eckles2016,Aronow2017,Forastiere2017,Jagadeesan2017Designs,Ogburn2017,Sussman2017,Basse2018Model}.
This allows for quite general forms of interactions, but the suggested estimation methods require detailed knowledge of the interference structure.

Previous investigations under unknown interference have primarily focused on the mean the sampling distribution of various estimators.
\citet{Sobel2006} derives the expectation of an instrumental variables estimator used in housing mobility experiments under unknown interference, showing that it is a mixture of primary and spillover effects for compilers and non-compilers.
\citet{Hudgens2008} derive similar results for the average distributional shift effect discussed below.
\citet{Egami2017} investigates a setting where the interference can be described by a set of networks.
This framework includes a stratified interference assumption, it admits quite general forms of interference since the networks are allowed to be overlapping and partially unobserved.
However, unlike the focus in this paper, these contributions either do not discuss the precision and consistency of the investigated estimators, or they only do so after assuming that the interference structure is known.

One exception is a study by \citet{Basse2018Limitations}.
The authors consider average treatment effects under arbitrary and unknown interference just as we do.
They, however, focus on inference about the contrast between the average outcome when all units are treated and the average outcome when no unit is treated.
As we discuss below, this estimand provides a different description of the causal setting than \EATE.
The authors show that no consistent estimator exists for their estimand under conditions similar to those used in this paper even when the interference structure is known.

	\section{Treatment effects under interference}

	Consider a sample of $n$ units indexed by the set $\bfU=\braces{1,2,\dotsc,n}$. An experimenter intervenes on the world in ways that potentially affect the units. The intervention is described by a $n$-dimensional binary vector $\z = \paren{z_1, z_2, \dotsc, z_n}\in\braces{0,1}^n$. A particular value of $\z$ could, for example, denote that some drug is given to a certain subset of the units in $\bfU$. We are particularly interested in how unit $i$ is affected by the $i$th dimension of $\z$. For short, we say that $z_i$ is unit $i$'s treatment.

	The effects of different interventions are defined as comparisons between the outcomes they produce. Each unit has a function $y_i \colon \braces{0,1}^n \to \Reals$ denoting the observed outcome for the unit under a specific (potentially counterfactual) intervention \citep{Neyman1923,Holland1986}. In particular, $y_i(\z)$ is the response of $i$ when the intervention is $\z$. We refer to the elements of the image of this function as \emph{potential outcomes}. It will prove convenient to write the potential outcomes in a slightly different form. Let $\z_{-i} = \paren{z_1, \dotsc, z_{i-1}, z_{i+1}, \dotsc, z_n}$ denote the $(n-1)$-element vector constructed by deleting the $i$th element from $\z$. The potential outcome $y_i(\z)$ can then be written as $y_i(z_i; \z_{-i})$.

	We assume that the potential outcomes are well-defined throughout the paper. The assumption implies that the manner in which the experimenter manipulates $\z$ is inconsequential; no matter how $\z$ came to take a particular value, the outcome is the same.
	Well-defined potential outcomes also imply that no physical law or other circumstances prohibit $\z$ to take any value in $\braces{0,1}^n$. This ensures that the potential outcomes are, indeed, potential. However, the assumption does not restrict the way the experimenter chooses to intervene on the world, and some interventions may have no probability of being realized.

	The experimenter sets $\z$ according to a random vector $\Z = \paren{Z_1, \dotsc, Z_n}$. The probability distribution of $\Z$ is the design of the experiment. The design is the sole source of randomness we will consider. Let $Y_i$ denote the observed outcome of unit $i$. The observed outcome is a random variable connected to the experimental design through the potential outcomes: $Y_i = y_i(\Z)$. As above, $\Z_{-i}$ denotes $\Z$ net of its $i$th element, so $Y_i = y_i(Z_i; \Z_{-i})$.

	\subsection{Expected average treatment effects}

	It is conventional to assume that the potential outcomes are restricted so a unit's outcome is only affected by its own treatment. That is, for any two assignments $\z$ and $\z'$, if a unit's treatment is the same for both assignments, then its outcome is the same. This no-interference assumption admits a definition of the treatment effect for unit $i$ as the contrast between its potential outcomes when we change its treatment:
	\begin{equation}
	\tau_i = y_i(1; \z_{-i}) - y_i(0; \z_{-i}),
	\end{equation}
	where $\z_{-i}$ is any value in $\braces{0,1}^{n-1}$. No interference implies that the choice of $\z_{-i}$ is inconsequential for the values of $y_i(1; \z_{-i})$ and $y_i(0; \z_{-i})$. The variable can therefore be left free without ambiguity, and it is common to use $y_i(z)$ as a shorthand for $y_i(z; \z_{-i})$. The average of the unit-level effects is the quantity experimenters often use to summarize treatment effects.

	\begin{definition} \label{def:ate}
		Under no interference, the \emph{average treatment effect} (\ATE) is the average unit-level treatment effect:
		\begin{equation}
		\eATE = \frac{1}{n} \sum_{i=1}^{n} \tau_i.
		\end{equation}
	\end{definition}

	The definition requires no interference. References to \emph{the} effect of a unit's treatment become ambiguous when units interfere since $\tau_i$ may vary under permutations of $\z_{-i}$. The ambiguity is contagious; the average treatment effect similarly becomes ill-defined.

To resolve the issue, we redefine the unit-level treatment effect for unit $i$ as the contrast between its potential outcomes when we change its treatment while holding all other treatments fixed at a given assignment $\z_{-i}$.
We call this quantity the \emph{assignment-conditional unit-level treatment effect}:
\begin{equation}
	\tau_i(\z_{-i}) = y_i(1; \z_{-i}) - y_i(0; \z_{-i}).
\end{equation}
To the best of our knowledge, this type of unit-level effect was first discussed by \citet{Halloran1995}.
The assignment-conditional effect differs from $\tau_i$ only in that the connection to other units' treatments is made explicit.
The redefined effect acknowledges that a unit's treatment may affect the unit differently depending on the treatments assigned to other units.
The change makes the unit-level effects unambiguous, and their average produces a version of the average treatment effect that remains well-defined under interference.

	\begin{definition} \label{def:acate}
		An \emph{assignment-conditional average treatment effect} is the average of the assignment-conditional unit-level treatment effect under a given assignment:
		\begin{equation}
		\eATE(\z) = \frac{1}{n} \sum_{i=1}^{n} \tau_i(\z_{-i}).
		\end{equation}
	\end{definition}

	The redefined effects are unambiguous under interference, but they are unwieldy. An average effect exists for each assignment, so their numbers grow exponentially in the sample size. Experimenters may for this reason not find it useful to study them individually. Similar to how unit-level effects are aggregated to an average effect, we focus on a summary of the assignment-conditional effects.

	\begin{definition} \label{def:eate}
		The \emph{expected average treatment effect} (\EATE) is the expected assignment-conditional average treatment effect:
		\begin{equation}
		\eEATE = \E{\eATE(\Z)},
		\end{equation}
		where the expectation is taken over the distribution of $\Z$ given by the experimental design.
	\end{definition}

The expected average treatment effect is a generalization of \ATE\ in the sense that the two estimands coincide whenever the no-interference assumption holds.
Under no interference, $\eATE(\z)$ does not depend on $\z$, so the marginalization is inconsequential.
When units interfere, $\eATE(\z)$ does depend on $\z$.
The random variable $\eATE(\Z)$ describes the distribution of average treatment effects under the implemented experimental design.
\EATE\ provides the best description of this distribution in mean square sense.

	\subsection{Related definitions}\label{sec:related-definitions}

	The \EATE\ estimand builds on previously proposed ideas. An estimand introduced by \citet{Hudgens2008} resolves the ambiguity of treatment effects under interference in a similar fashion. The authors refer to the quantity as the \emph{average direct causal effect}, but we opt for another name to highlight how it differs from \ATE\ and \EATE.

	\begin{definition} \label{def:adse}
		The \emph{average distributional shift effect} (\ADSE) is the average difference between the conditional expected outcomes for the two treatment conditions:
		\begin{equation}
		\eADSE = \frac{1}{n} \sum_{i=1}^{n} \paren[\Big]{\E{Y_i \given Z_i = 1} - \E{Y_i \given Z_i = 0}}.
		\end{equation}
	\end{definition}

	Similar to \EATE, the average distributional shift effect marginalizes the potential outcomes over the experimental design. The estimands differ in which distributions they use for the marginalization. The expectation in \EATE\ is over the unconditional assignment distribution, while \ADSE\ marginalizes each potential outcome separately over different conditional distributions.
	The difference becomes clear when the estimands are written in similar forms:
	\begin{align}
	\eEATE &=  \frac{1}{n} \sum_{i=1}^{n} \paren[\Big]{\E{y_i(1; \Z_{-i})} - \E{y_i(0; \Z_{-i})}},
	\\
	\eADSE &= \frac{1}{n} \sum_{i=1}^{n} \paren[\Big]{\E{y_i(1; \Z_{-i}) \given Z_i = 1} - \E{y_i(0; \Z_{-i}) \given Z_i = 0}}.
	\end{align}

	The two estimands provide different causal information. \EATE\ captures the expected effect of changing a random unit's treatment in the current experiment. It is the expected average unit-level treatment effect. \ADSE\ is the expected average effect of changing from an experimental design where we hold a unit's treatment fixed at $Z_i = 1$ to another design where its treatment is fixed at $Z_i = 0$. That is, the estimand captures the compound effect of changing a unit's treatment and simultaneously changing the experimental design. As a result, \ADSE\ may be non-zero even if all unit-level effects are exactly zero. That is, we may have $\eADSE\neq 0$ when $\tau_i(\z_{-i})=0$ for all $i$ and $\z_{-i}$. \citet{Eck2018Randomization} use a similar argument to show that \ADSE\ may not correspond to causal parameters capturing treatment effects in structural models.

	\citet{VanderWeele2011} introduced a version of \ADSE\ that resolves the issue by conditioning both terms with the same value on $i$'s treatment. Their estimand is a conditional average of unit-level effects and, thus, mixes aspects of \EATE\ and \ADSE.

An alternative definition of average treatment effects under interference is the contrast between the average outcome when all units are treated and the average outcome when no unit is treated: $n^{-1}\sum_{i=1}^{n} [y_i(\mathbf{1}) - y_i(\mathbf{0})]$ where $\mathbf{1}$ and $\mathbf{0}$ are the unit and zero vectors.
This all-or-nothing effect coincides with the conventional \ATE\ in Definition \ref{def:ate} (and thus \EATE) whenever the no-interference assumption holds.
However, it does not coincide with \EATE\ under interference, and the estimands provide different descriptions of the causal setting.
\EATE\ captures the typical treatment effect in the experiment actually implemented, while the alternative estimand captures the effect of completely scaling up or down treatment.
As we noted in the previous section, no consistent estimator exists for the all-or-nothing effect in the context considered in this paper \citep{Basse2018Limitations}.

\section{Quantifying interference}\label{sec:quantifying-interference}

Our results do not require detailed structural information about the interference.
No progress can, however, be made if it is left completely unrestricted.
Section~\ref{sec:restricting-interference} discusses this in more detail.
The intuition is simply that the change of a single unit's treatment could lead to non-negligible changes in all units' outcomes when the interference is unrestricted.
The following definitions quantify the amount of interference and are the basis for our restrictions.

	We say that unit $i$ interferes with unit $j$ if changing $i$'s treatment changes $j$'s outcome under at least one treatment assignment. We also say that a unit interferes with itself even if its treatment does not affect its outcome. The indicator $I_{ij}$ denotes such interference:
	\begin{equation}
	I_{ij} = \left\{\begin{array}{ll}
	1 & \text{if } y_j(\z) \neq y_j(\z') \text{ for some } \z, \z' \in \braces{0,1}^n \text{ such that } \z_{-i} = \z'_{-i},
	\\[0.2em]
	1 & \text{if } i = j,
	\\[0.2em]
	0 & \text{otherwise}.
	\end{array}\right.
	\end{equation}
	The definition allows for asymmetric interference; unit $i$ may interfere with unit $j$ without the converse being true.

	The collection of interference indicators simply describes the interference structure in an experimental sample. The definition itself does not impose restrictions on how the units may interfere. In particular, the indicators do not necessarily align with social networks or other structures through which units are thought to interact. Experimenters do not generally have enough information about how the units interfere to deduce or estimate the indicators. Their role is to act as a basis for an aggregated summary of the interference.

	\begin{definition}[Interference dependence] \label{def:interference-dependence}
		\begin{equation}
		\davg = \frac{1}{n}\sum_{i=1}^{n}\sum_{j=1}^n\di_{ij},
		\qquad
		\text{where}
		\qquad
		\di_{ij} = \left\{\begin{array}{ll}
		1 & \text{if } I_{\ell i}I_{\ell j} = 1 \text{ for some } \ell \in \bfU,
		\\[0.2em]
		0 & \text{otherwise}.
		\end{array}\right.
		\end{equation}
	\end{definition}

The interference dependence indicator $\di_{ij}$ captures whether units $i$ and $j$ are affected by a common treatment.
That is, $i$ and $j$ are interference dependent if they interfere directly with each other or if some third unit interferes with both $i$ and $j$.
The sum $\di_{i} = \sum_{j=1}^n\di_{ij}$ gives the number of interference dependencies for unit $i$, so the unit-average number of interference dependencies is $\davg$.
The quantity acts as a measure of how close an experiment is to no interference.
Absence of interference is equivalent to $\davg=1$, which indicates that the units are only interfering with themselves.
At the other extreme, $\davg=n$ indicates that interference is complete in the sense that all pairs of units are affected by a common treatment.
If sufficiently many units are interference dependent (i.e., $\davg$ is large), small perturbations of the treatment assignments may be amplified by the interference and induce large changes in many units' outcomes.

Interference dependence can be related to simpler descriptions of the interference.
Consider the following definitions:
\begin{equation}
	\icout_i = \sum_{j=1}^n I_{ij} \qquad\quad \text{and}\qquad\quad \icoutmom{p} = \bracket[\bigg]{\frac{1}{n}\sum_{i = 1}^n \icout_i^{p}}^{1/p}.
\end{equation}
The first quantity captures how many units $i$ interferes with.
That is, if changing unit $i$'s assignment would change the outcome of five other units, $i$ is interfering with the five units and itself, so $\icout_i=6$.
Information about these quantities would be useful, but such insights are generally beyond our grasp.
The subsequent quantity is the $p$th moment of the unit-level interference count, providing more aggregated descriptions.
For example, $\icoutavg$ and $\icoutmom{2}$ are the average and root mean square of the unit-level quantities.
We write $\icoutmax$ for the limit of $\icoutmom{p}$ as $p \to \infty$, which corresponds to the maximum $\icout_i$ over $\bfU$.
These moments bound $\davg$ from below and above.

\begin{lemma}\label{lem:interference-measure-inequalities}
	$\max\paren[\big]{\icoutavg, n^{-1}\icoutmax^2} \leq \davg \leq \icoutmsq \leq \icoutmax^2$.
\end{lemma}

All proofs, including the one for Lemma~\ref{lem:interference-measure-inequalities}, are given in Supplement~A.
The lemma implies that we can use $\icoutmom{2}$ or $\icoutmax$, rather than $\davg$, to restrict the interference.
While such restrictions are stronger than necessary, the connection is useful as it may be more intuitive to reason about these simpler descriptions than about interference dependence.

\section{Large sample properties}\label{sec:large-sample-prop}

We consider an asymptotic regime inspired by \citet{Isaki1982}.
An arbitrary sequence of samples indexed by their sample size is investigated.
It is not assumed that the samples are drawn from some larger population or otherwise randomly generated.
Neither are they assumed to be nested.
That is, the samples are not necessarily related other than through the conditions stated below.
All quantities related to the samples, such as the potential outcomes and experimental designs, have their own sequences also indexed by $n$.
The indexing is, however, left implicit.
We investigate how two common estimators of average treatment effects behave as the sample size grows subject to conditions on these sequences.

	\begin{definition}[Horvitz-Thompson, \HT, and H\'ajek, \HA, estimators] \label{def:estimators}
		\begin{align}
		\eHT &= \frac{1}{n} \sum_{i=1}^{n}\frac{Z_i Y_i}{p_i} - \frac{1}{n} \sum_{i=1}^{n} \frac{(1-Z_i) Y_i}{1-p_i},
		\\[0.6em]
		\eHA &= \paren[\Bigg]{\sum_{i=1}^{n}\frac{Z_i Y_i}{p_i} \Bigg/ \sum_{i=1}^{n}\frac{Z_i}{p_i}} - \paren[\Bigg]{\sum_{i=1}^{n}\frac{(1-Z_i) Y_i}{1-p_i} \Bigg/ \sum_{i=1}^{n}\frac{1-Z_i}{1-p_i}},
		\end{align}
		where $p_i = \prob{Z_i = 1}$ is the marginal treatment probability for unit $i$.
	\end{definition}

Estimators of this form were first introduced in the sampling literature to estimate population means under unequal inclusion probabilities \citep{Horvitz1952,Hajek1971}.
They have since received much attention from the causal inference and policy evaluation literatures where they are often referred to as \emph{inverse probability weighted estimators} \citep[see, e.g.,][]{Hahn1998,Hirano2003,Hernan2006}.
Other estimators commonly used to analyze experiments, such as the difference-in-means and ordinary least squares estimators, are special cases of the H\'ajek estimator.
As a consequence, the results apply to these estimators as well.

	We assume throughout the paper that the experimental design and potential outcomes are sufficiently well-behaved as formalized in the following assumption.

\begin{assumption}[Regularity conditions] \label{ass:regularity-conditions}
	There exist constants $k <\infty$, $q \geq 2$ and $s \geq 1$ such that for all $i\in\bfU$ in the sequence of samples:
	\begin{enumerate}[label=\ref*{ass:regularity-conditions}\Alph*]
		\setlength\itemsep{0.5em}
		\item (Probabilistic assignment). $\; k^{-1} \leq \prob{Z_i = 1} \leq 1 - k^{-1}$, \label{ass:probabilistic-assignment}

		\item (Outcome moments). $\E[\big]{\abs{Y_i}^{q}} \leq k^q$, \label{ass:bounded-moments}

		\item (Potential outcome moments). $\E[\big]{\abs{y_i(z; \Z_{-i})}^{s}} \leq k^s$ for $z\in\braces{0,1}$. \label{ass:bounded-expected-pos}
	\end{enumerate}
\end{assumption}

	The first regularity condition restricts the experimental design so that each treatment is realized with a positive probability. The condition does not restrict combinations of treatments, and assignments may exist such that $\prob{\Z = \z}=0$. The second condition restricts the distributions of the observed outcomes so they remain sufficiently well-behaved asymptotically. The last condition restricts the potential outcomes slightly off the support of the experimental design and ensures that \EATE\ is well-defined asymptotically.

The exact values of $q$ and $s$ are inconsequential for the results in Section \ref{sec:common-designs}.
The assumption can, in that case, be written simply with $q=2$ and $s=1$.
However, the rate of convergence for an arbitrary experimental design depends on which moments exist, and variance estimation generally require $q \geq 4$.
The ideal case is when the potential outcomes themselves are bounded, in which case Assumption \ref{ass:regularity-conditions} holds as $q\to\infty$ and $s\to\infty$.

The two moment conditions are similar in structure, but neither is implied by the other.
Assumption \ref{ass:bounded-moments} does not imply \ref{ass:bounded-expected-pos} because the former is only concerned with the potential outcomes on the support of the experimental design.
The opposite implication does not hold because $s$ may be smaller than $q$.

\subsection{Restricting interference}\label{sec:restricting-interference}

	The sequence of $\davg$ describes the amount of interference in the sequence of samples. Our notion of limited interference is formalized as a restriction on this sequence.

	\begin{assumption}[Restricted interference] \label{ass:restricted-interference}
		$\davg=\littleO{n}$.
	\end{assumption}

	The assumption stipulates that units, on average, are interference dependent with an asymptotically diminishing fraction of the sample. It still allows for substantial amounts of interference. The unit-average number of interference dependencies may grow with the sample size. The total amount of interference dependencies may, thus, grow at a faster rate than $n$. It is only assumed that the unit-average does not grow proportionally to the sample size.

In addition to restricting the amount of interference, Assumption \ref{ass:restricted-interference} imposes weak restrictions on the structure of the interference.
It rules out that the interference is so unevenly distributed that a few units are interfering with most other units.
If the interference is concentrated in such a way, small perturbations of the assignments could be amplified through the treatments of those units.
At the extreme, a single unit interferes with all other units, and all units' outcomes would change if we were to change its treatment.
The estimators would not stabilize if the interference is structured in this way even if it otherwise was sparse.

Restricted interference is not sufficient for consistency.
Sequences of experiments exist for which the assumption holds but the estimators do not converge to \EATE.
Assumption~\ref{ass:restricted-interference} is, however, necessary for consistency of the \HT\ and \HA\ estimators in the following sense.

\begin{proposition}\label{prop:restricted-interference-necessary}
	For every sequence of experimental designs, if Assumption~\ref{ass:restricted-interference} does not hold, there exists a sequence of potential outcomes satisfying Assumption~\ref{ass:regularity-conditions} such that the \HT\ and \HA\ estimators do not converge in probability to \EATE.
\end{proposition}

The proposition implies that the weakest possible restriction on $\davg$ is Assumption~\ref{ass:restricted-interference}.
If a weaker restriction is imposed, for example, that $\davg$ is on the order of $\varepsilon n$ for some small $\varepsilon>0$, potential outcomes exist for any experimental design so that the relaxed interference restriction is satisfied but the estimators do not converge.
A consequence is that experimental designs themselves cannot ensure consistency.
We must somehow restrict the interference to make progress.
It might, however, be possible to achieve consistency without Assumption~\ref{ass:restricted-interference} if one imposes stronger regularity conditions or restricts the interference in some other way.
For example, the estimators may be consistent if the magnitude of the interference, according to some suitable measure, approaches zero.

	\subsection{Common experimental designs} \label{sec:common-designs}

	The investigation starts with three specific experimental designs. The designs are commonly used by experimenters and, thus, of interest in their own right. They also provide a good illustration of the issues that arise under unknown interference and set the scene for the investigation of arbitrary designs in the subsequent section.

\subsubsection{Bernoulli and complete randomization}

The simplest experimental design assigns treatment independently.
The experimenter flips a coin for each unit and administers treatment accordingly.
We call this a \emph{Bernoulli randomization design}, and it satisfies
\begin{equation}
	\prob{\Z = \z} = \prod_{i=1}^{n} p_i^{z_i}\paren{1-p_i}^{1-z_i}
\end{equation}
for some set of assignment probabilities $p_1, p_2, \dotsc, p_n$ bounded away from zero and one.

	The outcomes of any two units are independent under a Bernoulli design when the no-interference assumption holds. This is not the case when units interfere. A single treatment may then affect two or more units, and the corresponding outcomes are dependent. That is, two units' outcomes are dependent when they are interference dependent according to Definition \ref{def:interference-dependence}. Restricting this dependence ensures that the effective sample size grows with the nominal size and grants consistency.

	\begin{proposition}\label{prop:bernoulli-rates}
		With a Bernoulli randomization design under restricted interference (Assumption \ref{ass:restricted-interference}), the \HT\ and \HA\ estimators are consistent for \EATE\ and converge at the following rates:
		\begin{equation}
		\eHT - \eEATE = \bigOp[\big]{n^{-0.5}\davg^{0.5}},
		\qquad\text{and}\qquad
		\eHA - \eEATE = \bigOp[\big]{n^{-0.5}\davg^{0.5}}.
		\end{equation}
	\end{proposition}

The Bernoulli design tends to be inefficient in small samples because the size of the treatment groups vary over assignments.
Experimenters often use designs that reduce the variability in the group sizes.
A common such design randomly selects an assignment with equal probability from all assignments with a certain proportion of treated units:
\begin{equation}
	\Pr(\Z = \z) = \left\{ \begin{array}{ll}
	\binom{n}{m}^{-1} & \text{if } \sum_{i=1}^{n} z_i = m,
	\\[0.2em]
	0 & \text{otherwise},
	\end{array}\right.
\end{equation}
where $m=\lfloor p n \rfloor$ for some fixed $p$ strictly between zero and one.
The parameter $p$ controls the desired proportion of treated units.
We call the design \emph{complete randomization}.

	Complete randomization introduces dependencies between assignments. These are not of concern under no interference. The outcomes are only affected by a single treatment, and the dependence between any two treatments is asymptotically negligible. This need not be the case when units interfere. There are two issues to consider.

The first issue is that the interference could interact with the experimental design so that two units' outcomes are strongly dependent asymptotically even when they are not affected by a common treatment (i.e., when $\di_{ij}=0$).
Consider, as an example, when one unit is affected by the first half of the sample and another unit is affected by the second half.
Complete randomization introduces a strong dependence between the two halves.
The number of treated units in the first half is perfectly correlated with the number of treated in the second half.
The outcomes of the two units may therefore be (perfectly) correlated even when no treatment affects them both.
We cannot rule out that such dependencies exist, but we can show that they are sufficiently rare under a slightly stronger version of Assumption \ref{ass:restricted-interference}.

The second issue is that the dependencies introduced by the design distort our view of the potential outcomes.
Whenever a unit is assigned to a certain treatment condition, units that interfere with the unit tend to be assigned to the other condition.
One of the potential outcomes in each assignment-conditional unit-level effect is therefore observed more frequently than the other.
The estimators implicitly weight the two potential outcomes proportionally to their frequency, but the \EATE\ estimand weights them equally.
The discrepancy introduces bias.
Seen from another perspective, the estimators do not separate the effect of a unit's own treatment from spillover effects of other units' treatments.

	As an illustration, consider when the potential outcomes are equal to the number of treated units: $y_i(\z) = \sum_{j=1}^{n}z_j$. \EATE\ equals one in this case, but the estimators are constant at zero since the number of treated units (and, thus, all revealed potential outcomes) are fixed at $m$. The design exactly masks the effect of a unit's own treatment with a spillover effect with the same magnitude but of the opposite sign.

	In general under complete randomization, if the number of units interfering with a unit is of the same order as the sample size, our view of the unit's potential outcomes will be distorted also asymptotically. Similar to the first issue, we cannot rule out that such distortions exist, but restricted interference implies that they are sufficiently rare. Taken together, this establishes consistency under complete randomization.

	\begin{proposition} \label{prop:complete-rates}
		With a complete randomization design under restricted interference (Assumption \ref{ass:restricted-interference}) and $\icoutavg=\littleO[\big]{n^{0.5}}$, the \HT\ and \HA\ estimators are consistent for \EATE\ and converge at the following rates:
		\begin{equation}
		\eHT - \eEATE = \bigOp[\big]{n^{-0.5}\davg^{0.5} + n^{-0.5}\icoutavg},
		\qquad\text{and}\qquad
		\eHA - \eEATE = \bigOp[\big]{n^{-0.5}\davg^{0.5} + n^{-0.5}\icoutavg}.
		\end{equation}
	\end{proposition}

The proposition requires $\icoutavg^2=\littleO{n}$ in addition to Assumption \ref{ass:restricted-interference}.
Both $\icoutavg^2$ and $\davg$ are bounded from below by $\icoutavg$ and from above by $\icoutmsq$, so they tend to be close.
It is when $\di_{ij}$ aligns with $I_{ij}$ to a large extent that $\icoutavg^2$ dominates $\davg$.
For example, if all interference dependent units are interfering with each other directly, so that $\di_{ij}=I_{ij}$, then $\icoutavg=\davg$.

	The \HT\ and \HA\ estimators are known to be root-$n$ consistent for \ATE\ under no interference. Reassuringly, the no-interference assumption is equivalent to $\davg=\icoutavg=1$, and Propositions \ref{prop:bernoulli-rates} and \ref{prop:complete-rates} reproduce the existing result. The propositions, however, make clear that absence of interference is not necessary for such rates, and we may still allow for non-trivial amounts of interference. In particular, root-$n$ rates follow whenever the interference dependence does not grow indefinitely with the sample size. That is, when $\davg$ is bounded.

	\begin{corollary} \label{coro:bernoulli-complete-root-n}
		With a Bernoulli or complete randomization design under bounded interference, $\davg=\bigO{1}$, the \HT\ and \HA\ estimators are root-$n$ consistent for \EATE.
	\end{corollary}

\subsubsection{Paired randomization}

Complete randomization restricts treatment assignment to ensure treatment groups with a fixed size.
The paired randomization design imposes even greater restrictions.
The sample is divided into pairs, and the units in each pair are assigned to different treatments.
It is implicit that the sample size is even so that all units are paired.
Paired randomization could be forced on the experimenter by external constraints or used to improve precision \citep[see, e.g.,][and the references therein]{Fogarty2018Mitigating}.

	Let $\rho : \bfU \to \bfU$ describe a pairing so that $\rho(i)=j$ indicates that units $i$ and $j$ are paired. The pairing is symmetric, so the self-composition of $\rho$ is the identity function. The \emph{paired randomization design} then satisfies
	\begin{equation}
	\Pr(\Z = \z) = \left\{ \begin{array}{ll}
	2^{-n/2} & \text{if } z_i \neq z_{\rho(i)} \text{ for all } i\in\bfU,
	\\[0.2em]
	0 & \text{otherwise}.
	\end{array}\right.
	\end{equation}

	The design accentuates the two issues we faced under complete randomization. Under paired randomization, $Z_i$ and $Z_j$ are perfectly correlated also asymptotically whenever $\rho(i)=j$. We must consider to what degree the dependencies between assignments introduced by the design align with the structure of the interference. The following two definitions quantify the alignment.

	\begin{definition}[Pair-induced interference dependence] \label{def:paired-interference-dependence}
		\begin{equation}
		\eavg = \frac{1}{n}\sum_{i=1}^{n}\sum_{j=1}^n\ei_{ij},
		\qquad
		\text{where}
		\qquad
		\ei_{ij} = \left\{\begin{array}{ll}
		1 & \text{if }(1-\di_{ij})I_{\ell i}I_{\rho(\ell) j} = 1 \text{ for some } \ell \in \bfU,
		\\[0.2em]
		0 & \text{otherwise}.
		\end{array}\right.
		\end{equation}
	\end{definition}

	\begin{definition}[Within-pair interference] \label{def:within-pair-interference}
		$\Rsum = \sum_{i=1}^n I_{\rho(i)i}$.
	\end{definition}

	The dependence within any set of finite number of treatments is asymptotically negligible under complete randomization, and issues only arose when the number of treatments affecting a unit was of the same order as the sample size. Under paired randomization, the dependence between the outcomes of two units not affected by a common treatment can be asymptotically non-negligible even when each unit is affected by an asymptotically negligible fraction of the sample. In particular, the outcomes of units $i$ and $j$ such that $\di_{ij}=0$ could be (perfectly) correlated if two other units $k$ and $\ell$ exist such that $k$ interferes with $i$ and $\ell$ interferes with $j$, and $k$ and $\ell$ are paired. The purpose of Definition \ref{def:paired-interference-dependence} is to capture such dependencies. The definition is similar in structure to Definition \ref{def:interference-dependence}. Indeed, the upper bound from Lemma \ref{lem:interference-measure-inequalities} applies so that $\eavg \leq \icoutmsq$.

	The second issue we faced under complete randomization is affected in a similar fashion. No matter the number of units that are interfering with unit $i$, if one of those units is the unit paired with $i$, we cannot separate the effects of $Z_i$ and $Z_{\rho(i)}$. The design imposes $Z_i=1 - Z_{\rho(i)}$, so any effect of $Z_i$ on $i$'s outcome could just as well be attributed to $Z_{\rho(i)}$. Such dependencies will introduce bias, just as they did under complete randomization. However, unlike the previous design, restricted interference does not imply that the bias will vanish as the sample grows. We must separately ensure that this type of alignment between the design and the interference is sufficiently rare. Definition \ref{def:within-pair-interference} captures how common interference is between paired units. The two definitions allow us to restrict the degree to which the interference aligns with the pairing in the design.

	\begin{assumption}[Restricted pair-induced interference] \label{ass:restricted-pair-interference}
		$\eavg = \littleO{n}$.
	\end{assumption}

	\begin{assumption}[Pair separation] \label{ass:pair-separation}
		$\Rsum = \littleO{n}$.
	\end{assumption}

	Experimenters may find that Assumption \ref{ass:restricted-pair-interference} is quite tenable under restricted interference. As both $\eavg$ and $\davg$ are bounded by $\icoutmsq$, restricted pair-induced interference tends to hold in cases where restricted interference can be assumed. It is, however, possible that the latter assumption holds even when the former does not if paired units are interfering with sufficiently disjoint sets of units.

Whether pair separation holds largely depends on how the pairs were formed.
It is, for example, common that the pairs reflect some social structure (e.g., paired units may live in the same household).
The interference tends to align with the pairing in such cases, and Assumption \ref{ass:pair-separation} is unlikely to hold.
Pair separation is more reasonable when pairs are formed based on generic background characteristics.
This is often the case when the experimenter uses the design to increase precision.
The assumption could, however, still be violated if the background characteristics include detailed geographic data or other information likely to be associated with the interference.

	\begin{proposition} \label{prop:paired-rates}
		With a paired randomization design under restricted interference, restricted pair-induced interference and pair separation (Assumptions \ref{ass:restricted-interference}, \ref{ass:restricted-pair-interference} and \ref{ass:pair-separation}), the \HT\ and \HA\ estimators are consistent for \EATE\ and converge at the following rates:
		\begin{align}
		\eHT - \eEATE &= \bigOp[\big]{n^{-0.5}\davg^{0.5} + n^{-0.5}\eavg^{0.5} + n^{-1}\Rsum},
		\\
		\eHA - \eEATE &= \bigOp[\big]{n^{-0.5}\davg^{0.5} + n^{-0.5}\eavg^{0.5} + n^{-1}\Rsum}.
		\end{align}
	\end{proposition}

\subsection{Arbitrary experimental designs}\label{sec:arbitrary-designs}

We conclude this section by considering sequences of experiments with unspecified designs.
Arbitrary experimental designs may align with the interference just like the paired design.
We start the investigation by introducing a set of definitions that allow us to characterize such alignment in a general setting.

It will prove useful to collect all treatments affecting a particular unit $i$ into a vector:
\begin{equation}
	\Zint_i = (I_{1i}Z_1, I_{2i}Z_2, \dotsc, I_{ni}Z_n).
\end{equation}
The vector is defined so that its $j$th element is $Z_j$ if unit $j$ is interfering with $i$, and zero otherwise.
Similar to above, let $\Zint_{-i}$ be the $(n-1)$-dimensional vector constructed by deleting the $i$th element from $\Zint_i$.
The definition has the following convenient property:
\begin{equation}
	Y_i = y_i(\Z) = y_i\paren[\big]{\Zint_i}, \quad\qquad\text{and}\quad\qquad y_i(z; \Z_{-i}) = y_i\paren[\big]{z; \Zint_{-i}}.
\end{equation}
This allows us to capture the alignment between the design and the interference using $\Zint_i$.
For example, because $Y_i = y_i\paren[\big]{\Zint_i}$, the outcomes of two units $i$ and $j$ are independent whenever $\Zint_i$ and $\Zint_j$ are independent.
The insight allows us to characterize the outcome dependence introduced by the experimental design by the dependence between $\Zint_i$ and $\Zint_j$.
Similarly, the dependence between $Z_i$ and $\Zint_{-i}$ governs how distorted our view of the potential outcomes is.

We use the alpha-mixing coefficient introduced by \citet{Rosenblatt1956} to measure the dependence between the assignment vectors.
Specifically, for two random variables $X$ and $Y$ defined on the same probability space, let
\begin{equation}
	\alpha\paren[\big]{X, Y} = \sup_{\substack{x\in\sigma\paren{X} \\ y \in\sigma\paren{Y}}} \abs[\big]{\,\prob{x \cap y} - \prob{x}\prob{y}\,},
\end{equation}
where $\sigma(X)$ and $\sigma(Y)$ denote the sub-sigma-algebras generated by the random variables.
The coefficient $\alpha\paren{X, Y}$ is zero if and only if $X$ and $Y$ are independent, and increasing values indicate increasing dependence.
The maximum is $\alpha\paren{X, Y} = 1/4$.
Unlike the Pearson correlation coefficient, the alpha-mixing coefficient is not restricted to linear associations between two scalar random variables, and it can capture any type of dependence between any two sets of random variables.
The coefficient allows us to define measures of the average amount of dependence between $\Zint_i$ and $\Zint_j$ and between $Z_i$ and $\Zint_{-i}$.

	\begin{definition}[External and internal average mixing coefficients] \label{def:mixing-coefficients}
		For the maximum values $q$ and $s$ such that Assumptions \ref{ass:bounded-moments} and \ref{ass:bounded-expected-pos} hold, let
		\begin{equation}
		\mxext = \frac{1}{n} \sum_{i=1}^{n} \sum_{j=1}^{n} (1-\di_{ij})\bracket[\big]{\alpha\paren[\big]{\Zint_i, \Zint_j}}^{\frac{q-2}{q}},
		\qquad
		\text{and}\qquad\mxint = \sum_{i=1}^{n} \bracket[\big]{\alpha\paren[\big]{Z_i, \Zint_{-i}}}^{\frac{s-1}{s}},
		\end{equation}
		where $0^0$ is defined as zero to accommodate the cases $q=2$ and $s=1$.
	\end{definition}

Each term in the external mixing coefficient captures the dependence between the treatments affecting unit $i$ and the treatments affecting unit $j$.
If the dependence between $\Zint_i$ and $\Zint_j$ tends to be weak or rare, $\mxext$ will be small compared to $n$.
Similarly, if dependence between $Z_i$ and $\Zint_{-i}$ tends to be weak or rare, $\mxint$ will be small relative to $n$.
In this sense, the external and internal mixing coefficients are generalizations of Definitions \ref{def:paired-interference-dependence} and \ref{def:within-pair-interference}.
Indeed, one can show that $\mxext \propto\eavg$ and $\mxint \propto \Rsum$ under paired randomization where the proportionality constants are given by $q$ and $s$.
The generalized definitions allow for generalized assumptions.

	\begin{assumption}[Design mixing] \label{ass:design-mixing}
		$\mxext = \littleO{n}$.
	\end{assumption}

	\begin{assumption}[Design separation] \label{ass:design-separation}
		$\mxint = \littleO{n}$.
	\end{assumption}

Design mixing and separation stipulate that dependence between treatments are sufficiently rare or sufficiently weak (or some combination thereof).
This encapsulates and extends the conditions in the previous sections.
In particular, complete randomization under bounded interference constitutes a setting where dependence is weak: $\alpha\paren[\big]{\Zint_i, \Zint_j}$ approaches zero for all pairs of units with $\di_{ij}=0$.
Paired randomization under Assumption \ref{ass:restricted-pair-interference} constitutes a setting where dependence is rare: $\alpha\paren[\big]{\Zint_i, \Zint_j}$ may be $1/4$ for some pairs of units with $\di_{ij}=0$, but these are an asymptotically diminishing fraction of the total number of pairs.
Complete randomization under the conditions of Proposition \ref{prop:complete-rates} combines the two settings: $\alpha\paren[\big]{\Zint_i, \Zint_j}$ might be non-negligible asymptotically for some pairs with $\di_{ij}=0$, but such pairs are rare.
For all other pairs with $\di_{ij}=0$, the pair-level mixing coefficient approaches zero quickly.
A similar comparison can be made for the design separation assumption.

	\begin{proposition} \label{prop:arbitrary-rates-eate}
		Under restricted interference, design mixing and design separation (Assumptions \ref{ass:restricted-interference}, \ref{ass:design-mixing} and \ref{ass:design-separation}), the \HT\ and \HA\ estimators are consistent for \EATE\ and converge at the following rates:
		\begin{align}
		\eHT - \eEATE &= \bigOp[\big]{n^{-0.5}\davg^{0.5} + n^{-0.5}\mxext^{0.5} + n^{-1}\mxint},
		\\
		\eHA - \eEATE &= \bigOp[\big]{n^{-0.5}\davg^{0.5} + n^{-0.5}\mxext^{0.5} + n^{-1}\mxint}.
		\end{align}
	\end{proposition}

\begin{remark}
The convergence results for Bernoulli and paired randomization presented in the previous sections can be proven as consequences of Proposition \ref{prop:arbitrary-rates-eate}.
This is not the case for complete randomization.
The current proposition applied to that design would suggest slower rates of convergence than given by Proposition \ref{prop:complete-rates}.
This highlights that Proposition \ref{prop:arbitrary-rates-eate} provides worst-case rates for all designs that satisfy the stated conditions.
Particular designs might be better behaved and, thus, ensure that the estimators converge at faster rates.
For complete randomization, one can prove that restricted interference implies a stronger mixing condition than the conditions defined above.
In particular, Lemmas~\ref{lem:complete-internal-bound} and~\ref{lem:complete-external-bound} in Supplement~A provide bounds on the external and internal mixing coefficients when redefined using the mixing concept introduced by \citet{Blum1963}.
Proposition \ref{prop:complete-rates} follows from this stronger mixing property.
\end{remark}

	\begin{remark}
	If no units interfere, $\Zint_{-i}$ is constant at zero, and Assumption \ref{ass:design-separation} is trivially satisfied. No interference does, however, not imply that Assumption \ref{ass:design-mixing} holds. Consider a design that restricts all treatments to be equal: $Z_1 = Z_2 = \dotsb = Z_n$. The external mixing coefficient would not be zero in this case; in fact, $\mxext \to n/4$. This shows that one must limit the dependencies between treatment assignments even when no units interfere. Proposition \ref{prop:arbitrary-rates-eate} can, in this sense, be seen as an extension of the restrictions imposed in Theorem 1 in \citet{Robinson1982}.
	\end{remark}

\subsection{When design separation fails}

Experimental designs tend to induce dependence between treatments of units that interfere, and experimenters might find it hard to satisfy design separation.
We saw one example of such a design with paired randomization.
It might for this reason be advisable to choose uniform designs such as the Bernoulli or complete randomization if one wants to investigate treatment effects under unknown interference.
These designs cannot align with the interference structure, and one need only consider whether the simpler interference conditions hold.
Another approach is to design the experiment so to ensure design separation.
For example, one should avoid pairing units that are suspected to interfere in the paired randomization design.

	It will, however, not always be possible to ensure that design separation holds. We may ask what the consequences of such departures are. Without Assumption \ref{ass:design-separation}, the effect of units' own treatments cannot be separated from potential spillover effects, and the estimators need not be consistent for \EATE. They may, however, converge to some other quantity, and indeed, they do. The average distributional shift effect from Definition \ref{def:adse} is defined using the conditional distributions of the outcomes. Thus, the estimand does not attempt to completely separate the effect of a unit's own treatment from spillover effects, and design separation is not needed.

	\begin{proposition} \label{prop:arbitrary-rates-adse}
		Under restricted interference and design mixing (Assumptions \ref{ass:restricted-interference} and \ref{ass:design-mixing}), the \HT\ and \HA\ estimators are consistent for \ADSE\ and converge at the following rates:
		\begin{align}
		\eHT - \eADSE = \bigOp[\big]{n^{-0.5}\davg^{0.5} + n^{-0.5}\mxext^{0.5}},
		\\
		\eHA - \eADSE = \bigOp[\big]{n^{-0.5}\davg^{0.5} + n^{-0.5}\mxext^{0.5}}.
		\end{align}
	\end{proposition}

\section{Confidence statements}

The previous section demonstrated that the estimators concentrate around an interpretable version of the average treatment effect in large samples.
Experimenters generally present such estimates together with various statements about the precision of the estimation method.
These statements should be interpreted with caution.
This is because the precision of the estimator may be worse when units interfere, and conventional methods used to estimate the precision may not be accurate.

The fact that the precision may be worse is clear from the rates of convergence.
Given that the design is prudently chosen, the estimators converge at a root-$n$ rate under no interference, but the propositions in the previous sections show that the estimators may converge at a slower rate when units interfere.
The precision will, however, generally improve as the size of the sample grows.
For example, as shown in the proof of Proposition~\ref{prop:bernoulli-rates}, the variance under a Bernoulli design can be bounded as $\Var{\eHT} \leq k^4 n^{-1} \davg$ where $k$ is the constant in Assumption~\ref{ass:regularity-conditions}.
Under Assumption~\ref{ass:restricted-interference}, the variance converges to zero, implying convergence in mean square.
This is in contrast to similar estimators for other estimands under arbitrary interference.
For example, \citet{Basse2018Limitations} show that the variance of the Horvitz-Thompson estimator does not generally decrease as the sample grows when used to estimate the all-or-nothing effect discussed in Section~\ref{sec:related-definitions}.

\subsection{The conventional variance estimator}\label{sec:conventional-var-est}

Realizing that the precision may be worse under interference, one may ask whether common methods used to gauge the precision accurately reflect this.
In this section, we elaborate on this question by investigating the validity of a commonly-used variance estimator.
To avoid some technical difficulties of little interest for the current discussion, we limit our focus to the conventional Horvitz-Thompson estimator of the variance of the Horvitz-Thompson point estimator under a Bernoulli design:
\begin{equation}
	\eVarBer = \frac{1}{n^2} \sum_{i=1}^n \frac{Z_i Y_i^2}{p_i^2} + \frac{1}{n^2} \sum_{i=1}^n \frac{(1 - Z_i) Y_i^2}{(1 - p_i)^2}.
\end{equation}

The estimator is conservative under no interference, meaning that its expectation is greater than the true variance.
On a normalized scale, the bias does not diminish asymptotically, so inferences based on the estimator will be conservative also in large samples.
In particular, with a Bernoulli design under no interference
\begin{equation}
	n \bracket[\Big]{\eVarBer - \Var{\eHT}} \parrow \eMSTE \geq 0,
\end{equation}
where $\eMSTE$ is the mean square treatment effect:
\begin{equation}
	\eMSTE = \frac{1}{n} \sum_{i=1}^{n} \paren[\big]{\E{\tau_i(\Z_{-i})}}^2.
\end{equation}
We define $\eMSTE$ using $\tau_i(\Z_{-i})$ because this will allow us to use it in the following discussion.
In the current setting, we could just as well have used $\tau_i$ because $\tau_i(\z_{-i})$ does not depend on $\z_{-i}$ when units do not interfere.

To characterize the behavior of the estimator under interference, it is helpful to introduce additional notation.
Let $\sote_{ij}(z)$ be the expected treatment effect on unit $i$'s outcome of changing unit $j$'s treatment given that $i$ is assigned to treatment $z$.
In other words, it is the spillover effect from $j$ to $i$ holding $i$'s treatment fixed at $z$.
Formally, we can express this as
\begin{equation}
	\sote_{ij}(z) = \E{y_{ij}(z, 1; \Z_{-ij}) - y_{ij}(z, 0; \Z_{-ij})},
\end{equation}
where, similar to above,
\begin{equation}
	\Z_{-ij} = \paren{Z_1, \dotsc, Z_{i-1}, Z_{i+1}, \dotsc, Z_{j-1}, Z_{j+1}, \dotsc, Z_n}
\end{equation}
is the treatment vector with the $i$th and $j$th elements deleted, and $y_{ij}(a, b; \z_{-ij})$ is unit $i$'s potential outcome when units $i$ and $j$ are assigned to treatment $a$ and $b$, respectively, and the remaining units' assignments are $\z_{-ij}$.

To describe the overall spillover effect between two units, consider
\begin{equation}
	\breve{\sote}_{ij} = \E[\big]{\sote_{ij}(1 - Z_i)} = (1 - p_i) \sote_{ij}(1) + p_i \sote_{ij}(0).
\end{equation}
This is the expected spillover effect using the opposite probabilities for $i$'s treatment.
That is, if $i$ has a low probability of be assigned $Z_i = 1$, so that $p_i$ is close to zero, then $\breve{\sote}_{ij}$ gives more weight to $\sote_{ij}(1)$, which is the spillover effect when $Z_i = 1$.
Let
\begin{equation}
	\breve{Y}_i = (1 - p_i) \E{Y_i \given Z_i = 1} + p_i \E{Y_i \given Z_i = 0},
\end{equation}
the same type of average for the outcome of unit $i$.
These type of quantities occasionally appear in variances of design-based estimators.
The ``tyranny of the minority'' estimator introduced by \citet{Lin2013Agnostic} is one such example.

\begin{proposition}\label{prop:var-est-bernoulli-limit}
	Under a Bernoulli design and Assumption~\ref{ass:regularity-conditions} with $q \geq 4$,
	\begin{equation}
		n\davg^{-1} \bracket[\Big]{\eVarBer - \Var{\eHT}}
		\parrow
		\frac{\eMSTE}{\davg} - B_1 - B_2,
	\end{equation}
	where
	\begin{align}
		B_1 &= \frac{1}{n\davg} \sum_{i=1}^{n} \sum_{j\neq i} \paren[\Big]{\breve{\sote}_{ij}\breve{\sote}_{ji} + 2\breve{Y}_j \bracket[\big]{\sote_{ij}(1) - \sote_{ij}(0)}},
		\\
		B_2 &= \frac{1}{n\davg} \sum_{i=1}^{n} \sum_{j\neq i} \sum_{a = 0}^1 \sum_{b = 0}^1 \paren{-1}^{a + b} \Cov[\big]{Y_i , Y_j \given Z_i = a, Z_j = b}.
	\end{align}
\end{proposition}

The proposition extends the previous limit result to settings with interference.
Indeed, as shown by Corollary~\ref{lem:var-est-no-interference-limit} in Supplement~A, the limit under no interference is a special case of Proposition~\ref{prop:var-est-bernoulli-limit}.
The relevant scaling under interference is $n\davg^{-1}$ rather than $n$, accounting for the fact that the variance may diminish at a slower rate.
In particular, the scaling ensures that $n\davg^{-1}\Var{\eHT}$ is on a constant scale.
The constant $q$ in Assumption~\ref{ass:regularity-conditions} is now required to be at least four, as is typically the case to ensure the convergence of variance estimators.

Compared to the setting without interference, the limit contains two additional terms.
The first additional term, $B_1$, captures the consequence of direct interference between the units.
That is, it captures whether unit $i$ interferes with $j$ directly, so $\breve{\sote}_{ji} \neq 0$.
If there is no interference, there are no spillover effects, so $\sote_{ij}(1) = \sote_{ij}(0) = 0$, and $B_1$ is zero.
The second term captures the consequence of indirect interference.
That is, when a third unit interferes with both units $i$ and $j$.
There are no such units when there is no interference, and $B_2$ is zero.

While all three terms can dominate the others asymptotically, the third will generally be the relevant one.
In particular, if the average interference dependence grows, so $\davg \to \infty$, then $\eMSTE$ is negligible on a normalize scale.
The amount of direct interference dependence is given by $\icoutavg$, so $B_1 = \bigO{\davg^{-1} \icoutavg}$.
The amount of indirect dependence is given by $\davg$, so $B_2 = \bigO{1}$.
Hence, $B_2$ is the dominating term whenever $\davg$ dominates $\icoutavg$, which, as we noted above, is the case whenever the interference is not too tightly clustered.

The key insight here is that when there is interference, these two additional terms are generally non-zero, and they may be both positive and negative.
As a consequence, the variance estimator may be asymptotically anti-conservative, painting an overly optimistic picture about the precision of the point estimator.
Using the conventional variance estimator under interference will, in other words, be misleading.

\subsection{Alternative estimators}

Having established that the variance estimator may be overly optimistic, the next question is if we can account for the potential anti-conservativeness.
The route we will explore here is to inflate the conventional estimator with various measures of the amount of interference.
Besides providing a simple way to construct a reasonable variance estimator when these measures are known, or presumed to be known, this route facilitates constructive discussions about the consequences of interference even when the measures are not known.
In particular, a sensitivity analysis becomes straightforward as the conventional variance estimates are simply multiplied by the sensitivity parameter.

We will not derive the limits of these modified variance estimators.
Indeed, they do not always have limits.
We will instead focus on the main property we seek, namely conservativeness in large samples.
This is captured by a one-sided consistency property, as described in the following definition.

\begin{definition}\label{def:asymptotically-conservative}
	A variance estimator $\hat{V}$ is said to be \emph{asymptotically conservative} with respect to the variance of $\eHT$ if
	\begin{equation}
	\lim_{\varepsilon \to 0^+} \lim_{n \to \infty} \prob[\Big]{n\davg^{-1} \bracket[\big]{\hat{V} - \Var{\eHT}} \leq - \varepsilon} = 0.
	\end{equation}
\end{definition}

The interference measure that first might come to mind is the one from above, namely the average interference dependence, $\davg$.
Using this quantity for the inflation, we would get the following estimator:
\begin{equation}
	\eVarAvg = \davg \eVarBer.
\end{equation}
This will, however, generally not be enough to ensure conservativeness.
The concern is that the interference structure could couple with the potential outcomes in such a way that the interference introduces dependence exactly between units with large outcomes.
Using $\davg$ for the inflation requires that no such coupling takes place, or that it is asymptotically negligible.
The following proposition formalizes the idea.

\begin{proposition}\label{prop:var-est-avg-conservative}
	The variance estimator $\eVarAvg$ is asymptotically conservative under a Bernoulli design if Assumptions~\ref{ass:regularity-conditions} and~\ref{ass:restricted-interference} hold with $q \geq 4$ and $\textsc{sd}_{\sigma^2} = \littleO{\drms^{-1} \davg}$ where
	\begin{align}
	&\textsc{sd}_{\sigma^2} = \sqrt{\frac{1}{n}\sum_{i=1}^{n}\bracket[\bigg]{\sigma_i^2 - \frac{1}{n} \sum_{j=1}^{n} \sigma_j^2}^2},
	\qquad
	\sigma_i^2 = \Var[\bigg]{\frac{Z_iY_i}{p_i} - \frac{(1- Z_i)Y_i}{1 - p_i}},
	\\
	&\qquad\qquad\qquad\qquad\qquad\text{and}\qquad
	\drms = \sqrt{\frac{1}{n} \sum_{i=1}^n \di_i^2}.
	\end{align}
\end{proposition}

The condition $\textsc{sd}_{\sigma^2} = \littleO{\drms^{-1} \davg}$ is the design-based equivalent of a homoscedasticity assumption.
In particular, $\sigma_i^2$ is the unit-level contribution to the variance of the point estimator, so $\textsc{sd}_{\sigma^2}$ is the standard deviation of the unit-level variances.
The quantity $\drms$ is the root mean square of $\di_i$, and $\davg$ is the average, so $\drms^{-1} \davg \leq 1$.
The condition thus states that the $\textsc{sd}_{\sigma^2}$ diminishes quickly, requiring that the unit-level variances are approximately the same.
When this is the case, no coupling of consequence can happen, so the inflated estimator is conservative.

The homoscedasticity condition in Proposition~\ref{prop:var-est-avg-conservative} is strong, and it will generally not hold.
The estimator must in these cases be further inflated.
In particular, to capture possible coupling, the inflation factor must take into account the skewness of the unit-level interference dependencies.
A straightforward way is to substitute the maximum for the mean, producing
\begin{equation}
	\eVarMax = \dmax \eVarBer,
\end{equation}
where $\dmax$ is the maximum of $\di_i$ over $i \in \bfU$.

\begin{proposition}\label{prop:var-est-max-conservative}
	The variance estimator $\eVarMax$ is asymptotically conservative under a Bernoulli design if either
	\begin{enumerate}
		\item Assumption~\ref{ass:regularity-conditions} holds with $q \geq 4$ and $\dmax = \littleO{n^{0.5}\davg^{0.5}}$, or
		\item Assumptions~\ref{ass:regularity-conditions} and~\ref{ass:restricted-interference} hold with $q \geq 4$ and $\eMSTE = \bigOmega{1}$.
	\end{enumerate}
\end{proposition}

The proposition demonstrates that the variance estimator inflated with $\dmax$ is conservative without homoscedasticity assumptions.
The first case states that $\dmax$ is dominated by the geometric mean of $n$ and $\davg$, saying that the maximum $\di_i$ does not grow too quickly compared to the average.
The condition ensures that the inflated estimator concentrates.
As noted above, however, an estimator can be asymptotically conservative even when it is not convergent.
The concern in that case is that part of the sampling distribution may approach zero at a faster rate than the inflation factor, leading to anti-conservativeness.
A sufficient condition to avoid such behavior is the second case, namely that $\eMSTE$ is asymptotically bounded from below.
In other words, that the unit-level treatment effects do not concentrate around zero, implying either that the average treatment effect is not zero or that there is some effect heterogeneity.

Using $\dmax$ for the inflation will generally be too conservative.
An intermediate option between the average and the maximum would be ideal.
Let $\mathbf{D}$ be a matrix whose typical argument is $\di_{ij}$, and let $\eigmax$ be the largest eigenvalue of this matrix.
Put differently, if $\di_{ij}$ denotes edges in a graph in which the units are vertices, then $\eigmax$ is the spectral radius of that graph.
The quantity acts as a measure of the amount of interference in the sense that $\eigmax = 1$ when there is no interference, and $\eigmax = n$ when all units are interference dependent.
Furthermore, $\eigmax$ weakly increases with the interference.
Using this quantity as the inflation factor, the variance estimator becomes
\begin{equation}
	\eVarSR = \eigmax \eVarBer.
\end{equation}
The spectral radius is such that $\davg \leq \eigmax \leq \dmax$, indicating that the estimator inflated by $\eigmax$ is more conservative than when inflated by $\davg$ but less conservative than when inflated by $\dmax$.
The inflation is sufficient for conservativeness under similar conditions as before.

\begin{proposition}\label{prop:var-est-eig-conservative}
	The variance estimator $\eVarSR$ is asymptotically conservative under a Bernoulli design if either
	\begin{enumerate}
		\item Assumption~\ref{ass:regularity-conditions} holds with $q \geq 4$ and $\eigmax = \littleO{n^{0.5}\davg^{0.5}}$, or
		\item Assumptions~\ref{ass:regularity-conditions} and~\ref{ass:restricted-interference} hold with $q \geq 4$ and $\eMSTE = \bigOmega{1}$.
	\end{enumerate}
\end{proposition}

The adjustments needed for conservativeness highlights that the interference may introduce considerable imprecision.
These inflation factors are, however, constructed to accommodate the worse case.
As noted above, interference may in some cases improve precision, and no inflation is required in such cases.
More information about the interference structure is, however, needed to take advantage of such precision improvements when estimating the variance.
For example, \citet{Aronow2018Confidence} construct a variance estimator when $\di_i$ is known for all units.
Improvements may also be possible if larger departures from the conventional variance estimator are acceptable.
For example, it is sufficient to use $\drms$ as the inflation factor if higher moments of the unit-level variances is substituted for the conventional estimator.
This will generally be less conservative than the estimator inflated by the spectral radius because $\davg \leq \drms \leq \eigmax$.

\subsection{Tail bounds}

Experimenters often use variance estimates to construct bounds on the tails of the sampling distribution.
A common approach is to combine a conservative variance estimator with a normal approximation of the sampling distribution of a point estimator, motivated by a central limit theorem.
Such approximations may be reasonable when the interference is very sparse.
For example, Theorem~2.7 in \citet{Chen2004Normal} applies to the \HT\ point estimator under the Bernoulli design if $\dmax = \bigO{1}$, showing that the sampling distribution is approximately normal in large samples.
\citet{Chin2019Central} improves this to $\dmax = \littleO{n^{0.25}}$.
Both conditions are, however, considerably stronger than Assumption~\ref{ass:restricted-interference}.
The following proposition shows that normal approximations will generally not be appropriate under the conditions considered in this paper.

\begin{proposition}\label{prop:Chebyshev-sharp}
	Chebyshev's inequality is asymptotically sharp with respect to the sampling distribution of the \HT\ estimator for every sequence of Bernoulli designs under Assumptions~\ref{ass:regularity-conditions} and~\ref{ass:restricted-interference}.
	The inequality remains sharp when Assumption~\ref{ass:restricted-interference} is strengthened to $\davg = \bigO{1}$ and $\dmax = \bigO{n^{0.5}}$.
\end{proposition}

Tail bounds based on Chebyshev's inequality are more conservative than those based on a normal distribution, so the proposition implies that a normal approximation is appropriate only under stronger conditions than those used for consistency in Section~\ref{sec:large-sample-prop}.
Indeed, $\davg = \bigO{1}$ was the strongest interference condition under consideration in that section, providing root-$n$ consistency.
Proposition~\ref{prop:Chebyshev-sharp} can be extended to other designs, including complete and paired randomization, as discussed in its proof in Supplement~A.

The conclusion is that Chebyshev's inequality is an appropriate way to construct confidence intervals and conduct testing when it is not reasonable to assume that $\dmax = \littleO{n^{0.5}}$.
The inequality will be overly conservative when $\dmax = \littleO{n^{0.25}}$, and it remains an open question whether it is sharp when $\dmax = \littleO{n^{x}}$ for $0.25 < x \leq 0.5$.
It may be possible to prove a central limit theorem, or otherwise derive less conservative tail bounds, even when $\dmax$ is large if one imposes other types of assumptions, for example on the magnitude of the interference.

\section{Other designs and external validity}

By marginalizing over the experimental design, \EATE\ captures an average treatment effect in the experiment that actually was implemented.
A consequence is that the estimand may have taken a different value if another design were used.
Experimenters know that the results from a single experiment may not extend beyond the present sample.
When units interfere, concerns about external validity should include experimental designs as well.

In this section, we elaborate on this concern by asking to what degree the effect for one design generalizes to other designs.
Because the experiment only provides information about the potential outcomes on its support, the prospects of extrapolation are limited.
The hope is that an experiment may be informative of the treatment effects under designs that are close to the one that was implemented.

It will be helpful to introduce notation that allows us to differentiate between the actual design and an alternative design that could have been implemented but was not.
Let $P$ denote the probability measure of the design that was implemented, and let $Q$ be the probability measure of the alternative design.
A subscript indicates which measure various operators refer to.
For example, $\Ep{Y_i}$ is the expected outcome for unit $i$ under the $P$ design, and $\Eq{Y_i}$ is the same under the alternative design.
The question we ask here is how informative
\begin{equation}
	\eEATEp = \Ep{\eATE(\Z)}
\end{equation}
is about
\begin{equation}
	\eEATEq = \Eq{\eATE(\Z)}.
\end{equation}

To answer this question, we will use measures of closeness of designs.
The measure is given meaning by coupling it with some type of structural assumption on the potential outcomes.
The stronger these structural assumptions, the more we can hope to extrapolate.
In line with the rest of the paper, we focus a relatively weak assumption here, effectively limiting ourselves to local extrapolation.

\begin{assumption}[Bounded unit-level effects]\label{ass:bounded-unit-effects}
	There exists a constant $\teconst$ such that $\abs{\tau_i(\z)} \leq \teconst$ for all $i \in \bfU$ and $\z \in \braces{0,1}^{n - 1}$.
\end{assumption}

The assumption is not implied by Assumption~\ref{ass:regularity-conditions}.
The regularity conditions used above ensure that the potential outcomes are well-behaved with respect to the design that was actually implemented.
The conditions do not ensure that they are well-behaved under other designs.
Assumption~\ref{ass:bounded-unit-effects} can from this perspective be seen as an extension of Assumption~\ref{ass:bounded-expected-pos}.

It remains to define a measure of closeness of designs.
A straightforward choice is the total variation distance between the distributions of the designs:
\begin{equation}
	\delta(P, Q) = \sup_{x \in \mathcal{F}} \, \abs{P(x) - Q(x)},
\end{equation}
where $\mathcal{F}$ is the event space of the designs, typically the power set of $\braces{0, 1}^n$.
The distance is the largest difference in probability between the designs for any event in the event space.
It is an unforgiving measure; it defines closeness purely as overlap between the two distributions.
Its advantage is that it requires little structure on the potential outcomes to be informative.
In particular, Assumption~\ref{ass:bounded-unit-effects} implies that the assignment-conditional average treatment effect function, $\eATE(\z)$, is bounded, which gives the following result.

\begin{proposition}\label{prop:total-variation-distance-bound}
	Given Assumption~\ref{ass:bounded-unit-effects},
	\begin{equation}
		\abs{\eEATEp - \eEATEq} \leq 2 \teconst \delta(P, Q).
	\end{equation}
\end{proposition}

The intuition behind the proposition is straightforward.
If the two designs overlap to a large degree, then the marginalization over $\eATE(\z)$ will overlap to a similar large degree.
While straightforward, it is a good illustration that an experiment is relevant beyond the current design under relatively weak assumptions.
However, given the unforgiving nature of the total variation distance, the bound says little more than that the designs must be close to identical to be informative of each other.
A simple example illustrates the concern.

The example compares the estimand under a Bernoulli design with $p_i = 1 / 2$ for all units with the estimand under complete randomization with $p = 1 / 2$.
Consider the event $\sum_{i=1}^{n} Z_i = \lfloor n / 2 \rfloor$.
By construction of the complete randomization design, the probability of this event is one.
Under Bernoulli randomization, the probability is bounded as
\begin{equation}
	\frac{1}{2^n} \binom{n}{ \lfloor n / 2 \rfloor } = \bigO[\big]{n ^{-0.5}}.
\end{equation}
Hence, the total variation distance between the designs approaches one.
When taken at face value, this means that \EATE\ under one of the designs provides no more information about the estimand under the other design than what is already provided by Assumption~\ref{ass:bounded-unit-effects}.

We can sharpen this bound if we know that the interference is limited.
There are several ways to take advantage of a sparse interference structure when extrapolating between designs.
The route we explore here is to consider how sparseness affects the average treatment effect function, as captured in the following lemma.

\begin{lemma}\label{lem:ate-function-Lipschitz}
	Given Assumption~\ref{ass:bounded-unit-effects}, $\eATE(\z)$ is $2 \teconst n^{-1/r} \icoutmom{r / (r - 1)}$-Lipschitz continuous with respect to the $L_r$ distance over $\braces{0, 1}^n$ for any $r \geq 1$.
\end{lemma}

The lemma shows that $\eATE(\z)$ does not change too quickly in $\z$ under sparse interference.
The intuition is that changing a unit's treatment can affect a limited number of other units when the interference is sparse, and the bound on the unit-level treatment effects limits how consequential the change can be on the affected units.
The lemma uses $\icoutmom{r / (r - 1)}$ to measure the amount of interference rather than $\davg$.
This is the moment of the unit-level interference count defined in Section~\ref{sec:quantifying-interference}, where these moments were used to bound $\davg$.
The lemma can be sharpened if more information about the potential outcomes is available.
For example, the factor $2 \teconst$ can be reduced if Assumption~\ref{ass:bounded-unit-effects} is substituted with a Lipschitz continuity assumption directly on unit-level effects.

Despite being intuitively straightforward, Lemma~\ref{lem:ate-function-Lipschitz} is powerful because the $L_r$ distances are fairly forgiving.
However, to take advantage of this forgivingness, we must modify the measure of design closeness so that it incorporates the geometry given by the distance.
The Wasserstein metric accomplishes this, and it couples well with Lipschitz continuity.

Let $\mathcal{J}(P, Q)$ collect all distributions $(\Z', \Z'')$ over $\braces{0, 1}^n \times \braces{0, 1}^n$ such that the marginal distributions of $\Z'$ and $\Z''$ are $P$ and $Q$, respectively.
Each distribution in $\mathcal{J}(P, Q)$ can be seen as a way to transform $P$ to $Q$ by moving mass between the $\braces{0, 1}^n$ points in the marginal distributions.
The Wasserstein metric is the least costly way to make this transformation with respect to the $L_r$ distance:
\begin{equation}
	W_r(P, Q) = \inf_{J \in \mathcal{J}(P, Q)} \Ej[\big]{\norm{\Z' - \Z''}_r},
\end{equation}
where the subscript denotes the underlying $L_r$ distance rather than the order of the Wasserstein metric, which is always taken to be one here.
This metric is more forgiving than the total variation distance because it goes beyond direct overlap and also considers how close the non-overlapping parts of the distributions are.
An application of the Kantorovich--Rubinstein duality theorem \citep{Edwards2011} provides the following result.

\begin{proposition}\label{prop:Wasserstein-bound}
	Given Assumption~\ref{ass:bounded-unit-effects},
	\begin{equation}
		\abs{\eEATEp - \eEATEq} \leq 2 \teconst n^{-1/r} \icoutmom{r / (r - 1)} W_r(P, Q).
	\end{equation}
\end{proposition}

The proposition highlights the key insight of this section.
If the amount of interference does not grow too quickly relative to the difference between the designs as captured by the Wasserstein metric,
\begin{equation}
	\icoutmom{r / (r - 1)} W_r(P, Q) = \littleO[\big]{n^{1/r}},
\end{equation}
then the expected average treatment effects under the two designs will converge.
The optimal choice of $r$ depends on how unevenly the interference is distributed among the units.
For example, if the interference is skewed, then $r = 2$ may be reasonable to avoid being sensitive to the outliers.
Here, $\icoutmom{r / (r - 1)} = \icoutmom{2}$ is the root mean square of the interference count.
If all units interfere with approximately the same number of other units, then $r = 1$ is a better choice, in which case $\icoutmom{r / (r - 1)}$ is taken to be $\icoutmax$.

Our example with the Bernoulli and complete randomization designs provides a good illustration.
When $r = 2$, the corresponding Wasserstein distance between the designs is $W_2(P, Q) = \bigO{n^{0.25}}$.
It follows that the two estimands converge as long as $\icoutmom{2} = \littleO{n^{0.25}}$.
When $r = 1$, the corresponding Wasserstein distance is $W_1(P, Q) = \bigO{n^{0.5}}$, so the estimands converge whenever $\icoutmax = \littleO{n^{0.5}}$.

The proposition may prove useful also when the estimands do not converge.
For example, we have $n^{-1/r} W_r(P, Q) \sim \abs{p - q}^{1/r}$ when $P$ and $Q$ are two complete randomization designs with $p$ and $q$ as their respective assignment probabilities.
The corresponding estimands will thus generally not converge, but Proposition~\ref{prop:Wasserstein-bound} may provide a useful bound if $\icoutmom{r / (r - 1)}$ and $\abs{p - q}$ are reasonably small.
The bound is especially informative when more is known about the potential outcomes so that the factor $2 \teconst$ can be reduced.

\section{Simulation study}

Supplement~B presents a simulation study to illustrate and complement the results presented here.
We focus on three types of data generating processes in this study, differing in the structure of the interference.
In particular, we investigate when the interference is contained within groups of units, when the interference structure is randomly generated and when only one unit is interfering with other units.
For each type of interference structure, we alter the amount of interference ranging from $\davg = 1$ to $\davg = n$.

The findings from the simulation study corroborate the theoretical results.
The estimators are shown to approach \EATE\ at the rates predicted by the propositions above, and they generally do not converge when Assumption~\ref{ass:restricted-interference} does not hold.
There are, however, situations where they converge at faster rates than those given by the propositions above.
This captures the fact that the theoretical results focus on the worst case given the stated conditions.
For example, in one of the settings we investigate in the simulations, the experimental design aligns with the interference structure in such a way that the design almost perfectly counteracts the interference.
The precision of the estimator does not significantly depend on the amount of interference in this case.
The setting is, however, rather artificial, and it was picked to illustrate exactly this point.
A slight modification of the data generating process breaks the behavior.
We direct readers to the supplement for further insights from the simulation study.

	\section{Concluding remarks}

	Experimenters worry about interference. The first line of attack tends to be to design experiments so to minimize the risk that units interfere. One could, for example, physically isolate the units throughout the study. The designs needed to rule out interference may, however, make the experiments so alien to the topics under study that the findings are no longer relevant; the results would not generalize to the real world where units do interfere. When design-based fixes are undesirable or infeasible, one could try to account for any lingering interference in the analysis. This, however, requires detailed knowledge about the structure of the interference. The typical experimenter neither averts all interference by design nor accounts for it in the analysis. They conduct and analyze the experiment as if no units interfere, even when the no-interference assumption, at best, holds true only approximately. The disconnect between assumptions and reality is reconciled by what appears to be a common intuition among experimenters that goes against the conventional view: unmodeled interference is not a fatal flaw so long as it is limited. Our results provide rigorous justification for this intuition.

The \EATE\ estimand generalizes the average treatment effect to experiments with interference.
All interpretations of \ATE\ do, however, not apply.
In particular, \EATE\ cannot be interpreted as the difference between the average outcome when no unit is treated and the average outcome when all units are treated.
The estimand is the expected, marginal effect of changing a single treatment in the current experiment.
From a practical perspective, these marginal effects are relevant to policy makers considering decisions along an intensive margin.
From a theoretical perspective, \EATE\ could act as a sufficient statistic for a structural model, thereby allowing researchers to pin down theoretically important parameters \citep{Chetty2009Sufficient}.

The main contribution of the paper is, however, to describe what can be learned from an experiment under unknown and arbitrary interference.
As shown by \citet{Basse2018Limitations} and others, causal inference under interference tends to require strong assumptions.
The current results nevertheless show that experiments are informative of \EATE\ under weak assumptions.
The insight is valuable even when \EATE\ is not the parameter of primary interest because it shows what can be learned from an experiment without imposing a model or making strong structural assumptions.
A comparison can be made with the \textit{local average treatment effect} (\LATE) estimand for the instrumental variable estimator \citep{Imbens1994Identification}.
The local effect may not be the parameter of primary interest, but it is relevant because it describes what can be learned in experiments with noncompliance without strong assumptions about, for example, constant treatment effects.

We conjecture that the results in this paper extend also to observational studies.
Several issues must, however, be addressed before such results can be shown formally.
These issues are mainly conceptual in nature.
We do not know of a stochastic framework that can accommodate unknown and arbitrary interference in an observational setting.
This is because the design (or the assignment mechanism as it is often called in an observational setting) is unknown here.
A common way to approach this problem is to approximate the assignment mechanism with a design that is easy to analyze, for example a Bernoulli design, and assume that the units' marginal treatment probabilities are given by a function of only the units' own characteristics.
The concern with this approach is that the behavior of the estimators is sensitive to details of the design, as shown in Proposition~\ref{prop:arbitrary-rates-eate}, so the design used for the approximation may not be appropriate.
Furthermore, the treatment probabilities may be functions of other units' characteristics, effectively capturing interference in treatment assignment.
\citet{Forastiere2017} address these concerns by assuming that a unit's treatment probability is a function of both its own characteristics and the characteristics of its neighbors in an interference graph.
This route can, however, not be taken here because the interference structure is not known.

Another potential solution is to extend the stochastic framework to include sampling variability, assuming that the current sample is drawn from some larger, possibly infinite, population.
When interference is investigated in this type of regime, the units are generally assumed to be sampled in such a way to maintain the interference structure.
The units can, for example, be assumed to be sampled in groups.
However, the only way to ensure that the interference structure is maintained under arbitrary interference is to consider the whole sample being sampled jointly.
\citet{TchetgenTchetgen2019Auto} address this concern by assuming that the outcomes are independently and identically generated from some common distribution conditional on all relevant aspects of the interference, which, for example, could include the number of treated units in a neighborhood in some interference graph.
However, this approach cannot be used here because the relevant aspects of the interference are not known.

We have focused on the effect of a unit's own treatment.
The results are, however, not necessarily restricted to ``direct'' treatment effects as typically defined.
In particular, the pairing between units and treatments is arbitrary in the causal model, and an experiment could have several reasonable pairings.
Consider, for example, when the intervention is to give some drug to the units in the sample.
The most natural pairing might be to let a unit's treatment indicator denote whether we give the drug to the unit itself.
We may, however, just as well let it denote whether we give the drug to, say, the unit's spouse (who may be in the sample as well).
\EATE\ would then be the expected spillover effect between spouses.
In this sense, the current investigation applies both to usual treatment effects and rudimentary spillover effects.
We conjecture that the results can be extended to other definitions of treatment, and they would in that case afford robustness to unknown and arbitrary interference for more intricate spillover effects.

\bibliographystyle{modapa}
\bibliography{eate-refs}

\begin{thebibliography}{}

\bibitem[Angrist, 2014]{Angrist2014}
Angrist, J.~D. (2014).
\newblock \href{https://doi.org/10.1016/j.labeco.2014.05.008}{The perils of
  peer effects}.
\newblock {\em Labour Economics}, 30, 98--108.

\bibitem[Aronow, 2012]{Aronow2012}
Aronow, P.~M. (2012).
\newblock \href{https://doi.org/10.1177/0049124112437535}{A general method for
  detecting interference between units in randomized experiments}.
\newblock {\em Sociological Methods \& Research}, 41(1), 3--16.

\bibitem[Aronow et~al., 2018]{Aronow2018Confidence}
Aronow, P.~M., Crawford, F.~W., \& Zubizarreta, J.~R. (2018).
\newblock \href{https://doi.org/10.1214/18-ejs1448}{Confidence intervals for
  linear unbiased estimators under constrained dependence}.
\newblock {\em Electronic Journal of Statistics}, 12(2), 2238--2252.

\bibitem[Aronow \& Samii, 2017]{Aronow2017}
Aronow, P.~M. \& Samii, C. (2017).
\newblock \href{https://doi.org/10.1214/16-aoas1005}{Estimating average causal
  effects under general interference}.
\newblock {\em Annals of Applied Statistics}, 11(4), 1912--1947.

\bibitem[Athey et~al., 2018]{Athey2018Exact}
Athey, S., Eckles, D., \& Imbens, G.~W. (2018).
\newblock \href{https://doi.org/10.1080/01621459.2016.1241178}{Exact p-values
  for network interference}.
\newblock {\em Journal of the American Statistical Association}, 113(521),
  230--240.

\bibitem[Basse \& Feller, 2018]{Basse2018Analyzing}
Basse, G. \& Feller, A. (2018).
\newblock \href{https://doi.org/10.1080/01621459.2017.1323641}{Analyzing
  two-stage experiments in the presence of interference}.
\newblock {\em Journal of the American Statistical Association}, 113(521),
  41--55.

\bibitem[Basse \& Airoldi, 2018a]{Basse2018Limitations}
Basse, G.~W. \& Airoldi, E.~M. (2018a).
\newblock \href{https://doi.org/10.1177/0081175018782569}{Limitations of
  design-based causal inference and {A/B} testing under arbitrary and network
  interference}.
\newblock {\em Sociological Methodology}, 48(1), 136--151.

\bibitem[Basse \& Airoldi, 2018b]{Basse2018Model}
Basse, G.~W. \& Airoldi, E.~M. (2018b).
\newblock \href{https://doi.org/10.1093/biomet/asy036}{Model-assisted design of
  experiments in the presence of network-correlated outcomes}.
\newblock {\em Biometrika}, 105(4), 849--858.

\bibitem[Basse et~al., 2019]{Basse2019Randomization}
Basse, G.~W., Feller, A., \& Toulis, P. (2019).
\newblock \href{https://doi.org/10.1093/biomet/asy072}{Randomization tests of
  causal effects under interference}.
\newblock {\em Biometrika}, 106(2), 487--494.

\bibitem[Blum et~al., 1963]{Blum1963}
Blum, J.~R., Hanson, D.~L., \& Koopmans, L.~H. (1963).
\newblock \href{https://doi.org/10.1007/BF00535293}{On the strong law of large
  numbers for a class of stochastic processes}.
\newblock {\em Zeitschrift f{\"u}r Wahrscheinlichkeitstheorie und verwandte
  Gebiete}, 2(1), 1--11.

\bibitem[Bowers et~al., 2013]{Bowers2013}
Bowers, J., Fredrickson, M.~M., \& Panagopoulos, C. (2013).
\newblock \href{https://doi.org/10.1093/pan/mps038}{Reasoning about
  interference between units: {A} general framework}.
\newblock {\em Political Analysis}, 21(1), 97--124.

\bibitem[Bramoull{\'e} et~al., 2009]{Bramoulle2009}
Bramoull{\'e}, Y., Djebbari, H., \& Fortin, B. (2009).
\newblock \href{https://doi.org/10.1016/j.jeconom.2008.12.021}{Identification
  of peer effects through social networks}.
\newblock {\em Journal of Econometrics}, 150(1), 41--55.

\bibitem[Chen \& Shao, 2004]{Chen2004Normal}
Chen, L. H.~Y. \& Shao, Q.-M. (2004).
\newblock \href{https://doi.org/10.1214/009117904000000450}{Normal
  approximation under local dependence}.
\newblock {\em Annals of Probability}, 32(3), 1985--2028.

\bibitem[Chetty, 2009]{Chetty2009Sufficient}
Chetty, R. (2009).
\newblock
  \href{https://doi.org/10.1146/annurev.economics.050708.142910}{Sufficient
  statistics for welfare analysis: {A} bridge between structural and
  reduced-form methods}.
\newblock {\em Annual Review of Economics}, 1(1), 451--488.

\bibitem[Chin, 2019]{Chin2019Central}
Chin, A. (2019).
\newblock \href{https://arxiv.org/abs/1804.03105v2}{Central limit theorems via
  {Stein's} method for randomized experiments under interference}.
\newblock arXiv:1804.03105v2.

\bibitem[Choi, 2017]{Choi2017}
Choi, D. (2017).
\newblock \href{https://doi.org/10.1080/01621459.2016.1194845}{Estimation of
  monotone treatment effects in network experiments}.
\newblock {\em Journal of the American Statistical Association}, 112(519),
  1147--1155.

\bibitem[Cox, 1958]{Cox1958}
Cox, D.~R. (1958).
\newblock {\em Planning of Experiments}.
\newblock New York: Wiley.

\bibitem[Davydov, 1970]{Davydov1970}
Davydov, Y.~A. (1970).
\newblock \href{https://doi.org/10.1137/1115050}{The invariance principle for
  stationary processes}.
\newblock {\em Theory of Probability \& Its Applications}, 15(3), 487--498.

\bibitem[Eck et~al., 2018]{Eck2018Randomization}
Eck, D.~J., Morozova, O., \& Crawford, F.~W. (2018).
\newblock \href{https://arxiv.org/abs/1808.05593}{Randomization for the direct
  effect of an infectious disease intervention in a clustered study
  population}.
\newblock arXiv:1808.05593.

\bibitem[Eckles et~al., 2016]{Eckles2016}
Eckles, D., Karrer, B., \& Ugander, J. (2016).
\newblock \href{https://doi.org/10.1515/jci-2015-0021}{Design and analysis of
  experiments in networks: {Reducing} bias from interference}.
\newblock {\em Journal of Causal Inference}, 5(1).

\bibitem[Edwards, 2011]{Edwards2011}
Edwards, D.~A. (2011).
\newblock \href{https://doi.org/10.1016/j.exmath.2011.06.005}{On the
  {Kantorovich–Rubinstein} theorem}.
\newblock {\em Expositiones Mathematicae}, 29(4), 387--398.

\bibitem[Egami, 2017]{Egami2017}
Egami, N. (2017).
\newblock \href{https://arxiv.org/abs/1708.08171v2}{Unbiased estimation and
  sensitivity analysis for network-specific spillover effects: {Application} to
  an online network experiment}.
\newblock arXiv:1708.08171v2.

\bibitem[Fisher, 1935]{Fisher1935}
Fisher, R.~A. (1935).
\newblock {\em The Design of Experiments}.
\newblock London: Oliver \& Boyd.

\bibitem[Fogarty, 2018]{Fogarty2018Mitigating}
Fogarty, C.~B. (2018).
\newblock \href{https://doi.org/10.1111/rssb.12290}{On mitigating the
  analytical limitations of finely stratified experiments}.
\newblock {\em Journal of the Royal Statistical Society: Series B (Statistical
  Methodology)}, 80(5), 1035--1056.

\bibitem[Forastiere et~al., 2017]{Forastiere2017}
Forastiere, L., Airoldi, E.~M., \& Mealli, F. (2017).
\newblock \href{https://arxiv.org/abs/1609.06245v3}{Identification and
  estimation of treatment and interference effects in observational studies on
  networks}.
\newblock arXiv:1609.06245v3.

\bibitem[Gilbert, 1959]{Gilbert1959Random}
Gilbert, E.~N. (1959).
\newblock \href{https://doi.org/10.1214/aoms/1177706098}{Random graphs}.
\newblock {\em Annals of Mathematical Statistics}, 30(4), 1141--1144.

\bibitem[Goldsmith-Pinkham \& Imbens, 2013]{Goldsmith2013}
Goldsmith-Pinkham, P. \& Imbens, G.~W. (2013).
\newblock \href{https://doi.org/10.1080/07350015.2013.801251}{Social networks
  and the identification of peer effects}.
\newblock {\em Journal of Business \& Economic Statistics}, 31(3), 253--264.

\bibitem[Graham, 2008]{Graham2008}
Graham, B.~S. (2008).
\newblock \href{https://doi.org/10.1111/j.1468-0262.2008.00850.x}{Identifying
  social interactions through conditional variance restrictions}.
\newblock {\em Econometrica}, 76(3), 643--660.

\bibitem[Green \& Gerber, 2004]{Green2004}
Green, D.~P. \& Gerber, A.~S. (2004).
\newblock
  \href{https://www.brookings.edu/book/get-out-the-vote-how-to-increase-voter-turnout/}{{\em
  Get out the vote: {How} to increase voter turnout}}.
\newblock Washington, D.C.: Brookings Institution Press.

\bibitem[Hahn, 1998]{Hahn1998}
Hahn, J. (1998).
\newblock \href{https://doi.org/10.2307/2998560}{On the role of the propensity
  score in efficient semiparametric estimation of average treatment effects}.
\newblock {\em Econometrica}, 66(2), 315--331.

\bibitem[H{\'a}jek, 1971]{Hajek1971}
H{\'a}jek, J. (1971).
\newblock Comment: {An} essay on the logical foundations of survey sampling,
  part one.
\newblock In V.~P. Godambe \& D.~A. Sprott (Eds.), {\em Foundations of
  Statistical Inference}. Toronto: Holt, Rinehart and Winston.

\bibitem[Halloran \& Hudgens, 2016]{Halloran2016}
Halloran, M.~E. \& Hudgens, M.~G. (2016).
\newblock \href{https://doi.org/10.1007/s40471-016-0086-4}{Dependent
  happenings: {A} recent methodological review}.
\newblock {\em Current Epidemiology Reports}, 3(4), 297--305.

\bibitem[Halloran \& Struchiner, 1995]{Halloran1995}
Halloran, M.~E. \& Struchiner, C.~J. (1995).
\newblock \href{https://doi.org/10.2307/3702315}{Causal inference in infectious
  diseases}.
\newblock {\em Epidemiology}, 6(2), 142--151.

\bibitem[Hern\'an \& Robins, 2006]{Hernan2006}
Hern\'an, M.~A. \& Robins, J.~M. (2006).
\newblock \href{https://doi.org/10.1136/jech.2004.029496}{Estimating causal
  effects from epidemiological data}.
\newblock {\em Journal of Epidemiology \& Community Health}, 60(7), 578--586.

\bibitem[Hirano et~al., 2003]{Hirano2003}
Hirano, K., Imbens, G.~W., \& Ridder, G. (2003).
\newblock \href{https://doi.org/10.1111/1468-0262.00442}{Efficient estimation
  of average treatment effects using the estimated propensity score}.
\newblock {\em Econometrica}, 71(4), 1161--1189.

\bibitem[Holland, 1986]{Holland1986}
Holland, P.~W. (1986).
\newblock \href{https://doi.org/10.1080/01621459.1986.10478354}{Statistics and
  causal inference}.
\newblock {\em Journal of the American Statistical Association}, 81(396),
  945--960.

\bibitem[Horvitz \& Thompson, 1952]{Horvitz1952}
Horvitz, D.~G. \& Thompson, D.~J. (1952).
\newblock \href{https://doi.org/10.1080/01621459.1952.10483446}{A
  generalization of sampling without replacement from a finite universe}.
\newblock {\em Journal of the American Statistical Association}, 47(260),
  663--685.

\bibitem[Hudgens \& Halloran, 2008]{Hudgens2008}
Hudgens, M.~G. \& Halloran, M.~E. (2008).
\newblock \href{https://doi.org/10.1198/016214508000000292}{Toward causal
  inference with interference}.
\newblock {\em Journal of the American Statistical Association}, 103(482),
  832--842.

\bibitem[Imbens \& Angrist, 1994]{Imbens1994Identification}
Imbens, G.~W. \& Angrist, J.~D. (1994).
\newblock \href{https://doi.org/10.2307/2951620}{Identification and estimation
  of local average treatment effects}.
\newblock {\em Econometrica}, 62(2), 467--475.

\bibitem[Imbens \& Rubin, 2015]{Imbens2015}
Imbens, G.~W. \& Rubin, D.~B. (2015).
\newblock
  \href{https://www.cambridge.org/core/books/causal-inference-for-statistics-social-and-biomedical-sciences/71126BE90C58F1A431FE9B2DD07938AB}{{\em
  Causal Inference for Statistics, Social, and Biomedical Sciences: An
  Introduction}}.
\newblock New York: Cambridge University Press.

\bibitem[Isaki \& Fuller, 1982]{Isaki1982}
Isaki, C.~T. \& Fuller, W.~A. (1982).
\newblock \href{https://doi.org/10.2307/2287773}{Survey design under the
  regression superpopulation model}.
\newblock {\em Journal of the American Statistical Association}, 77(377),
  89--96.

\bibitem[Jagadeesan et~al., 2017]{Jagadeesan2017Designs}
Jagadeesan, R., Pillai, N., \& Volfovsky, A. (2017).
\newblock \href{https://arxiv.org/abs/1705.08524v1}{Designs for estimating the
  treatment effect in networks with interference}.
\newblock arXiv:1705.08524v1.

\bibitem[Janson et~al., 2000]{Janson2000Random}
Janson, S., {\L}uczak, T., \& Rucinski, A. (2000).
\newblock \href{https://doi.org/10.1002/9781118032718}{{\em Random Graphs}}.
\newblock New York: Wiley.

\bibitem[Kang \& Imbens, 2016]{Kang2016}
Kang, H. \& Imbens, G.~W. (2016).
\newblock \href{https://arxiv.org/abs/1609.04464v1}{Peer encouragement designs
  in causal inference with partial interference and identification of local
  average network effects}.
\newblock arXiv:1609.04464v1.

\bibitem[Lee, 2007]{Lee2007}
Lee, L. (2007).
\newblock \href{https://doi.org/10.1016/j.jeconom.2006.07.001}{Identification
  and estimation of econometric models with group interactions, contextual
  factors and fixed effects}.
\newblock {\em Journal of Econometrics}, 140(2), 333--374.

\bibitem[Lin, 2013]{Lin2013Agnostic}
Lin, W. (2013).
\newblock \href{https://doi.org/10.1214/12-aoas583}{Agnostic notes on
  regression adjustments to experimental data: {Reexamining} {Freedman's}
  critique}.
\newblock {\em Annals of Applied Statistics}, 7(1), 295--318.

\bibitem[Liu \& Hudgens, 2014]{Liu2014}
Liu, L. \& Hudgens, M.~G. (2014).
\newblock \href{https://doi.org/10.1080/01621459.2013.844698}{Large sample
  randomization inference of causal effects in the presence of interference}.
\newblock {\em Journal of the American Statistical Association}, 109(505),
  288--301.

\bibitem[Liu et~al., 2016]{Liu2016}
Liu, L., Hudgens, M.~G., \& Becker-Dreps, S. (2016).
\newblock \href{https://doi.org/10.1093/biomet/asw047}{On inverse
  probability-weighted estimators in the presence of interference}.
\newblock {\em Biometrika}, 103(4), 829--842.

\bibitem[Luo et~al., 2012]{Luo2012}
Luo, X., Small, D.~S., Li, C.-S.~R., \& Rosenbaum, P.~R. (2012).
\newblock \href{https://doi.org/10.1080/01621459.2012.655954}{Inference with
  interference between units in an {fMRI} experiment of motor inhibition}.
\newblock {\em Journal of the American Statistical Association}, 107(498),
  530--541.

\bibitem[Manski, 1993]{Manski1993Identification}
Manski, C.~F. (1993).
\newblock \href{https://doi.org/10.2307/2298123}{Identification of endogenous
  social effects: {The} reflection problem}.
\newblock {\em Review of Economic Studies}, 60(3), 531--542.

\bibitem[Manski, 2013]{Manski2013}
Manski, C.~F. (2013).
\newblock
  \href{https://doi.org/10.1111/j.1368-423X.2012.00368.x}{Identification of
  treatment response with social interactions}.
\newblock {\em The Econometrics Journal}, 16(1), 1--23.

\bibitem[Neyman, 1923]{Neyman1923}
Neyman, J. (1990/1923).
\newblock \href{https://doi.org/10.1214/ss/1177012031}{On the application of
  probability theory to agricultural experiments. {Essay} on principles.
  {Section} 9}.
\newblock {\em Statistical Science}, 5(4), 465--472.
\newblock (Original work published 1923).

\bibitem[Nickerson, 2008]{Nickerson2008}
Nickerson, D.~W. (2008).
\newblock \href{https://doi.org/10.1017/S0003055408080039}{Is voting
  contagious? {Evidence} from two field experiments}.
\newblock {\em American Political Science Review}, 102(1), 49--57.

\bibitem[Ogburn \& VanderWeele, 2017]{Ogburn2017}
Ogburn, E.~L. \& VanderWeele, T.~J. (2017).
\newblock \href{https://doi.org/10.1214/17-AOAS1023}{Vaccines, contagion, and
  social networks}.
\newblock {\em The Annals of Applied Statistics}, 11(2), 919--948.

\bibitem[Rigdon \& Hudgens, 2015]{Rigdon2015}
Rigdon, J. \& Hudgens, M.~G. (2015).
\newblock \href{https://doi.org/10.1016/j.spl.2015.06.011}{Exact confidence
  intervals in the presence of interference}.
\newblock {\em Statistics \& Probability Letters}, 105, 130--135.

\bibitem[Robinson, 1982]{Robinson1982}
Robinson, P.~M. (1982).
\newblock \href{https://doi.org/10.1111/j.1467-842X.1982.tb00829.x}{On the
  convergence of the {Horvitz-Thompson} estimator}.
\newblock {\em Australian Journal of Statistics}, 24(2), 234--238.

\bibitem[Rosenbaum, 2007]{Rosenbaum2007}
Rosenbaum, P.~R. (2007).
\newblock \href{https://doi.org/10.1198/016214506000001112}{Interference
  between units in randomized experiments}.
\newblock {\em Journal of the American Statistical Association}, 102(477),
  191--200.

\bibitem[Rosenblatt, 1956]{Rosenblatt1956}
Rosenblatt, M. (1956).
\newblock \href{http://www.pnas.org/content/42/1/43.full.pdf}{A central limit
  theorem and a strong mixing condition}.
\newblock {\em Proceedings of the National Academy of Sciences}, 42(1), 43--47.

\bibitem[Rubin, 1980]{Rubin1980}
Rubin, D.~B. (1980).
\newblock \href{https://doi.org/10.2307/2287653}{Comment: Randomization
  analysis of experimental data: {The} {Fisher} randomization test}.
\newblock {\em Journal of the American Statistical Association}, 75(371), 591.

\bibitem[Sinclair et~al., 2012]{Sinclair2012Detecting}
Sinclair, B., McConnell, M., \& Green, D.~P. (2012).
\newblock \href{https://doi.org/10.1111/j.1540-5907.2012.00592.x}{Detecting
  spillover effects: {Design} and analysis of multilevel experiments}.
\newblock {\em American Journal of Political Science}, 56(4), 1055--1069.

\bibitem[Sobel, 2006]{Sobel2006}
Sobel, M.~E. (2006).
\newblock \href{https://doi.org/10.1198/016214506000000636}{What do randomized
  studies of housing mobility demonstrate?}
\newblock {\em Journal of the American Statistical Association}, 101(476),
  1398--1407.

\bibitem[Sussman \& Airoldi, 2017]{Sussman2017}
Sussman, D.~L. \& Airoldi, E.~M. (2017).
\newblock \href{https://arxiv.org/abs/1702.03578v1}{Elements of estimation
  theory for causal effects in the presence of network interference}.
\newblock arXiv:1702.03578v1.

\bibitem[{Tchetgen Tchetgen} et~al., 2019]{TchetgenTchetgen2019Auto}
{Tchetgen Tchetgen}, E.~J., Fulcher, I., \& Shpitser, I. (2019).
\newblock \href{https://arxiv.org/abs/1709.01577v3}{Auto-g-computation of
  causal effects on a network}.
\newblock arXiv:1709.01577v3.

\bibitem[{Tchetgen Tchetgen} \& VanderWeele, 2012]{Tchetgen2012}
{Tchetgen Tchetgen}, E.~J. \& VanderWeele, T.~J. (2012).
\newblock \href{https://doi.org/10.1177/0962280210386779}{On causal inference
  in the presence of interference}.
\newblock {\em Statistical Methods in Medical Research}, 21(1), 55--75.

\bibitem[Toulis \& Kao, 2013]{Toulis2013}
Toulis, P. \& Kao, E. (2013).
\newblock \href{http://proceedings.mlr.press/v28/toulis13.html}{Estimation of
  causal peer influence effects}.
\newblock In S. Dasgupta \& D. McAllester (Eds.), {\em Proceedings of the 30th
  International Conference on Machine Learning}, volume~28 of {\em Proceedings
  of Machine Learning Research}  (pp.\ 1489--1497).  Atlanta: PMLR.

\bibitem[Ugander et~al., 2013]{Ugander2013}
Ugander, J., Karrer, B., Backstrom, L., \& Kleinberg, J. (2013).
\newblock \href{https://doi.org/10.1145/2487575.2487695}{Graph cluster
  randomization: {Network} exposure to multiple universes}.
\newblock In {\em Proceedings of the 19th ACM SIGKDD International Conference
  on Knowledge Discovery and Data Mining}, KDD '13  (pp.\ 329--337).  New York:
  ACM.

\bibitem[VanderWeele \& {Tchetgen Tchetgen}, 2011]{VanderWeele2011}
VanderWeele, T.~J. \& {Tchetgen Tchetgen}, E.~J. (2011).
\newblock \href{https://doi.org/10.1016/j.spl.2011.02.019}{Effect partitioning
  under interference in two-stage randomized vaccine trials}.
\newblock {\em Statistics \& Probability Letters}, 81(7), 861--869.

\end{thebibliography}

\clearpage

\setcounter{section}{0}%
\setcounter{table}{0}%
\setcounter{figure}{0}%
\setcounter{equation}{0}%

\renewcommand\thesection{A\arabic{section}}%
\renewcommand\thetable{A\arabic{table}}%
\renewcommand\thefigure{A\arabic{figure}}%
\renewcommand\theequation{A\arabic{equation}}%

\begin{center}
	\Huge \textbf{Supplement A: Proofs}
\end{center}
\bigskip

	\section{Overview}

	We start by proving consistency under the Bernoulli design for the \HT\ estimator. A bound on the variance for each of the two terms in the estimator follows from the definition of interference dependence (Lemma \ref{lem:bernoulli-hmu-rates}). (We consider the two terms of the estimators separately throughout the proof as this simplifies the derivations considerably.) The estimator is unbiased for \EATE\ (Lemma \ref{lem:bernoulli-adse-eate-diff}), so the variance bound provides convergence in probability (and in mean square) using the restricted interference assumption. Consistency for the \HA\ estimator can then be proven by linearization (Lemma \ref{lem:bernoulli-hmu-hn-rates}).

	This approach must be modified to accommodate other designs. In particular, the \HT\ estimator may not be unbiased for \EATE\ when treatment assignment is not independent. The route we take is to show that the \HT\ estimator is unbiased for \ADSE\ under any design that satisfy the regularity conditions. Once the variance of the \HT\ estimator is derived, this provides the rate of convergence with respect to \ADSE\ using the same method as above. Consistency of the estimators with respect to \EATE\ is proven by showing that \ADSE\ converges to \EATE\ under the stipulated conditions.

	Deriving the variance of the \HT\ estimator and the rate of convergence of \ADSE\ to \EATE\ under paired randomization follow essentially the same structure as under Bernoulli design (Lemma \ref{lem:paired-hmu-rates}). The only difference is that we also need to account for pair-induced interference dependence. The tasks are less straightforward for complete randomization and arbitrary designs.

	For complete randomization, our approach is to show that the dependence between two disjoint sets of treatment assignments can be bounded by the product of the sizes of the sets if the sets are sufficiently small (Lemma \ref{lem:complete-external-bound}). The dependence bound can then be used to bound the covariance between the outcomes of units that are not interference dependent if the number of units that interfere with them are sufficiently few (Lemma \ref{lem:complete-hmu-rates}). We show that our interference conditions imply that most units are affected by a sufficiently small number of other units, so the covariance bound can be applied for most pairs of units. The covariance for the remaining pairs is simply bounded by the regularity conditions. Convergence of \ADSE\ to \EATE\ follows a similar logic but applied to the first moments (Lemmas \ref{lem:complete-internal-bound} and \ref{lem:complete-adse-eate-diff}).

	For arbitrary designs, a bound on the variance of the \HT\ estimator follows fairly straightforwardly from a covariance bound based on the alpha-mixing coefficient known in the literature (Lemma \ref{lem:arbitrary-hmu-rates}). The proof of the convergence of \ADSE\ to \EATE\ builds on the same logic, but is less straightforward since, to the best of our knowledge, no bound currently exist in the literature and we must derive it ourselves (Lemma \ref{lem:arbitrary-adse-eate-diff}).

	\section{Miscellaneous definitions and lemmas}

	\begin{appdefinition} \label{def:all-mu}
		\begin{align}
		\tmu &= \frac{1}{n} \sum_{i=1}^{n} \E{y_i(z; \Z_{-i})}, & \cmu &= \frac{1}{n} \sum_{i=1}^{n} \E{y_i(z; \Z_{-i}) \given Z_i = z},
		\\
		\hmu &= \frac{1}{n} \sum_{i=1}^{n} \frac{\indicator{Z_i = z}Y_i}{\prob{Z_i = z}},& \hn &= \sum_{i=1}^{n} \frac{\indicator{Z_i = z}}{\prob{Z_i = z}}.
		\end{align}
	\end{appdefinition}

	\begin{appcorollary} \label{coro:est-mu}
		\begin{equation}
		\eEATE = \mu_1 - \mu_0, \qquad\quad \eADSE = \breve{\mu}_1 - \breve{\mu}_0, \qquad\quad \eHT = \hat{\mu}_1 - \hat{\mu}_0,\quad\qquad \eHA = \frac{n}{\hat{n}_1}\hat{\mu}_1 - \frac{n}{\hat{n}_0}\hat{\mu}_0,
		\end{equation}
		\begin{equation}
		\eHT - \eADSE = (\hat{\mu}_1 - \breve{\mu}_1) - (\hat{\mu}_0 - \breve{\mu}_0), \qquad\qquad \eHA - \eADSE = \paren[\bigg]{\frac{n}{\hat{n}_1}\hat{\mu}_1 - \breve{\mu}_1} - \paren[\bigg]{\frac{n}{\hat{n}_0}\hat{\mu}_0 - \breve{\mu}_0}.
		\end{equation}
	\end{appcorollary}

	\begin{proof}
		Follows directly from Definitions \ref{def:all-mu}, \ref{def:eate}, \ref{def:adse} and \ref{def:estimators}.
	\end{proof}

	\begin{applemma} \label{lem:hmu-cmu-bias}
		$\E{\hmu} = \cmu$.
	\end{applemma}

	\begin{proof}
		Assumption \ref{ass:probabilistic-assignment} ensures that $\prob{Z_i = z}>0$. Recall that $Y_i = y_i(Z_i; \Z_{-i})$, so:
		\begin{align}
		\E{\hmu} &= \frac{1}{n} \sum_{i=1}^{n} \frac{\E[\big]{\indicator{Z_i = z}Y_i}}{\prob{Z_i = z}} = \frac{1}{n} \sum_{i=1}^{n} \frac{\prob{Z_i = z}\E{Y_i \given Z_i = z}}{\prob{Z_i = z}}
		\\
		&= \frac{1}{n} \sum_{i=1}^{n} \E{y_i(z; \Z_{-i}) \given Z_i = z}. \qedhere
		\end{align}
	\end{proof}

	\begin{applemma} \label{lem:hn-bias}
		$\E{\hn} = n$.
	\end{applemma}

	\begin{proof}
		Assumption \ref{ass:probabilistic-assignment} ensures that $\prob{Z_i = z}>0$, so:
		\begin{equation}
		\E{\hn} = \sum_{i=1}^{n} \frac{\E[\big]{\indicator{Z_i = z}}}{\prob{Z_i = z}} = \sum_{i=1}^{n} \frac{\prob{Z_i = z}}{\prob{Z_i = z}} = n. \tag*{\qedhere}
		\end{equation}
	\end{proof}

\begin{applemma}\label{lem:bounded-covariance}
	If Assumption~\ref{ass:bounded-moments} holds with $q \geq 2a$, then:
	\begin{equation}
		\Cov[\big]{\indicator{Z_i = z}Y_i^a, \indicator{Z_j = z}Y_j^a} \leq k^{2a},
	\end{equation}
	for all $i,j\in\bfU$ and $z\in\braces{0,1}$.
\end{applemma}

\begin{proof}
The Cauchy--Schwarz inequality gives:
\begin{align}
	\Cov[\big]{\indicator{Z_i = z}Y_i^a, \indicator{Z_j = z}Y_i^a}
	&\leq \bracket[\Big]{\Var[\big]{\indicator{Z_i = z}Y_i^a}\Var[\big]{\indicator{Z_j = z}Y_i^a}}^{0.5}
	\\
	&\leq \bracket[\Big]{\E[\big]{\abs{Y_i}^{2a}}\E[\big]{\abs{Y_j}^{2a}}}^{0.5}
	\\
	&\leq \bracket[\bigg]{\E[\big]{\abs{Y_i}^q}^{2a/q}\E[\big]{\abs{Y_j}^q}^{2a/q}}^{0.5}
	\tag*{Jensen's inequality}
	\\
	&\leq \bracket[\Big]{\bracket{k^q}^{2a/q}\bracket{k^q}^{2a/q}}^{0.5}
	\tag*{Assumption \ref{ass:bounded-moments}}
	\\
	& = \bracket[\big]{k^{4a}}^{0.5} = k^{2a}. \qedhere
\end{align}
\end{proof}

	\begin{applemma} \label{lem:bounded-uncond-cond-po-diff}
		$\abs[\big]{\E{y_i(z; \Z_{-i}) \given Z_i = z} - \E{y_i(z; \Z_{-i})}} \leq 2k^2$ for all $i\in\bfU$ and $z\in\braces{0,1}$.
	\end{applemma}

	\begin{proof}
		Assumption \ref{ass:probabilistic-assignment} and the law of total expectation give:
		\begin{align}
		\E[\big]{\abs{y_i(z; \Z_{-i})} \given Z_i = z} &\leq \E[\big]{\abs{y_i(z; \Z_{-i})} \given Z_i = z}
		\\
		&\qquad\qquad+ \frac{\prob{Z_i = 1-z}\E[\big]{\abs{y_i(1-z; \Z_{-i})} \given Z_i = 1-z} }{\prob{Z_i = z}}
		\\
		&= \frac{\E[\big]{\abs{y_i(Z_i; \Z_{-i})}}}{\prob{Z_i = z}} = \frac{\E[\big]{\abs{Y_i}}}{\prob{Z_i = z}} \leq k \E[\big]{\abs{Y_i}}.
		\end{align}
		Consider the expression to be bounded:
		\begin{align}
		\abs[\big]{\E{y_i(z; \Z_{-i}) \given Z_i = z} - \E{y_i(z; \Z_{-i})}} &\leq \E[\big]{\abs{y_i(z; \Z_{-i})} \given Z_i = z} + \E[\big]{\abs{y_i(z; \Z_{-i})}}
		\\
		&\leq k \E[\big]{\abs{Y_i}} + \E[\big]{\abs{y_i(z; \Z_{-i})}} \leq 2k^2,
		\end{align}
		where the last inequality follows from Assumptions \ref{ass:bounded-moments} and \ref{ass:bounded-expected-pos}.
	\end{proof}

	\begin{applemma} \label{lem:l2-to-probability}
		If the rate of convergence in the $L_2$-norm for a sequence of random variables is govern by a sequence $(a_n)$, then the rate of convergence in probability is govern by $(a_n)$.
	\end{applemma}

	\begin{proof}
		It is well-known that convergence in the $L_2$-norm implies convergence in probability. It may be less known that the same hold for the rates of convergence. For completeness, we provide a proof.

		Consider a sequence of random variables $(X_n)$ and a sequence of non-random variables $(x_n)$. Let the sequence $(a_n)$ denote the  rate of convergence in the $L_2$-norm of $X_n$ to $x_n$:
		\begin{equation}
			\sqrt{\E[\Big]{\paren[\big]{X_n - x_n}^2}} = \bigO{a_n}.
		\end{equation}
		Let $\kappa$ be an arbitrary positive constant. Markov's inequality gives:
		\begin{equation}
		\prob[\Big]{\abs[\big]{X_n-x_n} \geq \kappa a_n} = \prob[\big]{\paren[\big]{X_n-x_n}^2 \geq \kappa^2a_n^2} \leq \frac{\E[\Big]{\paren[\big]{X_n-x_n}^2}}{\kappa^2a_n^2},
		\end{equation}
		which, by the rate of convergence in the $L_2$-norm, is bounded by $\lambda\kappa^{-2}$ for some constant $\lambda$ when $n$ is sufficiently large. Thus, the probability that $X_n - x_n$ deviate from zero with more than a fixed multiple of $a_n$ can be bounded by any $\varepsilon > 0$ for sufficiently large $n$ by setting $\kappa=\lambda^{0.5}\varepsilon^{-0.5}$. In other words, $X_n - x_n = \bigOp[\big]{a_n}$.
	\end{proof}

	\begin{applemma} \label{lem:hmu-hn-rates}
		If $\hmu - \cmu=\littleOp{1}$ and $\hn/n - 1=\littleOp{1}$, then:
		\begin{equation}
		\frac{n}{\hn}\hmu - \cmu = \bigOp[\big]{\paren{\hmu - \cmu} + \paren{\hn/n - 1}}.
		\end{equation}
	\end{applemma}

	\begin{proof}
		Recall that $Y_i = y_i(Z_i; \Z_{-i})$ and consider the absolute value of $\cmu$:
		\begin{align}
		\abs{\cmu} &= \abs[\bigg]{\frac{1}{n} \sum_{i=1}^{n} \E{Y_i \given Z_i = z}} \leq \frac{1}{n} \sum_{i=1}^{n} \E[\big]{\abs{Y_i} \given Z_i = z}
		\\
		&\leq \frac{k}{n} \sum_{i=1}^{n} \prob{Z_i = z}\E[\big]{\abs{Y_i} \given Z_i = z} \tag*{Assumption \ref{ass:probabilistic-assignment}}
		\\
		&\leq \frac{k}{n} \sum_{i=1}^{n} \paren[\Big]{\prob{Z_i = 1}\E[\big]{\abs{Y_i} \given Z_i = 1} + \prob{Z_i = 0}\E[\big]{\abs{Y_i} \given Z_i = 0}}
		\\
		&= \frac{k}{n} \sum_{i=1}^{n} \E[\big]{\abs{Y_i}} \leq k^{2},
		\end{align}
		where the last inequality follows from Assumption \ref{ass:bounded-moments}. Together with the premise that $\hn/n - 1=\littleOp{1}$, this ensures that $\paren{\hn/n - 1}\cmu=\littleOp{1}$.

		The second premise ensures that we can disregard the event $\hn = 0$, which gives:
		\begin{equation}
		\frac{n}{\hn}\hmu - \cmu = \frac{\hmu - \cmu}{\hn/n} - \frac{\paren{\hn/n - 1}\cmu}{\hn/n}.
		\end{equation}
		Let $f(a,b) = a / b$ and consider the first order Taylor expansion of the two terms around $(0, 1)$:
		\begin{align}
		\frac{\hmu - \cmu}{\hn/n} &= f\paren[\big]{\hmu - \cmu,\hn/n} = (\hmu - \cmu) + r_1,
		\\
		\frac{\paren{\hn/n - 1}\cmu}{\hn/n} &= f\paren[\big]{\paren{\hn/n - 1}\cmu,\hn/n} = \paren{\hn/n - 1}\cmu + r_2,
		\end{align}
		so:
		\begin{equation}
		\frac{n}{\hn}\hmu - \cmu = (\hmu - \cmu) - \paren{\hn/n - 1}\cmu + r_1 - r_2,
		\end{equation}
		where $r_1 = \littleOp[\big]{(\hmu - \cmu) + \paren{\hn/n - 1}}$ and $r_2 = \littleOp[\big]{\paren{\hn/n - 1}\paren{\cmu + 1}}$ since $\hmu-\cmu$ converges in probability to zero and $\hn/n$ converges to one. The bound $\abs{\cmu} \leq k^{2}$ completes the proof. Note that we cannot expand $n\hmu/ \hn$ directly since $\hmu$ and $\cmu$ may not be convergent separately; our assumptions only ensure that $\hmu - \cmu$ converges.
	\end{proof}

\section{Proof of Lemma \ref{lem:interference-measure-inequalities}}

\begin{reflem}{\ref{lem:interference-measure-inequalities}}
	$\max\paren[\big]{\icoutavg, n^{-1}\icoutmax^2} \leq \davg \leq \icoutmsq \leq \icoutmax^2$.
\end{reflem}

\begin{proof}
Start by showing $\icoutavg \leq \davg \leq \icoutmsq \leq \icoutmax^2$.
Note that $I_{ij} \leq \di_{ij} \leq \sum_{\ell=1}^{n}I_{\ell i} I_{\ell j}$ by construction of Definition \ref{def:interference-dependence}, and consider:
\begin{align}
	\icoutavg &= \frac{1}{n}\sum_{i=1}^{n}\sum_{j=1}^n I_{ij} \leq \frac{1}{n}\sum_{i=1}^{n} \sum_{j=1}^n \di_{ij} = \davg \leq \frac{1}{n}\sum_{i=1}^{n}\sum_{j=1}^{n}\sum_{\ell=1}^{n}I_{\ell i} I_{\ell j}
	\\
	&= \frac{1}{n}\sum_{\ell=1}^{n} \paren[\bigg]{\sum_{i=1}^{n}I_{\ell i}} \paren[\bigg]{\sum_{j=1}^{n} I_{\ell j}} = \frac{1}{n}\sum_{\ell=1}^{n} \icout_\ell^2 = \icoutmsq \leq \frac{1}{n}\sum_{i=1}^{n}\icoutmax^2 = \icoutmax^2.
\end{align}

Next, show $n^{-1}\icoutmax^2 \leq \davg$.
Consider the smallest $\davg$ for a given $\icoutmax$.
This is when one unit $\ell$ interferes with $\icoutmax$ units (including itself) and all other units interfere only with themselves.
That is, when $\icout_\ell = \icoutmax$ and $\icout_i = 1$ for $i \neq \ell$.
For any other configuration, we can remove interference relations for some pairs of units without affecting $\icoutmax$, but the removal would weakly decrease $\davg$.
Let $A = \braces{i : I_{\ell i} = 1}$ be the set of units which unit $\ell$ interferes with.
The interference dependencies will be such that $\di_{ij} = 1$ if and only if $i, j \in A$ or $i = j$.
There are $\abs{A} = \icoutmax$ units in $A$ by construction, so the number of pairs of units with $i, j \in A$ is $\icoutmax^2$.
The number of pairs with $i = j$ is $n$, but $\icoutmax$ of them also satisfy $i, j \in A$.
Hence, the number of pairs $i, j \in \bfU$ such that $\di_{ij} = 1$ is $\icoutmax^2 + (n - \icoutmax)$, so for this configuration:
\begin{equation}
	\davg = \frac{1}{n}\sum_{i=1}^{n}\sum_{j=1}^n\di_{ij} = \frac{\icoutmax^2 + (n - \icoutmax)}{n}.
\end{equation}
Because this was the configuration with the smallest $\davg$ for a given $\icoutmax$, it is a lower bound of $\davg$ for any configuration.
That is, the following inequality holds for any interference structure:
\begin{equation}
	\frac{\icoutmax^2 + (n - \icoutmax)}{n} \leq \davg.
\end{equation}
After rearranging terms and noting that $(n - \icoutmax) \geq 0$, we get:
\begin{equation}
	\icoutmax^2 \leq n\davg - (n - \icoutmax) \leq n\davg. \tag*{\qedhere}
\end{equation}
\end{proof}

	\begin{applemma} \label{lem:paired-intmeasure-inequality}
		$\eavg \leq \icoutmsq$.
	\end{applemma}

	\begin{proof}
		The proof follows a similar structure as the proof of Lemma \ref{lem:interference-measure-inequalities}. Definition \ref{def:paired-interference-dependence} gives $\ei_{ij} = \sum_{\ell=1}^{n}(1-\di_{ij})I_{\ell i} I_{\rho(\ell) j}\leq \sum_{\ell=1}^{n}I_{\ell i} I_{\rho(\ell) j}$. Note that $2xy \leq x^2 + y^2$ for any $x$ and $y$ by Young's inequality for products. Now consider:
		\begin{align}
		\eavg &= \frac{1}{n}\sum_{i=1}^{n} \sum_{j=1}^n \ei_{ij} \leq \frac{1}{n}\sum_{i=1}^{n}\sum_{j=1}^{n}\sum_{\ell=1}^{n}I_{\ell i} I_{\rho(\ell) j} =\frac{1}{n}\sum_{\ell=1}^{n} \icout_\ell\icout_{\rho(\ell)}
		\\
		&\leq \frac{1}{2n}\sum_{\ell=1}^{n} \bracket[\big]{\icout_\ell^2 + \icout_{\rho(\ell)}^2} = \frac{1}{2n}\sum_{\ell=1}^{n} \icout_\ell^2 + \frac{1}{2n}\sum_{\ell=1}^{n} \icout_{\rho(\ell)}^2 = \icoutmsq. \qedhere
		\end{align}
	\end{proof}

	\section{Proof of Proposition \ref{prop:restricted-interference-necessary}}

	\begin{refprop}{\ref{prop:restricted-interference-necessary}}
		For every sequence of experimental designs, if Assumption~\ref{ass:restricted-interference} does not hold, there exists a sequence of potential outcomes satisfying Assumption~\ref{ass:regularity-conditions} such that the \HT\ and \HA\ estimators do not converge in probability to \EATE.
	\end{refprop}

	\begin{proof}
		Let $\lambda = \limsup_{n\to\infty}n^{-1}\davg$. Assumption \ref{ass:restricted-interference} is equivalent to $\lambda = 0$, so if it does not hold, we know that $\lambda$ is strictly greater than zero. The definition of $\lambda$ implies that $\davg \geq \lambda n$ happens infinitely often. In these cases, set the potential outcomes to:
		\begin{equation}
		y_i(\z) = \left\{ \begin{array}{ll}
		n\ceil{\lambda^{0.5}n}^{-1}z_1\bracket[\big]{z_i\prob{Z_i = 1} - (1-z_i)\prob{Z_i = 0}} & \text{if } i \leq \ceil{\lambda^{0.5}n},
		\\[0.5em]
		0 & \text{otherwise},
		\end{array}\right.
		\end{equation}
		where $\ceil{\cdot}$ is the ceiling function. This implies:
		\begin{equation}
		\di_{ij} = \left\{ \begin{array}{ll}
		1 & \text{if } i = j \text{ or } i,j\leq \ceil{\lambda^{0.5}n},
		\\[0.5em]
		0 & \text{otherwise},
		\end{array}\right.
		\end{equation}
		so we have:
		\begin{equation}
		\davg = \frac{1}{n}\sum_{i=1}^{n}\sum_{j=1}^n\di_{ij} = \frac{1}{n}\sum_{i=1}^{\ceil{\lambda^{0.5}n}}\sum_{j=1}^{\ceil{\lambda^{0.5}n}}\di_{ij} + \frac{1}{n}\sum_{i=\ceil{\lambda^{0.5}n}+1}^{n}\di_{ii} = \frac{\ceil{\lambda^{0.5}n}^2}{n} + \frac{n - \ceil{\lambda^{0.5}n}}{n},
		\end{equation}
		which gives $\davg \geq \lambda n$ for any finite $n$ and $\davg \sim \lambda n$ asymptotically, as desired. (These potential outcomes may not replicate a specified sequence of $\davg$ exactly. However, the limit superior of $n^{-1}\davg$ for the constructed sequence is $\lambda$, and it thus replicates the relevant asymptotic behavior of the specified sequence. One can modify the proof to replicate the specified sequence exactly, but this would complicate the proof without providing additional insights.)

		We must now ensure that the stipulated potential outcomes satisfy our regularity conditions. For $i > \ceil{\lambda^{0.5}n}$, we have $y_i(\z)=0$ for all $\z$, so the moment conditions in Assumption \ref{ass:regularity-conditions} trivially hold. For $i \leq \ceil{\lambda^{0.5}n}$, we can bound the potential outcomes for all $\z$ by:
		\begin{align}
		\abs{y_i(\z)} &\leq n\ceil{\lambda^{0.5}n}^{-1}\bracket[\big]{\prob{Z_i = 1} + \prob{Z_i = 0}}  \leq \lambda^{-0.5},
		\end{align}
		so Assumptions \ref{ass:bounded-moments} and \ref{ass:bounded-expected-pos} hold even when $q\to\infty$ and $s\to\infty$ since $k=\lambda^{-0.5}$ is a valid constant for the moment conditions.

		Consider the \EATE\ estimand:
		\begin{align}
		\eEATE &= \frac{1}{n}\sum_{i=1}^{n} \paren[\Big]{\E{y_i(1; \Z_{-i})} - \E{y_i(0; \Z_{-i})}}
		\\
		&= \frac{\E{y_1(1; \Z_{-1})}}{n} + \sum_{i=2}^{n} \paren[\bigg]{\frac{\E{y_i(1; \Z_{-i})}}{n} - \frac{\E{y_i(0; \Z_{-i})}}{n}}
		\\
		&= \frac{\prob{Z_1 = 1}}{\ceil{\lambda^{0.5}n}} + \sum_{i=2}^{\ceil{\lambda^{0.5}n}} \paren[\Big]{\E[\big]{\ceil{\lambda^{0.5}n}^{-1}Z_1\prob{Z_i = 1}} + \E[\big]{\ceil{\lambda^{0.5}n}^{-1}Z_1\prob{Z_i = 0}}}
		\\
		&= \frac{\prob{Z_1 = 1}}{\ceil{\lambda^{0.5}n}} + \frac{\prob{Z_1 = 1}}{\ceil{\lambda^{0.5}n}}\sum_{i=2}^{\ceil{\lambda^{0.5}n}} \paren[\Big]{\prob{Z_i = 1} + \prob{Z_i = 0}} = \prob{Z_1 = 1}.
		\end{align}
		Then consider the \HT\ estimator:
		\begin{align}
		\eHT &= \frac{1}{n} \sum_{i=1}^{n}\frac{Z_i Y_i}{\prob{Z_i = 1}} - \frac{1}{n} \sum_{i=1}^{n} \frac{(1-Z_i) Y_i}{\prob{Z_i = 0}}
		\\
		&= \frac{Z_1 y_1(\Z)}{n\prob{Z_1 = 1}} + \frac{1}{n} \sum_{i=2}^{n}\frac{Z_i y_i(\Z)}{\prob{Z_i = 1}} - \frac{1}{n} \sum_{i=2}^{n} \frac{(1-Z_i) y_i(\Z)}{\prob{Z_i = 0}}
		\\
		&= \frac{Z_1}{\ceil{\lambda^{0.5}n}} + \frac{1}{n} \sum_{i=2}^{\ceil{\lambda^{0.5}n}}\frac{n\ceil{\lambda^{0.5}n}^{-1}Z_1Z_i\prob{Z_i = 1}}{\prob{Z_i = 1}}
		\\
		&\qquad\qquad + \frac{1}{n} \sum_{i=2}^{\ceil{\lambda^{0.5}n}} \frac{ n\ceil{\lambda^{0.5}n}^{-1}Z_1(1-Z_i)\prob{Z_i = 0}}{\prob{Z_i = 0}}
		\\
		&= \frac{Z_1}{\ceil{\lambda^{0.5}n}} + \frac{Z_1}{\ceil{\lambda^{0.5}n}} \sum_{i=2}^{\ceil{\lambda^{0.5}n}}Z_i + \frac{Z_1}{\ceil{\lambda^{0.5}n}} \sum_{i=2}^{\ceil{\lambda^{0.5}n}}(1-Z_i) = Z_1.
		\end{align}
		By Assumption \ref{ass:probabilistic-assignment}, there exists a constant $k>0$ such that $k^{-1} \leq \prob{Z_1 = 1} \leq 1 - k^{-1}$. Subsequently, $\abs{\eHT - \eEATE} = \abs{Z_1 - \prob{Z_1 = 1}} \geq k^{-1}$ with probability one whenever $\davg \geq \lambda n$, which, as we already noted, happens infinitely often. The proof for $\eHA$ follows the same structure.
	\end{proof}

	\section{Proof of Proposition \ref{prop:bernoulli-rates}}

	\begin{applemma} \label{lem:bernoulli-adse-eate-diff}
		Under Bernoulli randomization, $\eEATE=\eADSE$.
	\end{applemma}

	\begin{proof}
		By Definition \ref{def:adse} and $Y_i = y_i(Z_i; \Z_{-i})$, we have:
		\begin{equation}
		\eADSE = \frac{1}{n} \sum_{i=1}^{n} \paren[\Big]{\E{y_i(1; \Z_{-i}) \given Z_i = 1} - \E{y_i(0; \Z_{-i}) \given Z_i = 0}}.
		\end{equation}
		The Bernoulli design makes the conditioning inconsequential since $\Z_{-i}$ and $Z_i$ are independent, and we get:
		\begin{equation}
		\eADSE = \frac{1}{n} \sum_{i=1}^{n} \paren[\Big]{\E{y_i(1; \Z_{-i})} - \E{y_i(0; \Z_{-i})}}.
		\end{equation}
		Finally, we move the summation inside the expectation and apply Definitions \ref{def:acate} and \ref{def:eate}:
		\begin{equation}
		\eADSE = \E[\bigg]{\frac{1}{n} \sum_{i=1}^{n} \bracket{y_i(1; \Z_{-i}) - y_i(0; \Z_{-i})}} = \E[\bigg]{\frac{1}{n} \sum_{i=1}^{n} \tau_i(\Z_{-i})} = \E[\big]{\eATE(\Z)} = \eEATE. \tag*{\qedhere}
		\end{equation}
	\end{proof}

	\begin{applemma} \label{lem:bernoulli-hmu-rates}
		Under Bernoulli randomization, $\hmu - \cmu = \bigOp[\big]{n^{-0.5}\davg^{0.5}}$.
	\end{applemma}

	\begin{proof}
		Consider the variance of $\hmu$:
		\begin{align}
		\Var{\hmu} &= \Var[\bigg]{\frac{1}{n} \sum_{i=1}^{n} \frac{\indicator{Z_i = z}Y_i}{\prob{Z_i = z}}} = \frac{1}{n^2}\sum_{i=1}^{n} \sum_{j=1}^{n} \frac{\Cov[\big]{\indicator{Z_i = z}Y_i, \indicator{Z_j = z}Y_j}}{\prob{Z_i = z}\prob{Z_j = z}}
		\\
		&\leq \frac{k^2}{n^2}\sum_{i=1}^{n} \sum_{j=1}^{n} \Cov[\big]{\indicator{Z_i = z}Y_i, \indicator{Z_j = z}Y_j}.
		\tag*{Assumption \ref{ass:probabilistic-assignment}}
		\end{align}
		Recall $\di_{ij}$ and $\davg$ from Definition \ref{def:interference-dependence}:
		\begin{align}
		\Var{\hmu} &\leq \frac{k^2}{n^2}\sum_{i=1}^{n} \sum_{j=1}^{n} \bracket[\big]{\di_{ij} + (1-\di_{ij})}\Cov[\big]{\indicator{Z_i = z}Y_i, \indicator{Z_j = z}Y_j}
		\\
		&\leq  \frac{k^4}{n^2}\sum_{i=1}^{n} \sum_{j=1}^{n} \di_{ij} + \frac{k^2}{n^2}\sum_{i=1}^{n} \sum_{j=1}^{n} (1-\di_{ij}) \Cov[\big]{\indicator{Z_i = z}Y_i, \indicator{Z_j = z}Y_j},
		\\
		&\leq k^4n^{-1}\davg + \frac{k^2}{n^2}\sum_{i=1}^{n} \sum_{j=1}^{n} (1-\di_{ij}) \Cov[\big]{\indicator{Z_i = z}Y_i, \indicator{Z_j = z}Y_j},
		\end{align}
		where the second inequality follows from Lemma \ref{lem:bounded-covariance} and the last from Definition \ref{def:interference-dependence}.

		Recall that we can write $Y_i = y_i(\Z) = y_i\paren[\big]{\Zint_i}$, and thus:
		\begin{equation}
		\Cov[\big]{\indicator{Z_i = z}Y_i, \indicator{Z_j = z}Y_j} = \Cov[\Big]{\indicator{Z_i = z}y_i\paren[\big]{\Zint_i}, \indicator{Z_j = z}y_j\paren[\big]{\Zint_j}}.
		\end{equation}
		Under the Bernoulli design, treatment assignments are independent. When $\di_{ij}=0$, no treatment affects both units $i$ and $j$. All treatments affecting a unit are collected in its $\Zint_i$, so $\di_{ij}=0$ implies that $\Zint_i$ and $\Zint_j$ are disjoint and, thus, independent. It follows:
		\begin{equation}
		(1-\di_{ij}) \Cov[\big]{\indicator{Z_i = z}Y_i, \indicator{Z_j = z}Y_j} = 0.
		\end{equation}
		We conclude that $\Var{\hmu} \leq k^4n^{-1}\davg$.

		Lemma \ref{lem:hmu-cmu-bias} provides the rate of convergence in the $L_2$-norm:
		\begin{equation}
			\sqrt{\E[\Big]{\paren[\big]{\hmu - \cmu}^2}} = \sqrt{\E[\Big]{\paren[\big]{\hmu - \E{\hmu}}^2}} = \sqrt{\Var{\hmu}} = \bigO{n^{-0.5}\davg^{0.5}},
		\end{equation}
		which gives the rate of convergence in probability using Lemma \ref{lem:l2-to-probability}.
	\end{proof}

	\begin{applemma} \label{lem:bernoulli-hmu-hn-rates}
		Under Bernoulli randomization and Assumption \ref{ass:restricted-interference}:
		\begin{equation}
		\frac{n}{\hn}\hmu - \cmu = \bigOp[\big]{n^{-0.5}\davg^{0.5}}.
		\end{equation}
	\end{applemma}

	\begin{proof}
		Consider the variance of $\hn / n$:
		\begin{align}
		\Var{\hn / n} &= \Var[\bigg]{\frac{1}{n} \sum_{i=1}^{n} \frac{\indicator{Z_i = z}}{\prob{Z_i = z}}} = \frac{1}{n^2} \sum_{i=1}^{n} \frac{\Var[\big]{\indicator{Z_i = z}}}{\bracket[\big]{\prob{Z_i = z}}^2}
		\\
		&= \frac{1}{n^2} \sum_{i=1}^{n} \frac{1 - \prob{Z_i = z}}{\prob{Z_i = z}} \leq kn^{-1},
		\end{align}
		where the last inequality follows from Assumption \ref{ass:probabilistic-assignment}. Lemma \ref{lem:hn-bias} gives the rate of convergence in the $L_2$-norm of $\hn / n$ with respect to a sequence of ones:
		\begin{equation}
			\sqrt{\E[\Big]{\paren[\big]{\hn / n - 1}^2}} = \sqrt{\E[\Big]{\paren[\big]{\hn / n - \E{\hn / n}}^2}} = \sqrt{\Var{\hn / n}} = \bigO{n^{-0.5}},
		\end{equation}
		which gives $\hn / n - 1 = \bigOp[\big]{n^{-0.5}} = \littleOp{1}$ using Lemma \ref{lem:l2-to-probability}.

		Assumption \ref{ass:restricted-interference} and Lemma \ref{lem:bernoulli-hmu-rates} imply $\hmu - \cmu = \littleOp{1}$. This allows for the application of Lemma \ref{lem:hmu-hn-rates}, which together with $\davg^{0.5} \geq 1$ give:
		\begin{equation}
		\frac{n}{\hn}\hmu - \cmu = \bigOp[\big]{n^{-0.5}\davg^{0.5} + n^{-0.5}} = \bigOp[\big]{n^{-0.5}\davg^{0.5}}. \tag*{\qedhere}
		\end{equation}
	\end{proof}

	\begin{refprop}{\ref{prop:bernoulli-rates}}
		With a Bernoulli randomization design under restricted interference (Assumption \ref{ass:restricted-interference}), the \HT\ and \HA\ estimators are consistent for \EATE\ and converge at the following rates:
		\begin{equation}
		\eHT - \eEATE = \bigOp[\big]{n^{-0.5}\davg^{0.5}},
		\qquad\text{and}\qquad
		\eHA - \eEATE = \bigOp[\big]{n^{-0.5}\davg^{0.5}}.
		\end{equation}
	\end{refprop}

	\begin{proof}
		Using Corollary \ref{coro:est-mu} and Lemma \ref{lem:bernoulli-adse-eate-diff}, decompose the errors of the estimators:
		\begin{align}
		\eHT - \eEATE &=  \eHT - \eADSE = (\hat{\mu}_1 - \breve{\mu}_1) - (\hat{\mu}_0 - \breve{\mu}_0),
		\\
		\eHA - \eEATE &= \eHA - \eADSE = \paren[\bigg]{\frac{n}{\hat{n}_1}\hat{\mu}_1 - \breve{\mu}_1} - \paren[\bigg]{\frac{n}{\hat{n}_0}\hat{\mu}_0 - \breve{\mu}_0}.
		\end{align}
		Lemmas \ref{lem:bernoulli-hmu-rates} and \ref{lem:bernoulli-hmu-hn-rates} give the rates. Assumption \ref{ass:restricted-interference} gives consistency.
	\end{proof}

	\section{Proof of Proposition \ref{prop:complete-rates}}

	\begin{appdefinition}
		Let $m = \sum_{i=1}^{n} Z_i $ be the element sum of $\Z$ (i.e., the number of treated units).
	\end{appdefinition}

	\begin{appremark}
		Complete randomization implies that $m=\lfloor p n \rfloor$ with probability one for some fixed $p$ strictly between zero and one.
	\end{appremark}

	\begin{appdefinition} \label{def:inwards-interference-measure}
		Let $\icin_i = \sum_{j=1}^n I_{ji}$ be the number of units interfering with unit $i$.
	\end{appdefinition}

	\begin{applemma} \label{lem:icoutavg-handshake}
		$n^{-1}\sum_{i=1}^n \icin_i = \icoutavg$.
	\end{applemma}

	\begin{proof}
		Follows from a handshaking argument. Recall that $\icoutavg = n^{-1}\sum_{i=1}^n \icout_i$, and reorder the summation as such:
		\begin{equation}
		\frac{1}{n}\sum_{i=1}^n \icin_i = \frac{1}{n}\sum_{i=1}^n \sum_{j=1}^n I_{ji} = \frac{1}{n}\sum_{j=1}^n \sum_{i=1}^n I_{ji} = \frac{1}{n}\sum_{j=1}^n \icout_j = \icoutavg. \tag*{\qedhere}
		\end{equation}
	\end{proof}

	\begin{applemma} \label{lem:complete-internal-bound}
		Under complete randomization, for any $i\in\bfU$ such that $\icin_i \leq \min(m, n-m)$ and any $z\in \braces{0,1}$:
		\begin{equation}
		\max_{\z\in \suppZint_{-i}} \;\abs[\bigg]{\frac{\prob{\Zint_{-i} = \z \given Z_i = z}}{\prob{\Zint_{-i} = \z}} - 1} \leq \frac{\icin_i}{\prob{Z_i = z}\min(m, n-m)},
		\end{equation}
		where $\suppZint_{-i} = \braces[\big]{\z\in\braces{0,1}^{n-1} : \prob{\Zint_{-i} = \z} > 0}$ is the support of $\Zint_{-i}$.
	\end{applemma}

	\begin{proof}
		Let $t_i(\z) = \sum_{j\neq i} I_{ji} z_j$ be the element sum of $\z\in \suppZint_{-i}$, and let $T_i = \sum_{j\neq i} I_{ji} Z_j$ be the element sum of $\Zint_{-i}$  (i.e., the number of treated units interfering with $i$, excluding $i$ itself). Complete randomization implies that $Z_i$ and $\Zint_{-i}$ are independent conditional on $T_i$, and as a consequence:
		\begin{align}
		\frac{\prob[\big]{\Zint_{-i} = \z \given Z_i = z}}{\prob[\big]{\Zint_{-i} = \z}} &= \frac{\prob[\big]{T_i = t_i(\z) \given Z_i = z}\prob[\big]{\Zint_{-i} = \z \given T_i = t_i(\z), Z_i = z}}{\prob[\big]{T_i = t_i(\z)}\prob[\big]{\Zint_{-i} = \z \given T_i = t_i(\z)}}
		\\
		&= \frac{\prob[\big]{T_i = t_i(\z) \given Z_i = z}\prob[\big]{\Zint_{-i} = \z \given T_i = t_i(\z)}}{\prob[\big]{T_i = t_i(\z)}\prob[\big]{\Zint_{-i} = \z \given T_i = t_i(\z)}}
		\\
		&= \frac{\prob[\big]{T_i = t_i(\z) \given Z_i = z}}{\prob[\big]{T_i = t_i(\z)}}.
		\end{align}
		We have $T_i \leq \icin_i \leq \min(m, n-m)$, which, together with complete randomization, implies that $T_i$ is hyper-geometric according to:
		\begin{align}
		\prob[\big]{T_i = t} &= \binom{m}{t}\binom{n-m}{\icin_i - t}\binom{n}{\icin_i}^{-1},
		\\
		\prob[\big]{T_i = t \given Z_i = z} &= z\binom{m-1}{t}\binom{n-m}{\icin_i - t}\binom{n-1}{\icin_i}^{-1}
		\\
		&\qquad\qquad + (1-z) \binom{m}{t}\binom{n-m-1}{\icin_i - t}\binom{n-1}{\icin_i}^{-1}.
		\end{align}
		It follows that:
		\begin{equation}
		\frac{\prob[\big]{T = t \given Z_i = z}}{\prob[\big]{T = t}} =  \frac{\binom{n}{\icin_i}}{\binom{n-1}{\icin_i}} \bracket[\Bigg]{z\frac{\binom{m-1}{t}}{\binom{m}{t}} + (1-z)\frac{ \binom{n-m-1}{\icin_i - t}}{\binom{n-m}{\icin_i - t}}}.
		\end{equation}

		Recall the definition of the binomial coefficient:
		\begin{equation}
		\frac{\binom{x}{y}}{\binom{x-1}{y}} = \paren[\bigg]{\frac{x!}{y!(x-y)!}} \paren[\bigg]{\frac{(x-1)!}{y!(x-1-y)!}}^{-1} = \frac{x}{x-y},
		\end{equation}
		so we have:
		\begin{align}
		\frac{\prob[\big]{T_i = t \given Z_i = z}}{\prob[\big]{T_i = t}} - 1&=  \frac{\binom{n}{\icin_i}}{\binom{n-1}{\icin_i}} \bracket[\Bigg]{z\frac{\binom{m-1}{t}}{\binom{m}{t}} + (1-z)\frac{\binom{n-m-1}{\icin_i - t}}{\binom{n-m}{\icin_i - t}}} - 1
		\\
		&=  \frac{n}{n - \icin_i} \bracket[\Bigg]{z\frac{m - t}{m} + (1-z)\frac{(n-m) - (\icin_i - t)}{n-m}} - 1
		\\
		&=  \frac{n}{n - \icin_i} \bracket[\Bigg]{1 - z\frac{t}{m} - (1-z)\frac{\icin_i - t}{n-m} - \frac{n - \icin_i}{n}}
		\\
		&=  \frac{n}{n - \icin_i} \bracket[\Bigg]{\frac{\icin_i}{n} - z\frac{t}{m} - (1-z)\frac{\icin_i - t}{n-m}}
		\\
		&=  \frac{\icin_i}{\prob{Z_i = z}(n - \icin_i)} \bracket[\Bigg]{\prob{Z_i = z} - z\frac{t}{\icin_i} - (1-z)\frac{\icin_i - t}{\icin_i}},
		\end{align}
		where the last equality exploits that $n\prob{Z_i = z} = z m + (1-z) (n-m)$ under complete randomization. The expression within brackets is strictly between $-1$ and $1$, so it follows:
		\begin{equation}
		\abs[\bigg]{\frac{\prob[\big]{T = t \given Z_i = z}}{\prob[\big]{T = t}} - 1} \leq  \frac{\icin_i}{\prob{Z_i = z}(n - \icin_i)} \leq \frac{\icin_i}{\prob{Z_i = z}\min(m, n-m)},
		\end{equation}
		where $\icin_i \leq \min(m, n-m) \leq \max(m, n-m)$ gives the last inequality.
	\end{proof}

	\begin{applemma} \label{lem:complete-external-bound}
		Under complete randomization, for any $i,j\in\bfU$ such that $\di_{ij}=0$ and $4\icin_i\icin_j \leq \min(m, n-m)$:
		\begin{equation}
		\max_{\substack{\z_{i}\in \suppZint_i \\ \z_{j}\in \suppZint_j}} \;\abs[\bigg]{\frac{\prob{\Zint_{i} = \z_{i}, \Zint_{j} = \z_{j}}}{\prob{\Zint_{i} = \z_{i}}\prob{\Zint_j = \z_j}} - 1} \leq \frac{2\icin_i\icin_j}{\min(m, n-m)},
		\end{equation}
		where $\suppZint_i = \braces[\big]{\z\in\braces{0,1}^{n} : \prob{\Zint_{i} = \z} > 0}$ is the support of $\Zint_{i}$.
	\end{applemma}

	\begin{proof}
		Let $T_\textsc{min} = \min(m, n-m)$ be the size of the smallest treatment group. Let $t_i(\z) = \sum_{j=1}^n I_{j i} z_j$ be the element sum of $\z\in \suppZint_{i}$, and let $T_{i} = \sum_{j=1}^n I_{j i} Z_j$ be the element sum of $\Zint_{i}$ (i.e., the number of treated units interfering with $i$ including $i$ itself). We have $0 \leq T_{i} + T_{j} \leq \icin_i + \icin_j \leq T_\textsc{min}$ since $4\icin_i\icin_j \leq T_\textsc{min}$ by assumption and $\icin_i\geq 1$ by definition. From $\di_{ij}=0$ we have that $\Zint_{i}$ and $\Zint_{j}$ are disjoint, thus complete randomization gives us $\Zint_{i} \indep \Zint_{j} \,|\, T_{i},T_{j}$ and $\Zint_{i} \indep T_{j} \,|\, T_{i}$:
		\begin{align}
		&\frac{\prob{\Zint_{i} = \z_{i}, \Zint_{j} = \z_{j}}}{\prob{\Zint_{i} = \z_{i}}\prob{\Zint_j = \z_j}}
		\\
		&\quad= \frac{\prob[\big]{T_{i} = t_{i}(\z_{i}), T_{j} = t_{j}(\z_{j})}\prob[\big]{\Zint_{i} = \z_{i}, \Zint_{j} = \z_{j} \given T_{i} = t_{i}(\z_{i}), T_{j} = t_{j}(\z_{j})}}{\prob[\big]{T_{i} = t_{i}(\z_{i})} \prob[\big]{\Zint_{i} = \z_{i} \given T_{i} = t_{i}(\z_{i})} \prob[\big]{T_{j} = t_{j}(\z_{j})} \prob[\big]{\Zint_j = \z_j\given T_{j} = t_{j}(\z_{j})}}
		\\
		&\quad= \frac{\prob[\big]{T_{i} = t_{i}(\z_{i}), T_{j} = t_{j}(\z_{j})}\prob[\big]{\Zint_{i} = \z_{i} \given T_{i} = t_{i}(\z_{i})}\prob[\big]{\Zint_{j} = \z_{j} \given T_{j} = t_{j}(\z_{j})}}{\prob[\big]{T_{i} = t_{i}(\z_{i})} \prob[\big]{\Zint_{i} = \z_{i} \given T_{i} = t_{i}(\z_{i})} \prob[\big]{T_{j} = t_{j}(\z_{j})} \prob[\big]{\Zint_j = \z_j\given T_{j} = t_{j}(\z_{j})}}
		\\
		&\quad= \frac{\prob[\big]{T_{i} = t_{i}(\z_{i}) \given T_{j} = t_{j}(\z_{j})}}{\prob{T_{i} = t_{i}(\z_{i})}}.
		\end{align}

		Complete randomization implies that $T_{i}$ is hypergeometric:
		\begin{align}
		\prob[\big]{T_i = t_i} &= \binom{m}{t_i}\binom{n-m}{\icin_i - t_i}\binom{n}{\icin_i}^{-1},
		\\
		\prob[\big]{T_i = t_i \given T_j = t_j} &= \binom{m - t_j}{t_i}\binom{n - m - \paren[\big]{\icin_j - t_j}}{\icin_i - t_i}\binom{n - \icin_j}{\icin_i}^{-1},
		\end{align}
		where $t_i$ and $t_j$ are free variables. It follows:
		\begin{equation}
		\frac{\prob[\big]{T_i = t_i \given T_j = t_j}}{\prob[\big]{T_i = t_i}} =
		\frac{\binom{m - t_j}{t_i}}{\binom{m}{t_i}}
		\frac{\binom{n - m - \paren{\icin_j - t_j}}{\icin_i - t_i}}{\binom{n-m}{\icin_i - t_i}}
		\frac{\binom{n}{\icin_i}}{\binom{n - \icin_j}{\icin_i}}.
		\end{equation}
		Recall the definition of the binomial coefficient:
		\begin{equation}
		\frac{\binom{x}{y}}{\binom{x-z}{y}} = \paren[\bigg]{\frac{x!}{y!(x-y)!}} \paren[\bigg]{\frac{(x-z)!}{y!(x-z-y)!}}^{-1} = \prod_{a=0}^{y - 1} \frac{x - a}{x - z - a},
		\end{equation}
		so we have:
		\begin{align}
		\frac{\binom{m - t_j}{t_i}}{\binom{m}{t_i}} &= \prod_{a=0}^{t_i - 1} \frac{m - t_j - a}{m - a} = \prod_{a=0}^{t_i - 1} \paren[\bigg]{1 - \frac{t_j}{m - a}},
		\\
		\frac{\binom{n-m - (\icin_j - t_j)}{\icin_i - t_i}}{\binom{n-m}{\icin_i - t_i}} &= \prod_{a=0}^{\icin_i - t_i - 1} \frac{n-m - (\icin_j - t_j) - a}{n-m - a} = \prod_{a=0}^{\icin_i - t_i - 1} \paren[\bigg]{1 - \frac{\icin_j - t_j}{n-m - a}},
		\\
		\frac{\binom{n}{\icin_i}}{\binom{n - \icin_j}{\icin_i}} &= \prod_{a=0}^{\icin_i-1} \frac{n - a}{n - \icin_j - a} = \prod_{a=0}^{\icin_i-1} \paren[\bigg]{1 + \frac{\icin_j}{n - \icin_j - a}},
		\\
		\frac{\prob{T_i = t_i \given T_j = t_j}}{\prob{T_i = t_i}} &=
		\bracket[\Bigg]{\prod_{a=0}^{t_i - 1} \paren[\bigg]{1 - \frac{t_j}{m - a}}}
		\bracket[\Bigg]{\prod_{a=0}^{\icin_i - t_i - 1} \paren[\bigg]{1 - \frac{\icin_j - t_j}{n-m - a}}}
		\\
		&\qquad\qquad\times \bracket[\Bigg]{\prod_{a=0}^{\icin_i-1} \paren[\bigg]{1 + \frac{\icin_j}{n - \icin_j - a}}}.
		\end{align}

		First, note that:
		\begin{align}
		\bracket[\Bigg]{\prod_{a=0}^{t_i - 1} \paren[\bigg]{1 - \frac{t_j}{m - a}}} &\geq \paren[\bigg]{1 - \frac{\icin_j}{T_\textsc{min} - \icin_i}}^{t_i},
		\\
		\bracket[\Bigg]{\prod_{a=0}^{\icin_i - t_i - 1} \paren[\bigg]{1 - \frac{\icin_j - t_j}{n-m - a}}} &\geq \paren[\bigg]{1 - \frac{\icin_j}{T_\textsc{min} - \icin_i}}^{\icin_i - t_i},
		\\
		\bracket[\Bigg]{\prod_{a=0}^{\icin_i-1} \paren[\bigg]{1 + \frac{\icin_j}{n - \icin_j - a}}} &\geq 1,
		\end{align}
		so we have:
		\begin{align}
		\frac{\prob{T_i = t_i \given T_j = t_j}}{\prob{T_i = t_i}} & \geq \paren[\bigg]{1 - \frac{\icin_j}{T_\textsc{min} - \icin_i}}^{t_i} \paren[\bigg]{1 - \frac{\icin_j}{T_\textsc{min} - \icin_i}}^{\icin_i - t_i} = \paren[\bigg]{1 - \frac{\icin_j}{T_\textsc{min} - \icin_i}}^{\icin_i}
		\\
		&= \sum_{a=0}^{\icin_i} \binom{\icin_i}{a} \paren[\bigg]{\frac{-\icin_j}{T_\textsc{min} - \icin_i}}^{a} \geq 1 - \sum_{a=1}^{\icin_i} \binom{\icin_i}{a} \paren[\bigg]{\frac{\icin_j}{T_\textsc{min} - \icin_i}}^{a},
		\end{align}
		where the third equality is a binomial expansion. It follows that:
		\begin{equation}
		\frac{\prob{T_i = t_i \given T_j = t_j}}{\prob{T_i = t_i}} - 1  \geq -\sum_{a=1}^{\icin_i} \binom{\icin_i}{a} \paren[\bigg]{\frac{\icin_j}{T_\textsc{min} - \icin_i}}^{a}.
		\end{equation}

		Next, bound the factors as:
		\begin{align}
		\bracket[\Bigg]{\prod_{a=0}^{t_i - 1} \paren[\bigg]{1 - \frac{t_j}{m - a}}} &\leq 1,
		\\
		\bracket[\Bigg]{\prod_{a=0}^{\icin_i - t_i - 1} \paren[\bigg]{1 - \frac{\icin_j - t_j}{n-m - a}}} &\leq 1,
		\\
		\bracket[\Bigg]{\prod_{a=0}^{\icin_i-1} \paren[\bigg]{1 + \frac{\icin_j}{n - \icin_j - a}}} &\leq \paren[\bigg]{1 + \frac{\icin_j}{T_\textsc{min} - \icin_i}}^{\icin_i},
		\end{align}
		so we get:
		\begin{align}
		\frac{\prob{T_i = t_i \given T_j = t_j}}{\prob{T_i = t_i}} & \leq \paren[\bigg]{1 + \frac{\icin_j}{T_\textsc{min} - \icin_i}}^{\icin_i} = \sum_{a=0}^{\icin_i} \binom{\icin_i}{a} \paren[\bigg]{\frac{\icin_j}{T_\textsc{min} - \icin_i}}^{a}
		\\
		&= 1 + \sum_{a=1}^{\icin_i} \binom{\icin_i}{a} \paren[\bigg]{\frac{\icin_j}{T_\textsc{min} - \icin_i}}^{a}.
		\end{align}
		Taken together, it follows that:
		\begin{equation}
		-\sum_{a=1}^{\icin_i} \binom{\icin_i}{a} \paren[\bigg]{\frac{\icin_j}{T_\textsc{min} - \icin_i}}^{a} \leq \frac{\prob{T_i = t_i \given T_j = t_j}}{\prob{T_i = t_i}} - 1 \leq \sum_{a=1}^{\icin_i} \binom{\icin_i}{a} \paren[\bigg]{\frac{\icin_j}{T_\textsc{min} - \icin_i}}^{a} .
		\end{equation}
		The absolute value can thus be bounded as:
		\begin{align}
		\abs[\bigg]{\frac{\prob{T_{i} = t_{i} \given T_{j} = t_{j}}}{\prob{T_{i} = t_{i}}} - 1} &\leq \sum_{a=1}^{\icin_i} \binom{\icin_i}{a} \paren[\bigg]{\frac{\icin_j}{T_\textsc{min} - \icin_i}}^{a} \leq \sum_{a=1}^{\icin_i} \icin_i^{a} \paren[\bigg]{\frac{\icin_j}{T_\textsc{min} - \icin_i}}^{a}
		\\
		&= \sum_{a=1}^{\icin_i} \paren[\bigg]{\frac{\icin_i\icin_j}{T_\textsc{min} - \icin_i}}^{a} = \paren[\bigg]{\frac{\icin_i\icin_j}{T_\textsc{min} - \icin_i}} \sum_{a=0}^{\icin_i-1} \paren[\bigg]{\frac{\icin_i\icin_j}{T_\textsc{min} - \icin_i}}^{a}
		\\
		& \leq \paren[\bigg]{\frac{\icin_i\icin_j}{T_\textsc{min} - \icin_i}} \sum_{a=0}^{\infty} \paren[\bigg]{\frac{\icin_i\icin_j}{T_\textsc{min} - \icin_i}}^{a}.
		\end{align}
		By assumption, $4\icin_i\icin_j \leq T_\textsc{min}$, so $\icin_i\icin_j \leq T_\textsc{min} - \icin_i$, and the geometric series converge:
		\begin{align}
		\paren[\bigg]{\frac{\icin_i\icin_j}{T_\textsc{min} - \icin_i}} \sum_{a=0}^{\infty} \paren[\bigg]{\frac{\icin_i\icin_j}{T_\textsc{min} - \icin_i}}^{a} &= \frac{\icin_i\icin_j}{T_\textsc{min} - \icin_i} \times \frac{T_\textsc{min} - \icin_i}{T_\textsc{min} - \icin_i - \icin_i\icin_j}
		\\
		&= \frac{\icin_i\icin_j}{T_\textsc{min} - \icin_i - \icin_i\icin_j} \leq \frac{\icin_i\icin_j}{T_\textsc{min} - 2\icin_i\icin_j} \leq \frac{2\icin_i\icin_j}{T_\textsc{min}},
		\end{align}
		where the last inequality follows from $4\icin_i\icin_j \leq T_\textsc{min}$.
	\end{proof}

	\begin{applemma} \label{lem:complete-adse-eate-diff}
		Under complete randomization, $\abs{\eADSE - \eEATE} \leq 4k^3n^{-1}\icoutavg$.
	\end{applemma}

	\begin{proof}
		Let $T_\textsc{min} = \min(m, n-m)$ and $h_i = \indicator{\icin_i \leq T_\textsc{min}}$. Corollary \ref{coro:est-mu} gives:
		\begin{align}
		\abs{\eADSE - \eEATE} &= \abs[\big]{(\breve{\mu}_1 - \mu_1) - (\breve{\mu}_0 - \mu_0)}
		\\
		&= \abs[\bigg]{\frac{1}{n} \sum_{i=1}^{n} \sum_{z=0}^1 (2z -1)\paren[\Big]{\E{y_i(z; \Zint_{-i}) \given Z_i = z} - \E{y_i(z; \Zint_{-i})}}}
		\\
		&\leq \frac{1}{n} \sum_{i=1}^{n} \sum_{z=0}^1\abs[\Big]{\E{y_i(z; \Zint_{-i}) \given Z_i = z} - \E{y_i(z; \Zint_{-i})}}
		\\
		&= \frac{1}{n} \sum_{i=1}^{n} h_i\sum_{z=0}^1\abs[\Big]{\E{y_i(z; \Zint_{-i}) \given Z_i = z} - \E{y_i(z; \Zint_{-i})}}
		\\
		&\qquad + \frac{1}{n} \sum_{i=1}^{n} (1 - h_i)\sum_{z=0}^1\abs[\Big]{\E{y_i(z; \Zint_{-i}) \given Z_i = z} - \E{y_i(z; \Zint_{-i})}}. \label{eq:complete-bias-diff}
		\end{align}

		Let $\suppZint_{-i} = \braces[\big]{\z\in\braces{0,1}^{n-1} : \prob{\Zint_{-i} = \z} > 0}$ be the support of $\Zint_{-i}$ and consider the terms in the first sum in \eqref{eq:complete-bias-diff}:
		\begin{align}
		&h_i\sum_{z=0}^1\abs[\Big]{\E{y_i(z; \Zint_{-i}) \given Z_i = z} - \E{y_i(z; \Zint_{-i})}}
		\\
		&\qquad\qquad= h_i\sum_{z=0}^1 \abs[\Big]{\sum_{\z\in\suppZint_{-i}} \bracket[\big]{\prob{\Zint_{-i} = \z \given Z_i = z} - \prob{\Zint_{-i} = \z}}y_i(z; \z)}
		\\
		&\qquad\qquad\leq h_i\sum_{z=0}^1 \sum_{\z\in\suppZint_{-i}} \abs[\big]{\prob{\Zint_{-i} = \z \given Z_i = z} - \prob{\Zint_{-i} = \z}}\abs{y_i(z; \z)}
		\\
		&\qquad\qquad= h_i\sum_{z=0}^1 \sum_{\z\in\suppZint_{-i}} \prob{\Zint_{-i} = \z}\abs[\bigg]{\frac{\prob{\Zint_{-i} = \z \given Z_i = z}}{\prob{\Zint_{-i} = \z}} - 1}\abs{y_i(z; \z)}
		\\
		&\qquad\qquad\leq h_i\sum_{z=0}^1 \frac{\icin_i}{\prob{Z_i = z}T_\textsc{min}}\sum_{\z\in\suppZint_{-i}} \prob{\Zint_{-i} = \z}\abs{y_i(z; \z)}
		\tag*{Lemma \ref{lem:complete-internal-bound}}
		\\
		&\qquad\qquad= h_i\sum_{z=0}^1 \frac{\icin_i}{\prob{Z_i = z}T_\textsc{min}}\E[\big]{\abs{y_i(z; \Z_{-i})}}
		\\
		&\qquad\qquad\leq \frac{2h_ik^2\icin_i}{T_\textsc{min}}.
		\tag*{As.\ \ref{ass:probabilistic-assignment} and \ref{ass:bounded-expected-pos}}
		\end{align}
		We can apply Lemma \ref{lem:complete-internal-bound} since $h_{i} =1$ implies $\icin_i \leq T_\textsc{min}$.

		Note that $(1-h_{i}) \leq (1-h_{i}) T_\textsc{min}^{-1}\icin_i$ since $h_{i} =0$ implies $\icin_i \geq T_\textsc{min}$, so with Lemma \ref{lem:bounded-uncond-cond-po-diff}, it follows that:
		\begin{equation}
		(1 - h_i)\sum_{z=0}^1\abs[\Big]{\E{y_i(z; \Zint_{-i}) \given Z_i = z} - \E{y_i(z; \Zint_{-i})}} \leq 4(1 - h_i)k^2 \leq \frac{4(1 - h_i)k^2\icin_i}{T_\textsc{min}},
		\end{equation}
		and we get:
		\begin{equation}
		\abs{\eADSE - \eEATE} \leq \frac{1}{n} \sum_{i=1}^{n} \frac{4k^2\icin_i}{T_\textsc{min}} = 4k^2T_\textsc{min}^{-1}\icoutavg \leq 4k^3n^{-1}\icoutavg ,
		\end{equation}
		where the second to last equality follows from Lemma \ref{lem:icoutavg-handshake}, and the last inequality follows from Assumption \ref{ass:probabilistic-assignment}.
	\end{proof}

	\begin{applemma} \label{lem:complete-hmu-rates}
		Under complete randomization, $\hmu - \cmu = \bigOp[\big]{n^{-0.5}\davg^{0.5} + n^{-0.5}\icoutavg}$.
	\end{applemma}

	\begin{proof}
		First consider the variance of $\hmu$. Assumption \ref{ass:probabilistic-assignment} gives:
		\begin{equation}
		\Var{\hmu} = \Var[\bigg]{\frac{1}{n} \sum_{i=1}^{n} \frac{\indicator{Z_i = z}Y_i}{\prob{Z_i = z}}} \leq \frac{k^2}{n^2}\sum_{i=1}^{n} \sum_{j=1}^{n} \Cov[\big]{\indicator{Z_i = z}Y_i, \indicator{Z_j = z}Y_j}.
		\end{equation}
		Let $h_{ij} = \indicator{4\icin_i\icin_j \leq T_\textsc{min}}$ and recall $\di_{ij}$ and $\davg$ from Definition \ref{def:interference-dependence}:
		\begin{align}
		\Var{\hmu} &\leq \frac{k^2}{n^2}\sum_{i=1}^{n} \sum_{j=1}^{n} \bracket[\big]{\di_{ij} + (1-\di_{ij})h_{ij} + (1-\di_{ij})(1-h_{ij})}
		\\
		&\qquad\qquad\times\Cov[\big]{\indicator{Z_i = z}Y_i, \indicator{Z_j = z}Y_j}
		\\
		&\leq  k^4n^{-1}\davg + \frac{k^2}{n^2}\sum_{i=1}^{n} \sum_{j=1}^{n} \bracket[\big]{(1-\di_{ij})h_{ij} + (1-\di_{ij})(1-h_{ij})}
		\\
		&\qquad\qquad\qquad\qquad\qquad\qquad\qquad\times\Cov[\big]{\indicator{Z_i = z}Y_i, \indicator{Z_j = z}Y_j},
		\end{align}
		where the last inequality follows from Lemma \ref{lem:bounded-covariance}.

		Recall that $Y_i = y_i(\Z) = y_i\paren[\big]{\Zint_i}$ and let $\suppZint_i = \braces[\big]{\z\in\braces{0,1}^{n} : \prob{\Zint_{i} = \z} > 0}$ be the support of $\Zint_{i}$. Furthermore, let $T_\textsc{min} = \min(m, n-m)$, and consider the covariance of the terms with $(1-\di_{ij})h_{ij}=1$ in the double sum:
		\begin{align}
		&\Cov[\big]{\indicator{Z_i = z}y_i\paren[\big]{\Zint_i}, \indicator{Z_j = z}y_j\paren[\big]{\Zint_j}}
		\\
		&\qquad= \E[\big]{\indicator{Z_i = z}y_i\paren[\big]{\Zint_i}\indicator{Z_j = z}y_j\paren[\big]{\Zint_j}} - \E[\big]{\indicator{Z_i = z}y_i\paren[\big]{\Zint_i}}\E[\big]{\indicator{Z_j = z}y_j\paren[\big]{\Zint_j}}
		\\
		&\qquad= \bracket[\bigg]{\sum_{\z_{i}\in \suppZint_i} \sum_{\z_{j}\in \suppZint_j} \prob{\Zint_{i} = \z_{i}, \Zint_{j} = \z_{j}}\indicator{z_i = z}y_i\paren[\big]{\z_i}\indicator{z_j = z}y_j\paren[\big]{\z_j}}
		\\
		&\qquad\qquad\qquad - \bracket[\bigg]{\sum_{\z_{i}\in \suppZint_i} \prob{\Zint_{i} = \z_{i}} \indicator{z_i = z}y_i\paren[\big]{\z_i}} \bracket[\bigg]{\sum_{\z_{j}\in \suppZint_j}  \prob{\Zint_{j} = \z_{j}}\indicator{z_j = z}y_j\paren[\big]{\z_j}}
		\\
		&\qquad= \sum_{\z_{i}\in \suppZint_i} \sum_{\z_{j}\in \suppZint_j} \bracket[\bigg]{\prob{\Zint_{i} = \z_{i}, \Zint_{j} = \z_{j}} - \prob{\Zint_{i} = \z_{i}}\prob{\Zint_{j} = \z_{j}}}
		\\
		&\qquad\qquad\qquad\times\indicator{z_i = z}y_i\paren[\big]{\z_i}\indicator{z_j = z}y_j\paren[\big]{\z_j}
		\\
		&\qquad\leq \sum_{\z_{i}\in \suppZint_i} \sum_{\z_{j}\in \suppZint_j} \abs[\big]{\prob{\Zint_{i} = \z_{i}, \Zint_{j} = \z_{j}} - \prob{\Zint_{i} = \z_{i}}\prob{\Zint_{j} = \z_{j}}}\abs{y_i\paren{\z_{i}}}\abs{y_j\paren{\z_{j}}}
		\\
		&\qquad= \sum_{\z_{i}\in \suppZint_i} \sum_{\z_{j}\in \suppZint_j} \prob{\Zint_{i} = \z_{i}}\prob{\Zint_{j} = \z_{j}}\abs[\bigg]{\frac{\prob{\Zint_{i} = \z_{i}, \Zint_{j} = \z_{j}}}{\prob{\Zint_{i} = \z_{i}}\prob{\Zint_{j} = \z_{j}}} - 1}\abs{y_i\paren{\z_{i}}}\abs{y_j\paren{\z_{j}}}
		\\
		&\qquad\leq \frac{2\icin_i\icin_j}{T_\textsc{min}} \sum_{\z_{i}\in \suppZint_i} \sum_{\z_{j}\in \suppZint_j} \prob{\Zint_{i} = \z_{i}}\prob{\Zint_{j} = \z_{j}}\abs{y_i\paren{\z_{i}}}\abs{y_j\paren{\z_{j}}}
		\\
		&\qquad= \frac{2\icin_i\icin_j}{T_\textsc{min}} \E[\big]{\abs{y_i\paren{\Zint_{i}}}} \E[\big]{\abs{y_j\paren{\Zint_{j}}}} = \frac{2\icin_i\icin_j}{T_\textsc{min}} \E[\big]{\abs{Y_i}} \E[\big]{\abs{Y_j}} \leq \frac{2k^2\icin_i\icin_j}{T_\textsc{min}},
		\end{align}
		where Lemma \ref{lem:complete-external-bound} and Assumption \ref{ass:bounded-moments} were used in the last two inequalities.
		We can apply Lemma \ref{lem:complete-external-bound} since $(1-\di_{ij})h_{ij} =1$ implies $4\icin_i\icin_j \leq T_\textsc{min}$.

		Recall that $h_{ij} = \indicator{4\icin_i\icin_j \leq T_\textsc{min}}$, so we have $(1-h_{ij}) \leq (1-h_{ij})4T_\textsc{min}^{-1}\icin_i\icin_j$. With Lemma \ref{lem:bounded-covariance}, it follows that:
		\begin{align}
		(1-\di_{ij})(1-h_{ij})\Cov[\big]{\indicator{Z_i = z}y_i\paren[\big]{\Zint_i}, \indicator{Z_j = z}y_j\paren[\big]{\Zint_j}} &\leq (1-\di_{ij})(1-h_{ij})k^2
		\\
		&\leq (1-\di_{ij})(1-h_{ij})\frac{4k^2\icin_i\icin_j}{T_\textsc{min}}.
		\end{align}
		Taken together, we have:
		\begin{align}
		\Var{\hmu} &\leq  k^4n^{-1}\davg + \frac{4k^4}{n^2}\sum_{i=1}^{n} \sum_{j=1}^{n} \bracket[\big]{(1-\di_{ij})h_{ij} + (1-\di_{ij})(1-h_{ij})}\frac{\icin_i\icin_j}{T_\textsc{min}}
		\\
		&\leq  k^4n^{-1}\davg + \frac{4k^4}{n^2}\sum_{i=1}^{n} \sum_{j=1}^{n} \frac{\icin_i\icin_j}{T_\textsc{min}} = k^4n^{-1}\davg + \frac{4k^4}{T_\textsc{min}} \paren[\bigg]{\frac{1}{n}\sum_{i=1}^{n} \icin_i} \paren[\bigg]{\frac{1}{n}\sum_{j=1}^{n} \icin_j}
		\\
		&= k^4n^{-1}\davg + 4k^4T_\textsc{min}^{-1}\icoutavg^2 \leq k^4n^{-1}\davg + 4k^5n^{-1}\icoutavg^2,
		\end{align}
		where the second to last equality follows from Lemma \ref{lem:icoutavg-handshake} and the last inequality follows from Assumption \ref{ass:probabilistic-assignment}.

		Lemma \ref{lem:hmu-cmu-bias} provides the rate of convergence in the $L_2$-norm:
		\begin{equation}
			\sqrt{\E[\Big]{\paren[\big]{\hmu - \cmu}^2}} = \sqrt{\Var{\hmu}} = \bigO[\Big]{\bracket[\big]{n^{-1}\davg + n^{-1}\icoutavg^2}^{0.5}}.
		\end{equation}
		Concavity of the square root implies:
		\begin{equation}
		\bracket[\big]{n^{-1}\davg + n^{-1}\icoutavg^2}^{0.5} \leq n^{-0.5}\davg^{0.5} + n^{-0.5}\icoutavg,
		\end{equation}
		so, with Lemma \ref{lem:l2-to-probability}:
		\begin{equation}
		\hmu - \cmu = \bigOp[\big]{n^{-0.5}\davg^{0.5} + n^{-0.5}\icoutavg}. \tag*{\qedhere}
		\end{equation}
	\end{proof}

	\begin{applemma} \label{lem:complete-ht-ha-same}
		Under complete randomization, $\eHA = \eHT$.
	\end{applemma}

	\begin{proof}
		The number of treated units, $m = \sum_{i=1}^{n} Z_i$, is fixed under complete randomization. Furthermore, the marginal treatment probability is $\prob{Z_i = z} = zmn^{-1} + (1-z)(n-m)n^{-1}$ for all units. It follows that:
		\begin{equation}
		\hn = \sum_{i=1}^{n} \frac{\indicator{Z_i = z}}{\prob{Z_i = z}} = \frac{n}{zm + (1-z)(n-m)} \sum_{i=1}^{n} \indicator{Z_i = z} = n.
		\end{equation}
		Together with Corollary \ref{coro:est-mu}, we have:
		\begin{equation}
		\eHA = \frac{n}{\hat{n}_1}\hat{\mu}_1 - \frac{n}{\hat{n}_0}\hat{\mu}_0 = \hat{\mu}_1 - \hat{\mu}_0 = \eHT. \tag*{\qedhere}
		\end{equation}
	\end{proof}

	\begin{refprop}{\ref{prop:complete-rates}}
		With a complete randomization design under restricted interference (Assumption \ref{ass:restricted-interference}) and $\icoutavg=\littleO[\big]{n^{0.5}}$, the \HT\ and \HA\ estimators are consistent for \EATE\ and converge at the following rates:
		\begin{equation}
		\eHT - \eEATE = \bigOp[\big]{n^{-0.5}\davg^{0.5} + n^{-0.5}\icoutavg},
		\quad\text{and}\quad
		\eHA - \eEATE = \bigOp[\big]{n^{-0.5}\davg^{0.5} + n^{-0.5}\icoutavg}.
		\end{equation}
	\end{refprop}

	\begin{proof}
		Using Corollary \ref{coro:est-mu} and Lemma \ref{lem:complete-ht-ha-same}, decompose the errors of the estimators:
		\begin{align}
		\eHT - \eEATE = \eHA - \eEATE &= (\eHT - \eADSE) + (\eADSE - \eEATE)
		\\
		&= (\hat{\mu}_1 - \breve{\mu}_1) - (\hat{\mu}_0 - \breve{\mu}_0) + (\eADSE - \eEATE).
		\end{align}
		Lemmas \ref{lem:complete-adse-eate-diff} and \ref{lem:complete-hmu-rates} give the rates. Assumption \ref{ass:restricted-interference} and $\icoutavg=\littleO[\big]{n^{0.5}}$ give consistency.
	\end{proof}

	\section{Proof of Corollary \ref{coro:bernoulli-complete-root-n}}

	\begin{refcoro}{\ref{coro:bernoulli-complete-root-n}}
		With a Bernoulli or complete randomization design under bounded interference, $\davg=\bigO{1}$, the \HT\ and \HA\ estimators are root-$n$ consistent for \EATE.
	\end{refcoro}

	\begin{proof}
		Together with Lemma \ref{lem:interference-measure-inequalities}, $\davg=\bigO{1}$ implies that $\icoutavg=\bigO{1}$. The result then follows directly from Propositions \ref{prop:bernoulli-rates} and \ref{prop:complete-rates}.
	\end{proof}

	\section{Proof of Proposition \ref{prop:paired-rates}}

	\begin{applemma} \label{lem:paired-adse-eate-diff}
		Under paired randomization, $\abs{\eADSE - \eEATE} \leq 2k^2n^{-1}\Rsum$.
	\end{applemma}

	\begin{proof}
		Corollary \ref{coro:est-mu} gives:
		\begin{align}
		\abs{\eADSE - \eEATE} &= \abs[\big]{(\breve{\mu}_1 - \mu_1) - (\breve{\mu}_0 - \mu_0)}
		\\
		&= \abs[\bigg]{\frac{1}{n} \sum_{i=1}^{n} \sum_{z=0}^1 (2z -1)\paren[\Big]{\E{y_i(z; \Zint_{-i}) \given Z_i = z} - \E{y_i(z; \Zint_{-i})}}}
		\\
		&\leq \frac{1}{n} \sum_{i=1}^{n} \sum_{z=0}^1\abs[\Big]{\E{y_i(z; \Zint_{-i}) \given Z_i = z} - \E{y_i(z; \Zint_{-i})}}.
		\end{align}
		Recall $\Rsum$ from Definition \ref{def:within-pair-interference} and that $\rho(i)$ gives the pairing in the paired randomization design. Under this design, $Z_i = 1 - Z_j$ when $i= \rho(j)$ and $Z_i \indep Z_j$ when $i\neq \rho(j)$. As a consequence, $(1 - I_{\rho(i)i})\E{y_i(z; \Zint_{-i}) \given Z_i = z} = (1 - I_{\rho(i)i})\E{y_i(z; \Zint_{-i})}$, which gives:
		\begin{align}
		\abs{\eADSE - \eEATE} &\leq \frac{1}{n} \sum_{i=1}^{n} \bracket[\big]{I_{\rho(i)i} + (1 - I_{\rho(i)i})}\sum_{z=0}^1\abs[\Big]{\E{y_i(z; \Zint_{-i}) \given Z_i = z} - \E{y_i(z; \Zint_{-i})}}
		\\
		&= \frac{1}{n} \sum_{i=1}^{n} I_{\rho(i)i}\sum_{z=0}^1\abs[\Big]{\E{y_i(z; \Zint_{-i}) \given Z_i = z} - \E{y_i(z; \Zint_{-i})}}
		\\
		&\leq \frac{2k^2}{n} \sum_{i=1}^{n} I_{\rho(i)i} = 2k^2n^{-1}\Rsum.
		\end{align}
		where the second to last inequality follows from Lemma \ref{lem:bounded-uncond-cond-po-diff}.
	\end{proof}

	\begin{applemma} \label{lem:paired-hmu-rates}
		Under paired randomization, $\hmu - \cmu = \bigOp[\big]{n^{-0.5}\davg^{0.5} + n^{-0.5}\eavg^{0.5}}$.
	\end{applemma}

	\begin{proof}
		We have $\prob{Z_i = z}=0.5$ under paired randomization, so:
		\begin{align}
		\Var{\hmu} &= \frac{1}{n^2}\sum_{i=1}^{n} \sum_{j=1}^{n} \frac{\Cov[\big]{\indicator{Z_i = z}Y_i, \indicator{Z_j = z}Y_j}}{\prob{Z_i = z}\prob{Z_j = z}}
		\\
		&\leq \frac{4}{n^2}\sum_{i=1}^{n} \sum_{j=1}^{n} \Cov[\big]{\indicator{Z_i = z}Y_i, \indicator{Z_j = z}Y_j}.
		\end{align}
		Recall $\di_{ij}$ and $\davg$ from Definition \ref{def:interference-dependence} and $\ei_{ij}$ and $\eavg$ from Definition \ref{def:paired-interference-dependence}:
		\begin{align}
		\Var{\hmu} &\leq \frac{4}{n^2}\sum_{i=1}^{n} \sum_{j=1}^{n} \bracket[\big]{\di_{ij} + (1-\di_{ij})\ei_{ij} + (1-\di_{ij})(1-\ei_{ij})}
		\\
		&\qquad\qquad\times\Cov[\big]{\indicator{Z_i = z}Y_i, \indicator{Z_j = z}Y_j}
		\\
		&\leq  4k^2n^{-1}\davg + 4k^2n^{-1}\eavg
		\\
		&\qquad\qquad+ \frac{4}{n^2}\sum_{i=1}^{n} \sum_{j=1}^{n} (1-\di_{ij})(1-\ei_{ij}) \Cov[\big]{\indicator{Z_i = z}Y_i, \indicator{Z_j = z}Y_j}.
		\end{align}
		where the last inequality follows from Lemma \ref{lem:bounded-covariance}.

		$\Cov[\big]{\indicator{Z_i = z}Y_i, \indicator{Z_j = z}Y_j}$ can be non-zero in only three cases: when units $i$ and $j$ are paired (i.e., $\rho(i)=j$); when some unit $\ell$ is interfering with both $i$ and $j$ (i.e., $\di_{ij}=1$); or when some unit $\ell$ is interfering with $i$ and its paired unit $\rho(\ell)$ is interfering with $j$ (i.e., $\ei_{ij}=1$). Our definitions imply that $\di_{ij} + \ei_{ij}=1$ whenever $\rho(i)=j$. As a consequence, the covariance is zero when $\di_{ij}=\ei_{ij}=0$:
		\begin{equation}
		(1-\di_{ij})(1-\ei_{ij}) \Cov[\big]{\indicator{Z_i = z}Y_i, \indicator{Z_j = z}Y_j} = 0,
		\end{equation}
		and it follows that $\Var{\hmu} \leq 4k^2n^{-1}\davg + 4k^2n^{-1}\eavg$.

		Lemma \ref{lem:hmu-cmu-bias} and concavity of the square root provide the rate of convergence in the $L_2$-norm:
		\begin{equation}
			\sqrt{\E[\Big]{\paren[\big]{\hmu - \cmu}^2}} = \sqrt{\Var{\hmu}} = \bigO[\big]{n^{-0.5}\davg^{0.5} + n^{-0.5}\eavg^{0.5}},
		\end{equation}
		so Lemma \ref{lem:l2-to-probability} gives the rate of convergence in probability.
	\end{proof}

	\begin{applemma} \label{lem:paired-ht-ha-same}
		Under paired randomization, $\eHA = \eHT$.
	\end{applemma}

	\begin{proof}
		Under paired randomization $\prob{Z_i = z}=0.5$ and $\sum_{i=1}^{n} \indicator{Z_i = z} = n/2$, so it follows that:
		\begin{equation}
		\hn = \sum_{i=1}^{n} \frac{\indicator{Z_i = z}}{\prob{Z_i = z}} = 2 \sum_{i=1}^{n} \indicator{Z_i = z} = n.
		\end{equation}
		Corollary \ref{coro:est-mu} completes the proof:
		\begin{equation}
		\eHA = \frac{n}{\hat{n}_1}\hat{\mu}_1 - \frac{n}{\hat{n}_0}\hat{\mu}_0 = \hat{\mu}_1 - \hat{\mu}_0 = \eHT. \tag*{\qedhere}
		\end{equation}
	\end{proof}

	\begin{refprop}{\ref{prop:paired-rates}}
		With a paired randomization design under restricted interference, restricted pair-induced interference and pair separation (Assumptions \ref{ass:restricted-interference}, \ref{ass:restricted-pair-interference} and \ref{ass:pair-separation}), the \HT\ and \HA\ estimators are consistent for \EATE\ and converge at the following rates:
		\begin{align}
		\eHT - \eEATE &= \bigOp[\big]{n^{-0.5}\davg^{0.5} + n^{-0.5}\eavg^{0.5} + n^{-1}\Rsum},
		\\
		\eHA - \eEATE &= \bigOp[\big]{n^{-0.5}\davg^{0.5} + n^{-0.5}\eavg^{0.5} + n^{-1}\Rsum}.
		\end{align}
	\end{refprop}

	\begin{proof}
		Using Corollary \ref{coro:est-mu} and Lemma \ref{lem:paired-ht-ha-same}, decompose the errors of the estimators:
		\begin{align}
		\eHT - \eEATE = \eHA - \eEATE &= (\eHT - \eADSE) + (\eADSE - \eEATE)
		\\
		&= (\hat{\mu}_1 - \breve{\mu}_1) - (\hat{\mu}_0 - \breve{\mu}_0) + (\eADSE - \eEATE).
		\end{align}
		Lemmas \ref{lem:paired-adse-eate-diff} and \ref{lem:paired-hmu-rates} give the rates. Assumptions \ref{ass:restricted-interference}, \ref{ass:restricted-pair-interference} and \ref{ass:pair-separation} give consistency.
	\end{proof}

	\section{Proof of Proposition \ref{prop:arbitrary-rates-eate}}

	\begin{applemma} \label{lem:arbitrary-adse-eate-diff}
		Under arbitrary experimental designs, $\abs{\eADSE - \eEATE} \leq 6k^2 n^{-1}\mxint$.
	\end{applemma}

	\begin{proof}
		Corollary \ref{coro:est-mu} gives:
		\begin{align}
		\abs{\eADSE - \eEATE} &= \abs[\big]{(\breve{\mu}_1 - \mu_1) - (\breve{\mu}_0 - \mu_0)}
		\\
		&= \abs[\bigg]{\frac{1}{n} \sum_{i=1}^{n} \sum_{z=0}^1 (2z -1)\paren[\Big]{\E{y_i(z; \Zint_{-i}) \given Z_i = z} - \E{y_i(z; \Zint_{-i})}}}
		\\
		&\leq \frac{1}{n} \sum_{i=1}^{n} \sum_{z=0}^1\abs[\Big]{\E{y_i(z; \Zint_{-i}) \given Z_i = z} - \E{y_i(z; \Zint_{-i})}}.
		\end{align}
		Let $L_i(z) = \indicator{\abs{y_i(z; \Zint_{-i})} \geq \lambda_i}$ for some $\lambda_i>0$, and consider the terms in the outer sum:
		\begin{align}
		&\sum_{z=0}^1\abs[\Big]{\E{y_i(z; \Zint_{-i}) \given Z_i = z} - \E{y_i(z; \Zint_{-i})}}
		\\
		&\quad= \sum_{z=0}^1\abs[\Big]{\E{L_i(z)y_i(z; \Zint_{-i}) + \paren{1-L_i(z)}y_i(z; \Zint_{-i}) \given Z_i = z}
		\\
		&\qquad\qquad\qquad - \E{L_i(z)y_i(z; \Zint_{-i}) + \paren{1-L_i(z)}y_i(z; \Zint_{-i})}}
		\\
		&\quad\leq \sum_{z=0}^1\abs[\Big]{\E{\paren{1-L_i(z)}y_i(z; \Zint_{-i}) \given Z_i = z} - \E{\paren{1-L_i(z)}y_i(z; \Zint_{-i})}} \label{eq:internal-mixing-unit-term}
		\\
		&\quad\qquad\qquad + \sum_{z=0}^1\abs[\Big]{\E{L_i(z)y_i(z; \Zint_{-i}) \given Z_i = z} - \E{L_i(z)y_i(z; \Zint_{-i})}}.
		\end{align}

		Consider the first term in expression \eqref{eq:internal-mixing-unit-term}. Let $\suppZint_{-i} = \braces[\big]{\z\in\braces{0,1}^{n-1} : \prob{\Zint_{-i} = \z} > 0}$ be the support of $\Zint_{-i}$:
		\begin{align}
		&\sum_{z=0}^1\abs[\Big]{\E{\paren{1-L_i(z)}y_i(z; \Zint_{-i}) \given Z_i = z} - \E{\paren{1-L_i(z)}y_i(z; \Zint_{-i})}}
		\\
		&\qquad= \sum_{z=0}^1 \abs[\Big]{\sum_{\z\in\suppZint_{-i}} \paren[\Big]{\prob{\Zint_{-i} = \z \given Z_i = z} - \prob{\Zint_{-i} = \z}}\indicator{\abs{y_i(z; \z)} < \lambda_i}y_i(z; \z)}
		\\
		&\qquad= \sum_{z=0}^1 \abs[\Big]{\sum_{\z\in\suppZint_{-i}} \frac{1}{\prob{Z_i = z}}\paren[\Big]{\prob{Z_i = z, \Zint_{-i} = \z} - \prob{Z_i = z}\prob{\Zint_{-i} = \z}}
		\\
		&\qquad\qquad\qquad\qquad\times\indicator{\abs{y_i(z; \z)} < \lambda_i}y_i(z; \z)}
		\\
		&\qquad\leq \sum_{z=0}^1 \sum_{\z\in\suppZint_{-i}} \frac{1}{\prob{Z_i = z}}\abs[\Big]{\prob{Z_i = z, \Zint_{-i} = \z} - \prob{Z_i = z}\prob{\Zint_{-i} = \z}}
		\\
		&\qquad\qquad\qquad\qquad\times\indicator{\abs{y_i(z; \z)} < \lambda_i}\abs{y_i(z; \z)}
		\\
		&\qquad\leq k\lambda_i \sum_{z=0}^1 \sum_{\z\in\suppZint_{-i}} \abs[\big]{\prob{Z_i = z, \Zint_{-i} = \z} - \prob{Z_i = z}\prob{\Zint_{-i} = \z}},
		\end{align}
		where the last equality follows from $\indicator{\abs{y_i(z; \z)} < \lambda_i}\abs{y_i(z; \z)} \leq \lambda_i$ and Assumption \ref{ass:probabilistic-assignment}.

		Note that, for any $z\in\braces{0,1}$ and $\z\in\suppZint_{-i}$:
		\begin{align}
		&\abs[\big]{\prob{Z_i = z, \Zint_{-i} = \z} - \prob{Z_i = z}\prob{\Zint_{-i} = \z}}
		\\
		&\qquad = \abs[\big]{\bracket[\big]{\prob{\Zint_{-i} = \z} - \prob{Z_i = 1 - z, \Zint_{-i} = \z}} - \bracket[\big]{1 - \prob{Z_i = 1 - z}}\prob{\Zint_{-i} = \z}}
		\\
		&\qquad = \abs[\big]{\prob{Z_i = 1 - z, \Zint_{-i} = \z} - \prob{Z_i = 1 - z}\prob{\Zint_{-i} = \z}},
		\end{align}
		so:
		\begin{align}
		&k\lambda_i \sum_{z=0}^1 \sum_{\z\in\suppZint_{-i}} \abs[\big]{\prob{Z_i = z, \Zint_{-i} = \z} - \prob{Z_i = z}\prob{\Zint_{-i} = \z}}
		\\
		&\qquad\qquad\qquad= 2k\lambda_i \sum_{\z\in\suppZint_{-i}} \abs[\big]{\prob{Z_i = 1, \Zint_{-i} = \z} - \prob{Z_i = 1}\prob{\Zint_{-i} = \z}}.
		\end{align}
		Consider the following sets:
		\begin{align}
		\mathcal{Z}_{-i}^+ &= \braces[\big]{\z\in \suppZint_{-i}: \prob{Z_i = 1, \Zint_{-i} = \z} \geq \prob{Z_i = 1}\prob{\Zint_{-i} = \z}},
		\\
		\mathcal{Z}_{-i}^- &= \braces[\big]{\z\in \suppZint_{-i}: \prob{Z_i = 1, \Zint_{-i} = \z} \leq \prob{Z_i = 1}\prob{\Zint_{-i} = \z}},
		\end{align}
		so we can write:
		\begin{align}
		& 2k\lambda_i \sum_{\z\in\suppZint_{-i}} \abs[\big]{\prob{Z_i = 1, \Zint_{-i} = \z} - \prob{Z_i = 1}\prob{\Zint_{-i} = \z}}
		\\
		&\qquad\qquad = 2k\lambda_i \sum_{\z\in\mathcal{Z}_{-i}^+} \paren[\big]{\prob{Z_i = 1, \Zint_{-i} = \z} - \prob{Z_i = 1}\prob{\Zint_{-i} = \z}}
		\\
		&\qquad\qquad\qquad\qquad+ 2k\lambda_i \sum_{\z\in\mathcal{Z}_{-i}^-} \paren[\big]{\prob{Z_i = 1}\prob{\Zint_{-i} = \z} - \prob{Z_i = 1, \Zint_{-i} = \z}}
		\\
		&\qquad\qquad = 2k\lambda_i\paren[\big]{\prob{Z_i = 1, A^+} - \prob{Z_i = 1}\prob{A^+}}
		\\
		&\qquad\qquad\qquad\qquad + 2k\lambda_i \paren[\big]{\prob{Z_i = 1}\prob{A^-} - \prob{Z_i = 1, A^-}}
		\\
		&\qquad\qquad = 2k\lambda_i\abs[\big]{\prob{Z_i = 1, A^+} - \prob{Z_i = 1}\prob{A^+}}
		\\
		&\qquad\qquad\qquad\qquad+ 2k\lambda_i \abs[\big]{\prob{Z_i = 1, A^-} - \prob{Z_i = 1}\prob{A^-}},
		\end{align}
		where $A^+$ is the event $\Zint_{-i}\in \mathcal{Z}_{-i}^+$ and $A^-$ is the event $\Zint_{-i}\in \mathcal{Z}_{-i}^-$. $A^+$ and $A^-$ are both in the sub-sigma-algebra generated by $\Zint_{-i}$ (and so is the event $Z_i=1$ with respect to the algebra generated by $Z_i$). It follows:
		\begin{align}
		& 2k\lambda_i\abs[\big]{\prob{Z_i = 1, A^+} - \prob{Z_i = 1}\prob{A^+}} + 2k\lambda_i \abs[\big]{\prob{Z_i = 1, A^-} - \prob{Z_i = 1}\prob{A^-}}
		\\
		&\qquad \leq 4k\lambda_i\max_{\substack{A\in\sigma\paren{Z_i} \\ B \in\sigma\paren{\Zint_{-i}}}} \abs[\big]{\,\prob{A \cap B} - \prob{A}\prob{B}\,}
		\\
		&\qquad = 4k\lambda_i \alpha\paren[\big]{Z_i, \Zint_{-i}},
		\end{align}
		where $\alpha\paren[\big]{Z_i, \Zint_{-i}}$ is the alpha-mixing coefficient as defined in the main text.

		Consider the second term in expression \eqref{eq:internal-mixing-unit-term}:
		\begin{align}
		&\sum_{z=0}^1\abs[\Big]{\E[\big]{L_i(z)y_i(z; \Zint_{-i}) \given Z_i = z} - \E[\big]{L_i(z)y_i(z; \Zint_{-i})}}
		\\
		&\qquad= \sum_{z=0}^1\abs[\Big]{\E[\big]{L_i(z)y_i(z; \Zint_{-i}) \given Z_i = z} - \prob{Z_i = z}\E[\big]{L_i(z)y_i(z; \Zint_{-i}) \given Z_i = z}
		\\
		&\qquad\qquad\qquad - \prob{Z_i = 1 - z}\E[\big]{L_i(z)y_i(z; \Zint_{-i}) \given Z_i = 1 - z}}
		\\
		&\qquad= \sum_{z=0}^1\abs[\Big]{\prob{Z_i = 1 - z}\E[\big]{L_i(z)y_i(z; \Zint_{-i}) \given Z_i = z}
		\\
		&\qquad\qquad\qquad - \prob{Z_i = 1 - z}\E[\big]{L_i(z)y_i(z; \Zint_{-i}) \given Z_i = 1 - z}}
		\\
		&\qquad\leq \sum_{z=0}^1\paren[\Big]{\prob{Z_i = 1 - z}\E[\big]{L_i(z)\abs{y_i(z; \Zint_{-i})} \given Z_i = z}
		\\
		&\qquad\qquad\qquad + \prob{Z_i = 1 - z}\E[\big]{L_i(z)\abs{y_i(z; \Zint_{-i})} \given Z_i = 1 - z}}
		\\
		&\qquad\leq k\sum_{z=0}^1\paren[\Big]{\prob{Z_i = z}\E[\big]{L_i(z)\abs{y_i(z; \Zint_{-i})} \given Z_i = z}
		\\
		&\qquad\qquad\qquad + \prob{Z_i = 1 - z}\E[\big]{L_i(z)\abs{y_i(z; \Zint_{-i})} \given Z_i = 1 - z}}
		\\
		&\qquad= k\sum_{z=0}^1\E[\big]{L_i(z)\abs{y_i(z; \Zint_{-i})}},
		\end{align}
		where Assumption \ref{ass:probabilistic-assignment} was used in the last inequality. Note that $\lambda_i^{s-1}L_i(z)\abs{y_i(z; \Zint_{-i})} \leq \abs{y_i(z; \Zint_{-i})}^s$ with probability one for both $z\in\{0,1\}$ whenever $s \geq 1$. It follows that:
		\begin{equation}
		\E[\big]{L_i(z)\abs{y_i(z; \Zint_{-i})}} \leq \lambda_i^{1-s}\E[\big]{\abs{y_i(z; \Zint_{-i})}^s},
		\end{equation}
		and:
		\begin{align}
		&k\E[\big]{L_i(1)\abs{y_i(1; \Zint_{-i})}} + k\E[\big]{L_i(0)\abs{y_i(0; \Zint_{-i})}}
		\\
		&\qquad\qquad\leq k\lambda_i^{1-s}\E[\big]{\abs{y_i(1; \Zint_{-i})}^s} + k\lambda_i^{1-s}\E[\big]{\abs{y_i(0; \Zint_{-i})}^s}.
		\end{align}

		The two terms in expression \eqref{eq:internal-mixing-unit-term} taken together yield:
		\begin{align}
		&\sum_{z=0}^1\abs[\Big]{\E{y_i(z; \Zint_{-i}) \given Z_i = z} - \E{y_i(z; \Zint_{-i})}}
		\\
		&\qquad\qquad\leq 4k\lambda_i \alpha\paren[\big]{Z_i, \Zint_{-i}} + k\lambda_i^{1-s}\E[\big]{\abs{y_i(1; \Zint_{-i})}^s} + k\lambda_i^{1-s}\E[\big]{\abs{y_i(0; \Zint_{-i})}^s}.
		\end{align}
		The expression is trivially zero when $\Zint_{-i}$ and $Z_i$ are independent, so we will consider $\alpha\paren[\big]{Z_i, \Zint_{-i}}>0$. Recall that the inequality holds for any $\lambda_i>0$ and $s\geq1$. Set $s$ to the maximum value such that Assumption \ref{ass:bounded-expected-pos} holds and set $\lambda_i = \bracket[\big]{\alpha\paren[\big]{Z_i, \Zint_{-i}}}^{-1/s} \paren[\big]{y_i^\textsc{max}}^{1/s}$ where:
		\begin{equation}
		y_i^\textsc{max} = \max\paren[\Big]{\E[\big]{\abs{y_i(1; \Zint_{-i})}^s}, \E[\big]{\abs{y_i(0; \Zint_{-i})}^s}}.
		\end{equation}
		It follows that:
		\begin{align}
		&4k\lambda_i \alpha\paren[\big]{Z_i, \Zint_{-i}} + k\lambda_i^{1-s}\E[\big]{\abs{y_i(1; \Zint_{-i})}^s} + k\lambda_i^{1-s}\E[\big]{\abs{y_i(0; \Zint_{-i})}^s}
		\\
		&\qquad\qquad= 6k\lambda_i\alpha\paren[\big]{Z_i, \Zint_{-i}}\bracket[\Bigg]{\frac{4}{6}  + \frac{\E[\big]{\abs{y_i(1; \Zint_{-i})}^s}}{6\lambda_i^{s}\alpha\paren[\big]{Z_i, \Zint_{-i}}} + \frac{\E[\big]{\abs{y_i(0; \Zint_{-i})}^s}}{6\lambda_i^{s}\alpha\paren[\big]{Z_i, \Zint_{-i}}}}
		\\
		&\qquad\qquad= 6k\bracket[\big]{\alpha\paren[\big]{Z_i, \Zint_{-i}}}^{\frac{s-1}{s}} \paren[\big]{y_i^\textsc{max}}^{1/s}\bracket[\Bigg]{\frac{4}{6}  + \frac{\E[\big]{\abs{y_i(1; \Zint_{-i})}^s}}{6y_i^\textsc{max}} + \frac{\E[\big]{\abs{y_i(0; \Zint_{-i})}^s}}{6y_i^\textsc{max}}}
		\\
		&\qquad\qquad\leq 6k\bracket[\big]{\alpha\paren[\big]{Z_i, \Zint_{-i}}}^{\frac{s-1}{s}} \paren[\big]{y_i^\textsc{max}}^{1/s} \leq 6k^2\bracket[\big]{\alpha\paren[\big]{Z_i, \Zint_{-i}}}^{\frac{s-1}{s}},
		\end{align}
		where the last inequality follows from $y_i^\textsc{max} \leq k^s$ which is implied by Assumption \ref{ass:bounded-expected-pos}.

		Returning to the main sum, Definition \ref{def:mixing-coefficients} finally gives us:
		\begin{equation}
		\abs{\eADSE - \eEATE} \leq \frac{1}{n} \sum_{i=1}^{n} 6k^2\bracket[\big]{\alpha\paren[\big]{Z_i, \Zint_{-i}}}^{\frac{s-1}{s}} = 6k^2 n^{-1}\mxint. \tag*{\qedhere}
		\end{equation}
	\end{proof}

	\begin{refprop}{\ref{prop:arbitrary-rates-eate}}
		Under restricted interference, design mixing and design separation (Assumptions \ref{ass:restricted-interference}, \ref{ass:design-mixing} and \ref{ass:design-separation}), the \HT\ and \HA\ estimators are consistent for \EATE\ and converge at the following rates:
		\begin{align}
		\eHT - \eEATE &= \bigOp[\big]{n^{-0.5}\davg^{0.5} + n^{-0.5}\mxext^{0.5} + n^{-1}\mxint},
		\\
		\eHA - \eEATE &= \bigOp[\big]{n^{-0.5}\davg^{0.5} + n^{-0.5}\mxext^{0.5} + n^{-1}\mxint}.
		\end{align}
	\end{refprop}

	\begin{proof}
		Decompose the errors of the estimators:
		\begin{align}
		\eHT - \eEATE &= (\eHT - \eADSE) + (\eADSE - \eEATE),
		\\
		\eHA - \eEATE &= (\eHA - \eADSE) + (\eADSE - \eEATE).
		\end{align}
		Proposition \ref{prop:arbitrary-rates-adse} and Lemma \ref{lem:arbitrary-adse-eate-diff} give the rates. Assumptions \ref{ass:restricted-interference}, \ref{ass:design-mixing} and \ref{ass:design-separation} give consistency.
	\end{proof}

	\section{Proof of Proposition \ref{prop:arbitrary-rates-adse}}

	\begin{applemma} \label{lem:davydov-lemma-7}
		Let $X$ and $Y$ be two arbitrary random variables defined on the same probability space such that $\E[\big]{\abs{X}^a} \leq c^a$ and $\E[\big]{\abs{Y}^b} \leq c^b$ for some constants $a$, $b$ and $c$ where $a^{-1} + b^{-1} \leq 1$, we then have:
		\begin{equation}
		\abs{\Cov{X, Y}} \leq 10 \bracket[\big]{\alpha\paren[\big]{X, Y}}^{1 - 1/a - 1/b} \E[\big]{\abs{X}^a}^{1/a} \E[\big]{\abs{Y}^b}^{1/b},
		\end{equation}
		where $\alpha\paren{X, Y}$ is the alpha-mixing coefficient for $X$ and $Y$ as defined in the main text.
	\end{applemma}

	\begin{proof}
		The lemma is, apart from trivial changes, equivalent to Lemma 7 in \cite{Davydov1970}. Note that the condition on the exponents is erroneously written as $a^{-1} + b^{-1} = 1$ there.
	\end{proof}

	\begin{applemma} \label{lem:arbitrary-hmu-rates}
		Under arbitrary experimental designs, $\hmu - \cmu = \bigOp[\big]{n^{-0.5}\davg^{0.5} + n^{-0.5}\mxext^{0.5}}$.
	\end{applemma}

	\begin{proof}
		Assumption \ref{ass:probabilistic-assignment} gives:
		\begin{align}
		\Var{\hmu} &= \Var[\bigg]{\frac{1}{n} \sum_{i=1}^{n} \frac{\indicator{Z_i = z}Y_i}{\prob{Z_i = z}}} \leq \frac{k^2}{n^2}\sum_{i=1}^{n} \sum_{j=1}^{n} \Cov[\big]{\indicator{Z_i = z}Y_i, \indicator{Z_j = z}Y_j}
		\\
		&\leq \frac{k^2}{n^2}\sum_{i=1}^{n} \sum_{j=1}^{n} \bracket[\big]{\di_{ij} + (1-\di_{ij})}\Cov[\big]{\indicator{Z_i = z}Y_i, \indicator{Z_j = z}Y_j}
		\\
		&\leq k^4n^{-1}\davg + \frac{k^2}{n^2}\sum_{i=1}^{n} \sum_{j=1}^{n} (1-\di_{ij})\Cov[\big]{\indicator{Z_i = z}Y_i, \indicator{Z_j = z}Y_j},
		\end{align}
		where the last inequality follows from Definition \ref{def:interference-dependence} and Lemma \ref{lem:bounded-covariance}. Consider the terms in the double sum. We have $I_{ii}=1$ by definition, so the sub-sigma-algebra generated by $\Zint_i$ is finer than the algebra generated by $\indicator{Z_i = z}Y_i=\indicator{Z_i = z}y_i\paren[\big]{\Zint_i}$, from which it follows that $\alpha\paren[\big]{\indicator{Z_i = z}Y_i, \indicator{Z_j = z}Y_j}\leq \alpha\paren[\big]{\Zint_i, \Zint_j}$. Lemma \ref{lem:davydov-lemma-7} then gives:
		\begin{align}
		&\Cov[\big]{\indicator{Z_i = z}Y_i, \indicator{Z_j = z}Y_j}
		\\
		&\qquad\qquad \leq 10 \bracket[\big]{\alpha\paren[\big]{\Zint_i, \Zint_j}}^{\frac{q - 2}{q}} \E[\big]{\indicator{Z_i = z}\abs{Y_i}^q}^{1/q} \E[\big]{\indicator{Z_j = z}\abs{Y_j}^q}^{1/q}
		\\
		&\qquad\qquad \leq 10 \bracket[\big]{\alpha\paren[\big]{\Zint_i, \Zint_j}}^{\frac{q - 2}{q}} \E[\big]{\abs{Y_i}^q}^{1/q} \E[\big]{\abs{Y_j}^q}^{1/q}
		\\
		&\qquad\qquad \leq 10k^2 \bracket[\big]{\alpha\paren[\big]{\Zint_i, \Zint_j}}^{\frac{q - 2}{q}}, \tag*{Assumption \ref{ass:bounded-moments}}
		\end{align}
		where $q$ is the maximum value satisfying Assumption \ref{ass:bounded-moments}. Definition \ref{def:mixing-coefficients} yields:
		\begin{align}
		\Var{\hmu} & \leq k^4n^{-1}\davg + \frac{10k^4}{n^2}\sum_{i=1}^{n} \sum_{j\neq i} (1-\di_{ij}) \bracket[\big]{\alpha\paren[\big]{\Zint_i, \Zint_j}}^{\frac{q - 2}{q}} = k^4n^{-1}\davg + 10k^4n^{-1}\mxext.
		\end{align}

		Lemma \ref{lem:hmu-cmu-bias} and concavity of the square root provide the rate of convergence in the $L_2$-norm:
		\begin{equation}
			\sqrt{\E[\Big]{\paren[\big]{\hmu - \cmu}^2}} = \sqrt{\Var{\hmu}} = \bigO[\big]{n^{-0.5}\davg^{0.5} + n^{-0.5}\mxext^{0.5}},
		\end{equation}
		so Lemma \ref{lem:l2-to-probability} gives the rate of convergence in probability.
	\end{proof}

	\begin{applemma} \label{lem:arbitrary-hmu-hn-rates}
		Under arbitrary experimental designs satisfying Assumptions \ref{ass:restricted-interference} and \ref{ass:design-mixing}:
		\begin{equation}
		\frac{n}{\hn}\hmu - \cmu = \bigOp[\big]{n^{-0.5}\davg^{0.5} + n^{-0.5}\mxext^{0.5}}.
		\end{equation}
	\end{applemma}

	\begin{proof}
		First, consider the variance of $\hn / n$. Assumption \ref{ass:probabilistic-assignment} gives:
		\begin{align}
		\Var{\hn/n} &= \Var[\bigg]{\frac{1}{n} \sum_{i=1}^{n} \frac{\indicator{Z_i = z}}{\prob{Z_i = z}}} \leq \frac{k^2}{n^2}\sum_{i=1}^{n} \sum_{j=1}^{n} \Cov[\big]{\indicator{Z_i = z}, \indicator{Z_j = z}}
		\\
		&= \frac{k^2}{n^2}\sum_{i=1}^{n} \sum_{j=1}^{n} \bracket[\Big]{\prob{Z_i = z, Z_j = z} - \prob{Z_i = z}\prob{Z_j = z}}
		\\
		&\leq \frac{k^2}{n^2}\sum_{i=1}^{n} \sum_{j\neq i} \alpha\paren[\big]{Z_i, Z_j} \leq \frac{k^2\davg}{n} + \frac{k^2}{n^2}\sum_{i=1}^{n} \sum_{j=1}^{n} (1- \di_{ij})\alpha\paren[\big]{Z_i, Z_j},
		\end{align}
		where the second to last inequality follows from the definition of the alpha-mixing coefficient. The last inequality follows from $0 \leq \alpha\paren[\big]{Z_i, Z_j} \leq 0.25$.

		Note that $I_{ii}=1$ by definition. It then follows that the sub-sigma-algebra generated by $\Zint_i$ is finer than the algebra generated by $Z_i$, and we have $\alpha\paren[\big]{Z_i, Z_j} \leq \alpha\paren[\big]{\Zint_i, \Zint_j}$:
		\begin{align}
		\Var{\hn / n} &\leq \frac{k^2\davg}{n} + \frac{k^2}{n^2}\sum_{i=1}^{n} \sum_{j=1}^{n} (1- \di_{ij})\alpha\paren[\big]{\Zint_i, \Zint_j}
		\\
		&\leq \frac{k^2\davg}{n} + \frac{k^2}{n^2}\sum_{i=1}^{n} \sum_{j=1}^{n} (1- \di_{ij})\bracket[\big]{\alpha\paren[\big]{\Zint_i, \Zint_j}}^{\frac{q - 2}{q}} = k^2n^{-1}\davg + k^2n^{-1}\mxext,
		\end{align}
		where $q$ is the maximum value satisfying Assumption \ref{ass:bounded-moments}. The second to last inequality follows from $q\geq 2$ and $0 \leq \alpha\paren[\big]{\Zint_i, \Zint_j} \leq 0.25$.

		Lemma \ref{lem:hn-bias} and concavity of the square root give the rate of convergence in the $L_2$-norm of $\hn / n$ with respect to a sequence of ones:
		\begin{equation}
		\sqrt{\E[\Big]{\paren[\big]{\hn / n - 1}^2}} = \sqrt{\E[\Big]{\paren[\big]{\hn / n - \E{\hn / n}}^2}} = \sqrt{\Var{\hn / n}} = \bigO[\big]{n^{-0.5}\davg^{0.5} + n^{-0.5}\mxext^{0.5}}.
		\end{equation}
		Lemma \ref{lem:l2-to-probability} then gives $\hn / n - 1 = \bigOp[\big]{n^{-0.5}\davg^{0.5} + n^{-0.5}\mxext^{0.5}}$. Thus, under Assumptions \ref{ass:restricted-interference} and \ref{ass:design-mixing}, we have $\hn / n - 1 = \littleOp{1}$. We can apply Lemma \ref{lem:hmu-hn-rates} since $\hmu - \cmu = \littleOp{1}$ is implied by Lemma \ref{lem:arbitrary-hmu-rates} and Assumptions \ref{ass:restricted-interference} and \ref{ass:design-mixing}, which completes the proof.
	\end{proof}

	\begin{refprop}{\ref{prop:arbitrary-rates-adse}}
		Under restricted interference and design mixing (Assumptions \ref{ass:restricted-interference} and \ref{ass:design-mixing}), the \HT\ and \HA\ estimators are consistent for \ADSE\ and converge at the following rates:
		\begin{align}
		\eHT - \eADSE = \bigOp[\big]{n^{-0.5}\davg^{0.5} + n^{-0.5}\mxext^{0.5}},
		\\
		\eHA - \eADSE = \bigOp[\big]{n^{-0.5}\davg^{0.5} + n^{-0.5}\mxext^{0.5}}.
		\end{align}
	\end{refprop}

	\begin{proof}
		Using Corollary \ref{coro:est-mu}, decompose the errors of the estimators:
		\begin{align}
		\eHT - \eADSE &= (\hat{\mu}_1 - \breve{\mu}_1) - (\hat{\mu}_0 - \breve{\mu}_0),
		\\
		\eHA - \eADSE &= \paren[\bigg]{\frac{n}{\hat{n}_1}\hat{\mu}_1 - \breve{\mu}_1} - \paren[\bigg]{\frac{n}{\hat{n}_0}\hat{\mu}_0 - \breve{\mu}_0}.
		\end{align}
		Lemmas \ref{lem:arbitrary-hmu-rates} and \ref{lem:arbitrary-hmu-hn-rates} give the rates. Assumptions \ref{ass:restricted-interference} and \ref{ass:design-mixing} give consistency.
	\end{proof}

\section{Proof of Proposition~\ref{prop:var-est-bernoulli-limit}}

\begin{applemma}\label{lem:cov-bound-varsum}
	For two random variables, $X$ and $Y$, on the same probability space:
	\begin{equation}
		\Cov{X, Y} \leq \frac{\Var{X} + \Var{Y}}{2}.
	\end{equation}
\end{applemma}

\begin{proof}
Apply the Cauchy--Schwarz inequality followed by Young's inequality for products to get:
\begin{equation}
	\Cov{X, Y} \leq \sqrt{\Var{X}\Var{Y}} \leq \frac{\Var{X} + \Var{Y}}{2}. \tag*{\qedhere}
\end{equation}
\end{proof}

\begin{appcorollary}\label{coro:varsum}
	For two random variables, $X$ and $Y$, on the same probability space:
	\begin{equation}
		\Var{X + Y} \leq 2\Var{X} + 2\Var{Y}.
	\end{equation}
\end{appcorollary}

\begin{proof}
Write:
\begin{equation}
	\Var{X + Y} = \Var{X} + \Var{Y} + 2 \Cov{X, Y},
\end{equation}
and apply Lemma~\ref{lem:cov-bound-varsum}.
\end{proof}

\begin{applemma}\label{lem:var-est-convergence}
	Under a Bernoulli design and Assumption~\ref{ass:regularity-conditions} with $q \geq 4$:
	\begin{equation}
		n\eVarBer - a_n = \bigOp[\big]{n^{-0.5}\davg^{0.5}},
	\end{equation}
	where:
	\begin{equation}
		a_n = \frac{1}{n} \sum_{i=1}^n \frac{\E{Y_i^2 \given Z_i = 1}}{p_i} + \frac{1}{n} \sum_{i=1}^n \frac{\E{Y_i^2 \given Z_i = 0}}{1 - p_i}.
	\end{equation}
\end{applemma}

\begin{proof}
Consider the expectation of the normalized estimator:
\begin{equation}
	\E[\Big]{n \eVarBer} = \frac{1}{n} \sum_{i=1}^n \frac{\E{Z_i Y_i^2}}{p_i^2} + \frac{1}{n} \sum_{i=1}^n \frac{\E{(1 - Z_i) Y_i^2}}{(1 - p_i)^2}.
\end{equation}
Assumption~\ref{ass:regularity-conditions} ensures that the expectations:
\begin{equation}
	\E{Z_i Y_i^2} = p_i \E{Y_i^2 \given Z_i = 1}
	\qquad\text{and}\qquad
	\E{(1 - Z_i) Y_i^2} = (1 - p_i) \E{Y_i^2 \given Z_i = 0}
\end{equation}
exist, so:
\begin{equation}
	\E[\Big]{n \eVarBer} = a_n.
\end{equation}
Next, write the estimator as:
\begin{equation}
	n\eVarBer = \frac{1}{n} \sum_{z = 0}^1 \sum_{i=1}^n \frac{\indicator{Z_i = z} Y_i^2}{\bracket{\prob{Z_i = z}}^2}.
\end{equation}
Using Corollary~\ref{coro:varsum}:
\begin{equation}
	\Var[\Big]{n\eVarBer} \leq \frac{2}{n^2} \sum_{z = 0}^1 \Var[\Bigg]{\sum_{i=1}^n \frac{\indicator{Z_i = z} Y_i^2}{\bracket{\prob{Z_i = z}}^2}}.
\end{equation}
Given Assumption~\ref{ass:probabilistic-assignment}, for $z \in \braces{0, 1}$:
\begin{equation}
	\Var[\Bigg]{\sum_{i=1}^n \frac{\indicator{Z_i = z} Y_i^2}{\bracket{\prob{Z_i = z}}^2}} \leq k^4 \sum_{i=1}^n \sum_{j=1}^n \Cov[\big]{\indicator{Z_i = z} Y_i^2, \indicator{Z_j = z} Y_j^2}.
\end{equation}
By the same argument as in the proof of Lemma~\ref{lem:bernoulli-hmu-rates}:
\begin{equation}
(1 - \di_{ij})\Cov[\big]{\indicator{Z_i = z} Y_i^2, \indicator{Z_j = z} Y_j^2} = 0,
\end{equation}
so:
\begin{equation}
k^4 \sum_{i=1}^n \sum_{j=1}^n \Cov[\big]{\indicator{Z_i = z} Y_i^2, \indicator{Z_j = z} Y_j^2}
= k^4 \sum_{i=1}^n \sum_{j=1}^n \di_{ij} \Cov[\big]{\indicator{Z_i = z} Y_i^2, \indicator{Z_j = z} Y_j^2}.
\end{equation}
Lemma~\ref{lem:bounded-covariance} can be applied with $a = 2$ because $2a = 4 \leq q$, which gives:
\begin{equation}
k^4 \sum_{i=1}^n \sum_{j=1}^n \di_{ij} \Cov[\big]{\indicator{Z_i = z} Y_i^2, \indicator{Z_j = z} Y_j^2} \leq k^8 n \davg,
\end{equation}
so:
\begin{equation}
	\frac{2}{n^2} \sum_{z = 0}^1 \Var[\Bigg]{\sum_{i=1}^n \frac{\indicator{Z_i = z} Y_i^2}{\bracket{\prob{Z_i = z}}^2}}
	\leq \frac{4 k^8 \davg}{n}.
\end{equation}
In conclusion:
\begin{equation}
	\sqrt{\E[\Big]{\paren[\Big]{n\eVarBer - a_n}^2}} = \sqrt{\Var[\Big]{n\eVarBer}} = \bigO[\big]{n^{-0.5}\davg^{0.5}},
\end{equation}
which gives the rate of convergence in probability using Lemma \ref{lem:l2-to-probability}.
\end{proof}

\begin{appcorollary}\label{coro:mod-var-est-convergence}
	Let $\hat{V} = b_n \eVarBer$ for some positive sequence $(b_n)$.
	Under a Bernoulli design given that Assumptions~\ref{ass:regularity-conditions} and~\ref{ass:restricted-interference} hold with $q \geq 4$:
	\begin{equation}
		n b_n^{-1} \hat{V} - a_n = \littleOp{1},
	\end{equation}
	where:
	\begin{equation}
		a_n = \frac{1}{n} \sum_{i=1}^n \frac{\E{Y_i^2 \given Z_i = 1}}{p_i} + \frac{1}{n} \sum_{i=1}^n \frac{\E{Y_i^2 \given Z_i = 0}}{1 - p_i}.
	\end{equation}
\end{appcorollary}

\begin{proof}
By definition $n b_n^{-1} \hat{V} = n \eVarBer$.
Apply Lemma~\ref{lem:var-est-convergence} and Assumption~\ref{ass:restricted-interference} to complete the proof.
\end{proof}

\begin{appcorollary}\label{coro:mod-var-est-fast-convergence}
	Let $\hat{V} = b_n \eVarBer$ for some positive sequence $(b_n)$.
	Under a Bernoulli design given that Assumption~\ref{ass:regularity-conditions} holds with $q \geq 4$ and $b_n = \littleO[\big]{n^{0.5} \davg^{0.5}}$:
	\begin{equation}
		n b_n^{-1} \hat{V} - a_n = \littleOp[\big]{b_n^{-1} \davg},
	\end{equation}
	where:
	\begin{equation}
		a_n = \frac{1}{n} \sum_{i=1}^n \frac{\E{Y_i^2 \given Z_i = 1}}{p_i} + \frac{1}{n} \sum_{i=1}^n \frac{\E{Y_i^2 \given Z_i = 0}}{1 - p_i}.
	\end{equation}
\end{appcorollary}

\begin{proof}
By definition $n b_n^{-1} \hat{V} = n \eVarBer$.
Apply Lemma~\ref{lem:var-est-convergence} to get:
\begin{equation}
	n b_n^{-1} \hat{V} - a_n = \bigOp[\big]{n^{-0.5}\davg^{0.5}}.
\end{equation}
It follows that:
\begin{equation}
\frac{b_n}{\davg} \bracket[\big]{n b_n^{-1} \hat{V} - a_n} = \bigOp[\big]{n^{-0.5}\davg^{-0.5} b_n} = \littleOp{1},
\end{equation}
because $b_n = \littleO[\big]{n^{0.5} \davg^{0.5}}$.
\end{proof}

\begin{applemma}\label{lem:unit-term-variance}
	Under a Bernoulli design:
	\begin{equation}
	\Var[\Bigg]{\frac{Z_i Y_i}{p_i} - \frac{(1 - Z_i) Y_i}{1 - p_i}} = \frac{\E{Y_i^2 \given Z_i = 1}}{p_i} + \frac{\E{Y_i^2 \given Z_i = 0}}{1 - p_i} - \paren[\big]{\E{\tau_i(\Z_{-i})}}^2.
	\end{equation}
\end{applemma}

\begin{proof}
Consider:
\begin{equation}
\Var[\Bigg]{\frac{Z_i Y_i}{p_i} - \frac{(1 - Z_i) Y_i}{1 - p_i}}
= \frac{\Var{Z_i Y_i}}{p_i^2}  + \frac{\Var{(1 - Z_i) Y_i}}{(1 - p_i)^2} - \frac{2 \Cov{Z_i Y_i, (1 - Z_i) Y_i}}{p_i (1 - p_i)}.
\end{equation}
Continue with each term separately:
\begin{align}
\Var{Z_i Y_i} &= \E{Z_i Y_i^2} - \paren[\big]{\E{Z_i Y_i}}^2
\\
				  &= p_i \E{Y_i^2 \given Z_i = 1} - p_i^2 \paren[\big]{\E{Y_i \given Z_i = 1}}^2,
				  \\[1em]
\Var{(1 - Z_i) Y_i} & = (1 - p_i) \E{Y_i^2 \given Z_i = 0} - (1 - p_i)^2 \paren[\big]{\E{Y_i \given Z_i = 0}}^2,
\\[1em]
\Cov{Z_i Y_i, (1 - Z_i) Y_i} &= \E{Z_i (1 - Z_i) Y_i^2} - \E{Z_i Y_i} \E{(1 - Z_i) Y_i}
\\
									  &= - p_i (1 - p_i) \E{Y_i \given Z_i = 1} \E{Y_i \given Z_i = 0},
\end{align}
where the last equality follows from that $Z_i (1 - Z_i) = 0$ with probability one.
Taken together:
\begin{align}
&\frac{\Var{Z_i Y_i}}{p_i^2}  + \frac{\Var{(1 - Z_i) Y_i}}{(1 - p_i)^2} - \frac{2 \Cov{Z_i Y_i, (1 - Z_i) Y_i}}{p_i (1 - p_i)}
\\[1em]
&\qquad \qquad = \frac{\E{Y_i^2 \given Z_i = 1}}{p_i} - \paren[\big]{\E{Y_i \given Z_i = 1}}^2
+ \frac{\E{Y_i^2 \given Z_i = 0}}{1 - p_i} - \paren[\big]{\E{Y_i \given Z_i = 0}}^2
\\
&\qquad \qquad \qquad \qquad + 2 \E{Y_i \given Z_i = 1} \E{Y_i \given Z_i = 0}
\\[1em]
&\qquad \qquad = \frac{\E{Y_i^2 \given Z_i = 1}}{p_i} + \frac{\E{Y_i^2 \given Z_i = 0}}{1 - p_i}
- \paren[\big]{\E{Y_i \given Z_i = 1} - \E{Y_i \given Z_i = 0}}^2.
\end{align}
Because of the Bernoulli design:
\begin{equation}
	\E{Y_i \given Z_i = z} = \E{y_i(z; \Z_{-i}) \given Z_i = z} = \E{y_i(z; \Z_{-i})},
\end{equation}
so:
\begin{equation}
	\paren[\big]{\E{Y_i \given Z_i = 1} - \E{Y_i \given Z_i = 0}}^2 = \paren[\big]{\E{\tau_i(\Z_{-i})}}^2. \tag*{\qedhere}
\end{equation}
\end{proof}

\begin{refprop}{\ref{prop:var-est-bernoulli-limit}}
	Under a Bernoulli design and Assumption~\ref{ass:regularity-conditions} with $q \geq 4$:
	\begin{equation}
		n\davg^{-1} \bracket[\Big]{\eVarBer - \Var{\eHT}}
		\parrow
		\frac{\eMSTE}{\davg} - B_1 - B_2,
	\end{equation}
	where:
	\begin{align}
		B_1 &= \frac{1}{n\davg} \sum_{i=1}^{n} \sum_{j\neq i} \paren[\Big]{\breve{\sote}_{ij}\breve{\sote}_{ji} + 2\breve{Y}_j \bracket[\big]{\sote_{ij}(1) - \sote_{ij}(0)}},
		\\
		B_2 &= \frac{1}{n\davg} \sum_{i=1}^{n} \sum_{j\neq i} \sum_{a = 0}^1 \sum_{b = 0}^1 \paren{-1}^{a + b} \Cov[\big]{Y_i , Y_j \given Z_i = a, Z_j = b}.
	\end{align}
\end{refprop}

\begin{proof}
Write the variance as:
\begin{equation}
	n\Var{\eHT} = \frac{1}{n} \sum_{i = 1}^n \Var[\big]{h_i^{-1}Y_i} + \frac{1}{n} \sum_{i = 1}^n \sum_{j\neq i} \Cov[\big]{h_i^{-1}Y_i, h_j^{-1}Y_j},
\end{equation}
where $h_i = Z_i + p_i - 1$, so that:
\begin{equation}
	\Var[\big]{h_i^{-1}Y_i} = \Var[\Bigg]{\frac{Z_i Y_i}{p_i} - \frac{(1 - Z_i) Y_i}{1 - p_i}}.
\end{equation}
By Lemma~\ref{lem:unit-term-variance}:
\begin{equation}
	\frac{1}{n} \sum_{i = 1}^n \Var[\big]{h_i^{-1}Y_i} = a_n - \eMSTE,
\end{equation}
where:
\begin{equation}
	a_n = \frac{1}{n} \sum_{i=1}^n \frac{\E{Y_i^2 \given Z_i = 1}}{p_i} + \frac{1}{n} \sum_{i=1}^n \frac{\E{Y_i^2 \given Z_i = 0}}{1 - p_i}
	\qquad\text{and}\qquad
	\eMSTE = \frac{1}{n} \sum_{i=1}^{n} \paren[\big]{\E{\tau_i(\Z_{-i})}}^2.
\end{equation}
Next, for $i \neq j$:
\begin{multline}
\Cov[\big]{h_i^{-1}Y_i, h_j^{-1}Y_j} = \Cov[\bigg]{\frac{Z_i Y_i}{p_i} - \frac{(1 - Z_i) Y_i}{1 - p_i}, \frac{Z_j Y_j}{p_j} - \frac{(1 - Z_j) Y_j}{1 - p_j}}
\\
= \sum_{a = 0}^1 \sum_{b = 0}^1 \frac{\paren{-1}^{a + b} \Cov[\big]{\indicator{Z_i = a} Y_i, \indicator{Z_j = b} Y_j}}{\prob{Z_i = a} \prob{Z_j = b}}.
\end{multline}
For any $a, b \in \braces{0, 1}$:
\begin{align}
	&\frac{\Cov[\big]{\indicator{Z_i = a} Y_i, \indicator{Z_j = b} Y_j}}{\prob{Z_i = a} \prob{Z_j = b}}
	\\[1em]
	&\qquad\qquad=
	\frac{\E{\indicator{Z_i = a} \indicator{Z_j = b} Y_i Y_j} - \E{\indicator{Z_i = a} Y_i} \E{\indicator{Z_j = b} Y_j}}{\prob{Z_i = a} \prob{Z_j = b}}
	\\[1em]
	&\qquad\qquad=
	\frac{\prob{Z_i = a, Z_j = b}}{\prob{Z_i = a} \prob{Z_j = b}} \E{Y_i Y_j \given Z_i = a, Z_j = b} - \E{Y_i \given Z_i = a} \E{Y_j \given Z_j = b}
	\\[1em]
	&\qquad\qquad=
	\E{Y_i Y_j \given Z_i = a, Z_j = b} - \E{Y_i \given Z_i = a} \E{Y_j \given Z_j = b},
\end{align}
where the final equality follows from the Bernoulli design.
Add and substract:
\begin{equation}
	\E{Y_i \given Z_i = a, Z_j = b}\E{Y_j \given Z_i = a, Z_j = b},
\end{equation}
to get:
\begin{align}
	&\E{Y_i Y_j \given Z_i = a, Z_j = b} - \E{Y_i \given Z_i = a} \E{Y_j \given Z_j = b}
	\\
	&\qquad\qquad=
	\Cov{Y_i, Y_j \given Z_i = a, Z_j = b}
	+ \E{Y_i \given Z_i = a, Z_j = b}\E{Y_j \given Z_i = a, Z_j = b}
	\\
	&\qquad\qquad\qquad\qquad - \E{Y_i \given Z_i = a} \E{Y_j \given Z_j = b}.
\end{align}
Add and substract:
\begin{equation}
	\E{Y_i \given Z_i = a} \E{Y_j \given Z_i = a, Z_j = b},
\end{equation}
to get:
\begin{align}
	&\E{Y_i \given Z_i = a, Z_j = b}\E{Y_j \given Z_i = a, Z_j = b} - \E{Y_i \given Z_i = a} \E{Y_j \given Z_j = b}
	\\
	&\qquad\qquad\qquad\qquad= \E{Y_j \given Z_i = a, Z_j = b} \paren[\Big]{\E{Y_i \given Z_i = a, Z_j = b} - \E{Y_i \given Z_i = a}}
	\\
	&\qquad\qquad\qquad\qquad\qquad\qquad+ \E{Y_i \given Z_i = a} \paren[\Big]{\E{Y_j \given Z_i = a, Z_j = b} - \E{Y_j \given Z_j = b}}
\end{align}
Note that:
\begin{equation}
	\E{Y_i \given Z_i = a} = p_j \E{Y_i \given Z_i = a, Z_j = 1} + (1 - p_j) \E{Y_i \given Z_i = a, Z_j = 0},
\end{equation}
where $\prob{Z_j = b \given Z_i = a} = \prob{Z_j = b}$ follows from the Bernoulli design.
It follows that:
\begin{multline}
	\E{Y_i \given Z_i = a, Z_j = b} - \E{Y_i \given Z_i = a}
	\\
	=
	(b - p_j) \paren[\Big]{\E{Y_i \given Z_i = a, Z_j = 1} - \E{Y_i \given Z_i = a, Z_j = 0}}.
\end{multline}
Because of the Bernoulli design:
\begin{equation}
	\E{Y_i \given Z_i = a, Z_j = b} = \E{y_{ij}(a, b; \Z_{-ij}) \given Z_i = a, Z_j = b}
	= \E{y_{ij}(a, b; \Z_{-ij})},
\end{equation}
so:
\begin{equation}
	(b - p_j) \paren[\Big]{\E{Y_i \given Z_i = a, Z_j = 1} - \E{Y_i \given Z_i = a, Z_j = 0}}
	=
	(b - p_j) \sote_{ij}(a).
\end{equation}
Similarly:
\begin{equation}
	\E{Y_j \given Z_i = a, Z_j = b} - \E{Y_j \given Z_j = b} = (a - p_i) \sote_{ji}(b),
\end{equation}
hence:
\begin{align}
	&\E{Y_j \given Z_i = a, Z_j = b} \paren[\Big]{\E{Y_i \given Z_i = a, Z_j = b} - \E{Y_i \given Z_i = a}}
	\\
	&\qquad\qquad\qquad\qquad+ \E{Y_i \given Z_i = a} \paren[\Big]{\E{Y_j \given Z_i = a, Z_j = b} - \E{Y_j \given Z_j = b}}
	\\
	&\qquad\qquad= (b - p_j) \E{Y_j \given Z_i = a, Z_j = b} \sote_{ij}(a)
	+ (a - p_i) \E{Y_i \given Z_i = a} \sote_{ji}(b).
\end{align}
Add and substract:
\begin{equation}
	(b - p_j) \E{Y_j \given Z_j = b} \sote_{ij}(a),
\end{equation}
to get:
\begin{align}
	&(b - p_j) \E{Y_j \given Z_i = a, Z_j = b} \sote_{ij}(a)
	+ (a - p_i) \E{Y_i \given Z_i = a} \sote_{ji}(b)
	\\
	&\qquad\qquad \qquad\qquad = (a - p_i) (b - p_j) \sote_{ij}(a) \sote_{ji}(b)
	+ (b - p_j) \E{Y_j \given Z_j = b} \sote_{ij}(a)
	\\
	&\qquad\qquad\qquad \qquad\qquad\qquad + (a - p_i) \E{Y_i \given Z_i = a} \sote_{ji}(b).
\end{align}
Note that:
\begin{equation}
	\sum_{a = 0}^1 \paren{-1}^{a} (a - p_i) \sote_{ij}(a) = - (1 - p_i) \sote_{ij}(1) - p_i \sote_{ij}(0) = - \breve{\sote}_{ij},
\end{equation}
where $\breve{\sote}_{ij}$ was defined in Section~\ref{sec:conventional-var-est}, so:
\begin{equation}
	\sum_{a = 0}^1 \sum_{b = 0}^1 \paren{-1}^{a + b} (a - p_i) (b - p_j) \sote_{ij}(a) \sote_{ji}(b)
	= \breve{\sote}_{ij}\breve{\sote}_{ji}.
\end{equation}
Similarly:
\begin{multline}
\sum_{z = 0}^1 \paren{-1}^{z} \sote_{ij}(z)
= - \bracket[\big]{\sote_{ij}(1) - \sote_{ij}(0)},
\qquad\text{and}
\\
\sum_{z = 0}^1 \paren{-1}^{z} (z - p_i) \E{Y_i \given Z_i = z}
= - (1 - p_i) \E{Y_i \given Z_i = 1} - p_i \E{Y_i \given Z_i = 0}
= - \breve{Y}_i,
\end{multline}
so:
\begin{align}
	\sum_{a = 0}^1 \sum_{b = 0}^1 \paren{-1}^{a + b} (b - p_j) \E{Y_j \given Z_j = b} \sote_{ij}(a)
	&= \breve{Y}_j \bracket[\big]{\sote_{ij}(1) - \sote_{ij}(0)},
	\\
	\sum_{a = 0}^1 \sum_{b = 0}^1 \paren{-1}^{a + b} (a - p_i) \E{Y_i \given Z_i = a} \sote_{ji}(b)
	&= \breve{Y}_i \bracket[\big]{\sote_{ji}(1) - \sote_{ji}(0)}
\end{align}
Because the summations over $i$ and $j$ are symmetric, summing over all off-diagonals:
\begin{equation}
	\frac{1}{n} \sum_{i = 1}^n \sum_{j\neq i} \breve{Y}_j \bracket[\big]{\sote_{ij}(1) - \sote_{ij}(0)}
	= \frac{1}{n} \sum_{i = 1}^n \sum_{j\neq i} \breve{Y}_i \bracket[\big]{\sote_{ji}(1) - \sote_{ji}(0)}.
\end{equation}

Taken together:
\begin{align}
	\frac{1}{n} \sum_{i = 1}^n \sum_{j\neq i} \Cov[\big]{h_i^{-1}Y_i, h_j^{-1}Y_j}
	&=
	\frac{1}{n} \sum_{i = 1}^n \sum_{j\neq i} \breve{\sote}_{ij}\breve{\sote}_{ji}
	 +
	\frac{2}{n} \sum_{i = 1}^n \sum_{j\neq i} \breve{Y}_j \bracket[\big]{\sote_{ij}(1) - \sote_{ij}(0)}
	\\
	&\qquad \qquad + \frac{1}{n} \sum_{i = 1}^n \sum_{j\neq i} \sum_{a = 0}^1 \sum_{b = 0}^1 \paren{-1}^{a + b} \Cov{Y_i, Y_j \given Z_i = a, Z_j = b}.
\end{align}
The normalized variance can therefore be written:
\begin{equation}
	n\davg^{-1} \Var{\eHT} = \frac{a_n}{\davg} - \frac{\eMSTE}{\davg} + B_1 + B_2,
\end{equation}
where $B_1$ and $B_2$ are defined in the proposition.
Together with the variance estimator:
\begin{equation}
	n\davg^{-1} \bracket[\Big]{\eVarBer - \Var{\eHT}} = \frac{n\eVarBer - a_n}{\davg} + \frac{\eMSTE}{\davg} - B_1 - B_2.
\end{equation}
By Lemma~\ref{lem:var-est-convergence}, $n\eVarBer - a_n = \bigOp[\big]{n^{-0.5}\davg^{0.5}}$, so:
\begin{equation}
	\frac{n\eVarBer - a_n}{\davg} = \bigOp[\big]{n^{-0.5}\davg^{-0.5}} = \littleOp{1},
\end{equation}
for any sequence $(\davg)$.
\end{proof}

\begin{appcorollary}\label{lem:var-est-no-interference-limit}
	Under a Bernoulli design, no interference and Assumption~\ref{ass:regularity-conditions} with $q \geq 4$
	\begin{equation}
		n \bracket[\Big]{\eVarBer - \Var{\eHT}} \parrow \eMSTE.
	\end{equation}
\end{appcorollary}

\begin{proof}
	From Proposition~\ref{prop:var-est-bernoulli-limit}:
	\begin{equation}
		n\davg^{-1} \bracket[\Big]{\eVarBer - \Var{\eHT}}
		\parrow
		\frac{\eMSTE}{\davg} - B_1 - B_2.
	\end{equation}
	Under no interference:
	\begin{equation}
		\Cov[\big]{Y_i , Y_j \given Z_i = a, Z_j = b} = 0
		\qquad\text{and}\qquad
		\sote_{ij}(a) = 0,
	\end{equation}
	for any $i \neq j$ and $a, b \in \{0, 1\}$, so $B_1 = B_2 = 0$.
	Furthermore, $\davg = 1$.
\end{proof}

\section{Proofs of Propositions~\ref{prop:var-est-avg-conservative},~\ref{prop:var-est-max-conservative} and~\ref{prop:var-est-eig-conservative}}

\begin{applemma}\label{lem:fast-convergence-conservative-route}
	For a variance estimator $\hat{V}$ and two positive sequences $(a_n)$ and $(b_n)$, if:
	\begin{equation}
		\liminf_{n \to \infty} b_n \davg^{-1} \bracket[\big]{a_n - n b_n^{-1} \Var{\eHT}} \geq 0
		\qquad \text{and} \qquad
		n b_n^{-1} \hat{V} - a_n = \littleOp[\big]{b_n^{-1} \davg},
	\end{equation}
	then $\hat{V}$ is asymptotically conservative.
\end{applemma}

\begin{proof}
Starting with the expression for asymptotically conservative in Definition~\ref{def:asymptotically-conservative}, for some $\varepsilon > 0$, write:
\begin{align}
	&\prob[\Big]{n \davg^{-1} \bracket[\big]{\hat{V} - \Var{\eHT}} \leq - \varepsilon}
	\\
	&\qquad \qquad \qquad =
	\prob[\Big]{n \davg^{-1} \Var{\eHT} - n \davg^{-1} \hat{V} \geq \varepsilon}
	\\
	&\qquad \qquad \qquad =
	\prob[\Big]{- n \davg^{-1} \hat{V} \geq - n \davg^{-1} \Var{\eHT} + \varepsilon}
	\\
	&\qquad \qquad \qquad =
	\prob[\Big]{a_n b_n \davg^{-1} - n \davg^{-1} \hat{V} \geq a_n b_n \davg^{-1} - n \davg^{-1} \Var{\eHT} + \varepsilon}
	\\
	&\qquad \qquad \qquad =
	\prob[\Big]{b_n \davg^{-1} \bracket[\big]{a_n - n b_n^{-1} \hat{V}} \geq b_n \davg^{-1} \bracket[\big]{a_n - n b_n^{-1} \Var{\eHT}} + \varepsilon},
\end{align}
where $b_n > 0$ was used.
Next:
\begin{multline}
	\prob[\Big]{b_n \davg^{-1} \bracket[\big]{a_n - n b_n^{-1} \hat{V}} \geq b_n \davg^{-1} \bracket[\big]{a_n - n b_n^{-1} \Var{\eHT}} + \varepsilon}
	\\
	\leq
	\prob[\Big]{b_n \davg^{-1} \abs[\big]{n b_n^{-1} \hat{V} - a_n} \geq b_n \davg^{-1} \bracket[\big]{a_n - n b_n^{-1} \Var{\eHT}} + \varepsilon},
\end{multline}
because $b_n \davg^{-1} \geq 0$, and:
\begin{multline}
	\lim_{n \to \infty} \prob[\Big]{b_n \davg^{-1} \abs[\big]{n b_n^{-1} \hat{V} - a_n} \geq b_n \davg^{-1} \bracket[\big]{a_n - n b_n^{-1} \Var{\eHT}} + \varepsilon}
	\\
	\leq
	\lim_{n \to \infty} \prob[\Big]{b_n \davg^{-1} \abs[\big]{n b_n^{-1} \hat{V} - a_n} \geq \varepsilon},
\end{multline}
because $\liminf_{n \to \infty} b_n \davg^{-1} \bracket[\big]{a_n - n b_n^{-1} \Var{\eHT}} \geq 0$.
Finally, by convergence in probability of $n b_n^{-1} \hat{V}$ to $a_n$ at a rate of $b_n^{-1} \davg$:
\begin{equation}
	\lim_{\varepsilon \to 0^+} \lim_{n \to \infty} \prob[\Big]{b_n \davg^{-1} \abs[\big]{n b_n^{-1} \hat{V} - a_n} \geq \varepsilon} = 0. \tag*{\qedhere}
\end{equation}
\end{proof}

\begin{applemma}\label{lem:positive-variance-conservative-route}
	For a variance estimator $\hat{V}$ and two positive sequences $(a_n)$ and $(b_n)$, if:
	\begin{equation}
		\liminf_{n \to \infty} \bracket[\big]{a_n - n b_n^{-1} \Var{\eHT}} \geq c > 0
		\qquad \text{and} \qquad
		n b_n^{-1} \hat{V} - a_n = \littleOp{1},
	\end{equation}
	then $\hat{V}$ is asymptotically conservative.
\end{applemma}

\begin{proof}
Starting with the expression in Definition~\ref{def:asymptotically-conservative}, note that:
\begin{equation}
	\prob[\Big]{n \davg^{-1} \bracket[\big]{\hat{V} - \Var{\eHT}} \leq - \varepsilon} \leq \prob[\Big]{n \davg^{-1} \bracket[\big]{\hat{V} - \Var{\eHT}} \leq 0},
\end{equation}
for any $\varepsilon \geq 0$.
Similar to the proof of Lemma~\ref{lem:fast-convergence-conservative-route}, write:
\begin{align}
	\prob[\Big]{n \davg^{-1} \bracket[\big]{\hat{V} - \Var{\eHT}} \leq 0}
	&=
	\prob[\Big]{- n \davg^{-1} \hat{V} \geq - n \davg^{-1} \Var{\eHT}}
	\\
	&=
	\prob[\Big]{- n b_n^{-1} \hat{V} \geq - n b_n^{-1} \Var{\eHT}}
	\\
	&=
	\prob[\Big]{a_n - n b_n^{-1} \hat{V} \geq a_n - n b_n^{-1} \Var{\eHT}}
	\\
	&\leq
	\prob[\Big]{\abs[\big]{n b_n^{-1} \hat{V} - a_n} \geq a_n - n b_n^{-1} \Var{\eHT}},
\end{align}
where $b_n^{-1} \davg > 0$ was used.
Next:
\begin{equation}
	\lim_{n \to \infty} \prob[\Big]{\abs[\big]{n b_n^{-1} \hat{V} - a_n} \geq a_n - n b_n^{-1} \Var{\eHT}}
	\leq
	\lim_{n \to \infty} \prob[\Big]{\abs[\big]{n b_n^{-1} \hat{V} - a_n} \geq c},
\end{equation}
because $\liminf_{n \to \infty} \bracket[\big]{a_n - n b_n^{-1} \Var{\eHT}} \geq c$.
Finally:
\begin{equation}
	\lim_{n \to \infty} \prob[\Big]{\abs[\big]{n b_n^{-1} \hat{V} - a_n} \geq c}
	\leq
	\lim_{\varepsilon \to 0^+} \lim_{n \to \infty} \prob[\Big]{\abs[\big]{n b_n^{-1} \hat{V} - a_n} \geq \varepsilon} = 0,
\end{equation}
which follows from $c > 0$ and convergence in probability of $n b_n^{-1} \hat{V}$ to $a_n$.
\end{proof}

\begin{applemma}\label{lem:bound-normalized-variance}
	Under a Bernoulli design:
	\begin{equation}
	n\Var{\eHT} \leq \frac{1}{n} \sum_{i=1}^{n} \di_{i} \sigma_i^2,
	\qquad\text{where}\qquad \sigma_i^2 = \Var[\bigg]{\frac{Z_iY_i}{p_i} - \frac{(1- Z_i)Y_i}{1 - p_i}}
	\end{equation}
	and $\di_{i} = \sum_{j = 1}^n \di_{ij}$.
\end{applemma}

\begin{proof}
As in the proof of Proposition~\ref{prop:var-est-bernoulli-limit}, write the variance as:
\begin{equation}
	n\Var{\eHT} = \frac{1}{n} \sum_{i = 1}^n \Var[\big]{h_i^{-1}Y_i} + \frac{1}{n} \sum_{i = 1}^n \sum_{i \neq j} \Cov[\big]{h_i^{-1}Y_i, h_j^{-1}Y_j},
\end{equation}
where $h_i = Z_i + p_i - 1$.
Note that $\di_{ii} = 1$ for all $i \in \bfU$, so:
\begin{equation}
	\frac{1}{n} \sum_{i = 1}^n \Var[\big]{h_i^{-1}Y_i} = \frac{1}{n} \sum_{i = 1}^n \di_{ii} \Var[\big]{h_i^{-1}Y_i}.
\end{equation}
Using the same argument as in the proofs of Lemmas~\ref{lem:bernoulli-hmu-rates} and~\ref{lem:var-est-convergence}:
\begin{equation}
	(1 - \di_{ij})\Cov{h_i^{-1}Y_i, h_j^{-1}Y_j} = 0,
\end{equation}
so:
\begin{equation}
	\frac{1}{n} \sum_{i = 1}^n \sum_{i \neq j} \Cov[\big]{h_i^{-1}Y_i, h_j^{-1}Y_j}
	=
	\frac{1}{n} \sum_{i = 1}^n \sum_{i \neq j} \di_{ij} \Cov[\big]{h_i^{-1}Y_i, h_j^{-1}Y_j}.
\end{equation}
Using Lemma~\ref{lem:cov-bound-varsum}:
\begin{align}
	&\frac{1}{n} \sum_{i = 1}^n \di_{ii} \Var[\big]{h_i^{-1}Y_i} + \frac{1}{n} \sum_{i = 1}^n \sum_{i \neq j} \di_{ij} \Cov[\big]{h_i^{-1}Y_i, h_j^{-1}Y_j}
	\\
	&\qquad \leq \frac{1}{n} \sum_{i=1}^{n} \di_{ii} \Var{h_i^{-1}Y_i} + \frac{1}{2n} \sum_{i=1}^{n} \sum_{j\neq i} \di_{ij}\Var{h_i^{-1}Y_i} + \frac{1}{2n} \sum_{i=1}^{n} \sum_{j\neq i} \di_{ij}\Var{h_j^{-1}Y_j}
	\\
	&\qquad = \frac{1}{n} \sum_{i=1}^{n} \di_{i} \Var{h_i^{-1}Y_i}.
\end{align}
The proof is completed by noting that $\Var[\big]{h_i^{-1}Y_i} = \sigma_i^2$.
\end{proof}

\begin{refprop}{\ref{prop:var-est-avg-conservative}}
	The variance estimator $\eVarAvg$ is asymptotically conservative under a Bernoulli design if Assumptions~\ref{ass:regularity-conditions} and~\ref{ass:restricted-interference} hold with $q \geq 4$ and $\textsc{sd}_{\sigma^2} = \littleO{\drms^{-1} \davg}$ where:
	\begin{align}
	&\textsc{sd}_{\sigma^2} = \sqrt{\frac{1}{n}\sum_{i=1}^{n}\bracket[\bigg]{\sigma_i^2 - \frac{1}{n} \sum_{j=1}^{n} \sigma_j^2}^2},
	\qquad
	\sigma_i^2 = \Var[\bigg]{\frac{Z_iY_i}{p_i} - \frac{(1- Z_i)Y_i}{1 - p_i}},
	\\
	&\qquad\qquad\qquad\qquad\qquad\text{and}\qquad
	\drms = \sqrt{\frac{1}{n} \sum_{i=1}^n \di_i^2}.
	\end{align}
\end{refprop}

\begin{proof}
The proof uses Lemma~\ref{lem:fast-convergence-conservative-route} with the following sequences:
\begin{equation}
a_n = \frac{1}{n} \sum_{i=1}^{n} \frac{\E{Y_i^2 \given Z_i = 1}}{p_i} + \frac{1}{n} \sum_{i=1}^{n} \frac{\E{Y_i^2 \given Z_i = 0}}{1 - p_i}
\qquad\text{and}\qquad
b_n = \davg.
\end{equation}

Starting with the variance, use Lemma~\ref{lem:bound-normalized-variance} to write:
\begin{equation}
n\Var{\eHT} \leq \frac{1}{n} \sum_{i=1}^{n} \di_{i} \sigma_i^2.
\end{equation}
Let $\bar{\sigma}^2 = n^{-1} \sum_{i=1}^{n} \sigma_i^2$ and write:
\begin{equation}
	\frac{1}{n} \sum_{i=1}^{n} \di_{i} \sigma_i^2
	= \frac{1}{n} \sum_{i=1}^{n} \di_{i} \bracket[\Big]{\sigma_i^2 - \bar{\sigma}^2 + \bar{\sigma}^2}
	= \davg \bar{\sigma}^2 + \frac{1}{n} \sum_{i=1}^{n} \di_{i} \bracket[\Big]{\sigma_i^2 - \bar{\sigma}^2}.
\end{equation}
Using Lemma~\ref{lem:unit-term-variance}:
\begin{equation}
\bar{\sigma}^2 = \frac{1}{n} \sum_{i=1}^{n} \sigma_i^2
\leq \frac{1}{n} \sum_{i=1}^{n} \frac{\E{Y_i^2 \given Z_i = 1}}{p_i} + \frac{1}{n} \sum_{i=1}^{n} \frac{\E{Y_i^2 \given Z_i = 0}}{1 - p_i} = a_n,
\end{equation}
hence $\davg \bar{\sigma}^2 \leq a_n \davg$.
Use the Cauchy--Schwarz inequality for inner products:
\begin{multline}
\frac{1}{n} \sum_{i=1}^{n} \di_{i} \bracket[\Big]{\sigma_i^2 - \bar{\sigma}^2}
\leq
\sqrt{\frac{1}{n^2} \paren[\Bigg]{\sum_{i=1}^{n} \di_{i} \bracket[\Big]{\sigma_i^2 - \bar{\sigma}^2}}^2}
\\
\leq
\sqrt{\paren[\Bigg]{\frac{1}{n}\sum_{i=1}^{n} \di_{i}^2} \paren[\Bigg]{\frac{1}{n}\sum_{i=1}^{n}\bracket[\Big]{\sigma_i^2 - \bar{\sigma}^2}^2}} = \drms \textsc{sd}_{\sigma^2}.
\end{multline}
Taken together:
\begin{equation}
	n b_n^{-1} \Var{\eHT} = n \davg^{-1} \Var{\eHT} \leq a_n + \frac{\drms \textsc{sd}_{\sigma^2}}{\davg},
\end{equation}
and the first expression in Lemma~\ref{lem:fast-convergence-conservative-route} is bounded as:
\begin{equation}
	b_n \davg^{-1} \bracket[\big]{a_n - n b_n^{-1} \Var{\eHT}} \geq - \frac{\drms \textsc{sd}_{\sigma^2}}{\davg}.
\end{equation}
By assumption:
\begin{equation}
	\lim_{n \to \infty} \frac{\drms \textsc{sd}_{\sigma^2}}{\davg} = 0,
\end{equation}
so:
\begin{equation}
	\liminf_{n \to \infty} b_n \davg^{-1} \bracket[\big]{a_n - n b_n^{-1} \Var{\eHT}} \geq 0,
\end{equation}
and the first premise of Lemma~\ref{lem:fast-convergence-conservative-route} is satisfied.
The second premise is satisfied by Corollary~\ref{coro:mod-var-est-convergence} because $\eVarAvg = b_n \eVarBer$ and $b_n^{-1} \davg = 1$.
\end{proof}

\begin{refprop}{\ref{prop:var-est-max-conservative}}
	The variance estimator $\eVarMax$ is asymptotically conservative under a Bernoulli design if either:
	\begin{enumerate}
		\item Assumption~\ref{ass:regularity-conditions} holds with $q \geq 4$ and $\dmax = \littleO{n^{0.5}\davg^{0.5}}$, or
		\item Assumptions~\ref{ass:regularity-conditions} and~\ref{ass:restricted-interference} hold with $q \geq 4$ and $\eMSTE = \bigOmega{1}$.
	\end{enumerate}
\end{refprop}

\begin{proof}
The proof uses Lemmas~\ref{lem:fast-convergence-conservative-route} and~\ref{lem:positive-variance-conservative-route} with the following sequences:
\begin{equation}
a_n = \frac{1}{n} \sum_{i=1}^{n} \frac{\E{Y_i^2 \given Z_i = 1}}{p_i} + \frac{1}{n} \sum_{i=1}^{n} \frac{\E{Y_i^2 \given Z_i = 0}}{1 - p_i}
\qquad\text{and}\qquad
b_n = \dmax.
\end{equation}

Starting with the variance, use Lemma~\ref{lem:bound-normalized-variance} to write:
\begin{equation}
n\Var{\eHT} \leq \frac{1}{n} \sum_{i=1}^{n} \di_{i} \sigma_i^2.
\end{equation}
By construction, $\sigma_i^2 \geq 0$, so:
\begin{equation}
\frac{1}{n} \sum_{i=1}^{n} \di_{i} \sigma_i^2 \leq \frac{\dmax}{n} \sum_{i=1}^{n} \sigma_i^2.
\end{equation}
Using Lemma~\ref{lem:unit-term-variance}:
\begin{multline}
\frac{\dmax}{n} \sum_{i=1}^{n} \sigma_i^2
= \frac{\dmax}{n} \sum_{i=1}^{n} \bracket[\bigg]{\frac{\E{Y_i^2 \given Z_i = 1}}{p_i} + \frac{\E{Y_i^2 \given Z_i = 0}}{1 - p_i} - \paren[\big]{\E{\tau_i(\Z_{-i})}}^2}
\\
= \dmax \paren{a_n - \eMSTE},
\end{multline}
so:
\begin{equation}
	n b_n^{-1} \Var{\eHT} = n \dmax^{-1} \Var{\eHT} \leq a_n - \eMSTE,
\end{equation}
and:
\begin{equation}
	a_n - n b_n^{-1} \Var{\eHT} \geq \eMSTE.
\end{equation}

Consider the first case in the proposition using Lemma~\ref{lem:fast-convergence-conservative-route}.
By construction, $\eMSTE \geq 0$ and $b_n \davg^{-1} = \davg^{-1} \dmax \geq 1$, so:
\begin{equation}
	\liminf_{n \to \infty} b_n \davg^{-1} \bracket[\big]{a_n - n b_n^{-1} \Var{\eHT}} \geq \liminf_{n \to \infty} \eMSTE \geq 0,
\end{equation}
and the first premise of the lemma is satisfied without invoking additional assumptions.
The second premise is satisfied by Corollary~\ref{coro:mod-var-est-fast-convergence} because $\eVarAvg = b_n \eVarBer$ and $b_n = \dmax = \littleO{n^{0.5}\davg^{0.5}}$ by assumption.

Consider the second case in the proposition using Lemma~\ref{lem:positive-variance-conservative-route}.
By assumption:
\begin{equation}
	\liminf_{n \to \infty} \eMSTE \geq c > 0,
\end{equation}
so:
\begin{equation}
	\liminf_{n \to \infty} \bracket[\big]{a_n - n b_n^{-1} \Var{\eHT}} \geq c > 0,
\end{equation}
and the first premise of the lemma is satisfied.
Corollary~\ref{coro:mod-var-est-convergence} gives the second premise.
\end{proof}

\begin{refprop}{\ref{prop:var-est-eig-conservative}}
	The variance estimator $\eVarSR$ is asymptotically conservative under a Bernoulli design if either:
	\begin{enumerate}
		\item Assumption~\ref{ass:regularity-conditions} holds with $q \geq 4$ and $\eigmax = \littleO{n^{0.5}\davg^{0.5}}$, or
		\item Assumptions~\ref{ass:regularity-conditions} and~\ref{ass:restricted-interference} hold with $q \geq 4$ and $\eMSTE = \bigOmega{1}$.
	\end{enumerate}
\end{refprop}

\begin{proof}
The proof uses Lemmas~\ref{lem:fast-convergence-conservative-route} and~\ref{lem:positive-variance-conservative-route} with the following sequences:
\begin{equation}
a_n = \frac{1}{n} \sum_{i=1}^{n} \frac{\E{Y_i^2 \given Z_i = 1}}{p_i} + \frac{1}{n} \sum_{i=1}^{n} \frac{\E{Y_i^2 \given Z_i = 0}}{1 - p_i}
\qquad\text{and}\qquad
b_n = \eigmax.
\end{equation}

Similar to the proof of Lemma~\ref{lem:bound-normalized-variance}, write the variance as:
\begin{equation}
	n\Var{\eHT} = \frac{1}{n} \sum_{i = 1}^n \sum_{j = 1}^n \di_{ij} \Cov[\big]{h_i^{-1}Y_i, h_j^{-1}Y_j}.
\end{equation}
Using the Cauchy--Schwarz inequality:
\begin{multline}
	\frac{1}{n} \sum_{i = 1}^n \sum_{j = 1}^n \di_{ij} \Cov[\big]{h_i^{-1}Y_i, h_j^{-1}Y_j}
	\leq
	\sum_{i = 1}^n \sum_{j = 1}^n \di_{ij} \sqrt{n^{-1}\Var[\big]{h_i^{-1}Y_i}} \sqrt{n^{-1}\Var[\big]{h_j^{-1}Y_j}}
	\\
	=
	\sum_{i = 1}^n \sum_{j = 1}^n \di_{ij} \sqrt{n^{-1}\sigma_i^2} \sqrt{n^{-1}\sigma_j^2}
\end{multline}
Let $\mathbf{v}$ be a column vector of length $n$ whose typical element is $\sqrt{n^{-1}\sigma_i^2}$, and let $\mathbf{D}$ be a $n \times n$ matrix whose typical element is $\di_{ij}$.
Then:
\begin{equation}
\sum_{i = 1}^n \sum_{j = 1}^n \di_{ij} \sqrt{n^{-1}\sigma_i^2} \sqrt{n^{-1}\sigma_j^2}
= \mathbf{v}^\transpose \mathbf{D} \mathbf{v}
= \mathbf{v}^\transpose \mathbf{v} \bracket[\bigg]{\frac{\mathbf{v}^\transpose \mathbf{D} \mathbf{v}}{\mathbf{v}^\transpose \mathbf{v}}}.
\end{equation}
The first factor is:
\begin{equation}
\mathbf{v}^\transpose \mathbf{v}
=
\sum_{i = 1}^n \bracket[\bigg]{\sqrt{n^{-1}\sigma_i^2}}^2
=
\frac{1}{n} \sum_{i = 1}^n \sigma_i^2.
\end{equation}
The second factor is the Rayleigh quotient, which is bounded by the largest eigenvalue of $\mathbf{D}$ for all $\mathbf{v} \in \Reals^n$.
It follows that:
\begin{equation}
\mathbf{v}^\transpose \mathbf{v} \bracket[\bigg]{\frac{\mathbf{v}^\transpose \mathbf{D} \mathbf{v}}{\mathbf{v}^\transpose \mathbf{v}}}
\leq
\frac{\eigmax}{n} \sum_{i = 1}^n \sigma_i^2.
\end{equation}
Similar to the proof of Proposition~\ref{prop:var-est-max-conservative}, use Lemma~\ref{lem:unit-term-variance} to write:
\begin{equation}
\frac{\eigmax}{n} \sum_{i = 1}^n \sigma_i^2
= \eigmax \paren{a_n - \eMSTE},
\end{equation}
so:
\begin{equation}
	n b_n^{-1} \Var{\eHT} = n \eigmax^{-1} \Var{\eHT} \leq a_n - \eMSTE,
\end{equation}
and:
\begin{equation}
	a_n - n b_n^{-1} \Var{\eHT} \geq \eMSTE.
\end{equation}

Consider the first case in the proposition using Lemma~\ref{lem:fast-convergence-conservative-route}.
By construction, $\eMSTE \geq 0$ and $b_n \davg^{-1} = \davg^{-1} \eigmax \geq 1$, so:
\begin{equation}
	\liminf_{n \to \infty} b_n \davg^{-1} \bracket[\big]{a_n - n b_n^{-1} \Var{\eHT}} \geq \liminf_{n \to \infty} \eMSTE \geq 0,
\end{equation}
and the first premise of the lemma is satisfied without invoking additional assumptions.
The second premise is satisfied by Corollary~\ref{coro:mod-var-est-fast-convergence} because $\eVarAvg = b_n \eVarBer$ and $b_n = \eigmax = \littleO{n^{0.5}\davg^{0.5}}$ by assumption.

Consider the second case in the proposition using Lemma~\ref{lem:positive-variance-conservative-route}.
By assumption:
\begin{equation}
	\liminf_{n \to \infty} \eMSTE \geq c > 0,
\end{equation}
so:
\begin{equation}
	\liminf_{n \to \infty} \bracket[\big]{a_n - n b_n^{-1} \Var{\eHT}} \geq c > 0,
\end{equation}
and the first premise of the lemma is satisfied.
Corollary~\ref{coro:mod-var-est-convergence} gives the second premise.
\end{proof}

\section{Proofs of Proposition~\ref{prop:Chebyshev-sharp}}

\begin{appdefinition}\label{def:Chebyshev-sharp-RV}
	Let $X(\alpha)$ be a random variable such that:
	\begin{equation}
		\prob[\big]{X(\alpha) = x}
		=
		\begin{cases}
			\alpha / 2 & \text{if } x = -1,
			\\
			1 - \alpha & \text{if } x = 0,
			\\
			\alpha / 2 & \text{if } x = 1.
		\end{cases}
	\end{equation}
\end{appdefinition}

\begin{applemma}
	The $\alpha$-tail bounds provided by Chebyshev's inequality is sharp for $X(\alpha)$.
\end{applemma}

\begin{proof}
If Chebyshev's inequality is sharp, then:
\begin{equation}
\prob[\Bigg]{\abs[\Big]{X(\alpha) - \E[\big]{X(\alpha)}} \geq \sqrt{\frac{\Var{X(\alpha)}}{\alpha}} } = \alpha
\end{equation}

Consider the expectation and variance:
\begin{equation}
\E[\big]{X(\alpha)} = 0
\qquad\text{and}\qquad
\Var{X(\alpha)} = \alpha.
\end{equation}
Hence:
\begin{equation}
\prob[\Bigg]{\abs[\Big]{X(\alpha) - \E[\big]{X(\alpha)}} \geq \sqrt{\frac{\Var{X(\alpha)}}{\alpha}} }
=
\prob[\Big]{\abs[\big]{X(\alpha)} \geq 1} = \alpha. \tag*{\qedhere}
\end{equation}
\end{proof}

\begin{refprop}{\ref{prop:Chebyshev-sharp}}
	Chebyshev's inequality is asymptotically sharp with respect to the sampling distribution of the \HT\ estimator for every sequence of Bernoulli designs under Assumptions~\ref{ass:regularity-conditions} and~\ref{ass:restricted-interference}.
	The inequality remains sharp when Assumption~\ref{ass:restricted-interference} is strengthened to $\davg = \bigO{1}$ and $\dmax = \bigO{n^{0.5}}$.
\end{refprop}

\begin{proof}
The proposition states that under Assumption~\ref{ass:regularity-conditions}, for every sequence of Bernoulli designs, for every $\alpha \in (0,1)$, and for $\davg = \bigO{n^x}$ for every $0 \leq x < 1$, there exists a sequence of potential outcomes such that the normalized sampling distribution of the \HT\ estimator convergence in distribution to the random variable in Definition~\ref{def:Chebyshev-sharp-RV}:
\begin{equation}
	n^{0.5} \davg^{-0.5} \paren[\big]{\eHT - \E{\eHT}} \darrow X(\alpha).
\end{equation}

Let $\z_{\textsc{sub},n} = \paren{z_1, z_2, \dotsc, z_{b_n}}$ be a vector containing the treatment assignments of first $b_n$ units in the sample.
Let $b_n \to \infty$ at a rate such that $b_n = \littleO[\big]{n^{(x + 1) / 2}}$ where $x$ is such that $\davg = \bigO{n^x}$ for the desired rate.
Let $\mathcal{F}_n$ be all functions mapping from the support of $\z_{\textsc{sub},n}$ to $\braces{-1, 0, 1}$.
That is all functions $f : \braces{0, 1}^{b_n} \to \braces{-1, 0, 1}$.
Let:
\begin{equation}
	f^*_n = \argmin_{f \in \mathcal{F}_n} \bracket[\Big]{ \abs[\big]{\prob[\big]{f(\z_{\textsc{sub},n}) = 1} - \alpha / 2} + \abs[\big]{\prob[\big]{f(\z_{\textsc{sub},n}) = -1} - \alpha / 2} }.
\end{equation}
Ties are broken arbitrarily.

The Bernoulli design and Assumption~\ref{ass:regularity-conditions} imply that the support of $\z'$ consists of $2^{b_n}$ points each with a probability of at least $k^{-b_n}$.
It follows that:
\begin{equation}
	\lim_{n \to \infty} \prob[\big]{f^*_n(\z_{\textsc{sub},n}) = 1} = \alpha / 2
	\qquad\text{and}\qquad
	\lim_{n \to \infty} \prob[\big]{f^*_n(\z_{\textsc{sub},n}) = -1} = \alpha / 2.
\end{equation}

Let $h_n = \ceil[\big]{n^{(x + 1) / 2}}$, where $\ceil{\cdot}$ is the ceiling function.
Consider potential outcomes:
\begin{equation}
	y_i(\z) = \left\{ \begin{array}{ll}
	h_n^{-1} n^{0.5} \davg^{0.5} \bracket[\big]{z_i\prob{Z_i = 1} - (1-z_i)\prob{Z_i = 0}} f^*_n(\z_{\textsc{sub},n}) & \text{if } i \leq h_n,
	\\[0.5em]
	0 & \text{otherwise}.
	\end{array}\right.
\end{equation}
These potential outcomes imply:
\begin{equation}
	\di_{ij} = \left\{ \begin{array}{ll}
	1 & \text{if } i = j \text{ or } i,j\leq h_n,
	\\[0.5em]
	0 & \text{otherwise},
	\end{array}\right.
\end{equation}
so with a similar argument as in the proof of Proposition~\ref{prop:restricted-interference-necessary}:
\begin{equation}
	\davg = \frac{h_n^2}{n} + \frac{n - h_n}{n} = \frac{\ceil[\big]{n^{(x + 1) / 2}}^2}{n} + \frac{n - \ceil[\big]{n^{(x + 1) / 2}}}{n} \sim n^x + 1,
\end{equation}
as desired, namely $\davg = \bigO{n^x}$.
Furthermore:
\begin{equation}
	\dmax = h_n = \ceil[\big]{n^{(x + 1) / 2}} \sim n^{(x + 1) / 2}.
\end{equation}
Note that $x = 0$ implies $\dmax = \bigO{n^{0.5}}$.

Next, make sure the potential outcomes satisfy Assumption~\ref{ass:regularity-conditions}.
First, note that:
\begin{equation}
	-1 \leq \bracket[\big]{z_i\prob{Z_i = 1} - (1-z_i)\prob{Z_i = 0}} f^*_n(\z_{\textsc{sub},n}) \leq 1.
\end{equation}
Next:
\begin{align}
h_n^{-1} n^{0.5} \davg^{0.5} &= h_n^{-1} n^{0.5} \bracket[\bigg]{\frac{h_n^2}{n} + \frac{n - h_n}{n}}^{0.5}
\leq h_n^{-1} n^{0.5} \bracket[\bigg]{\frac{h_n}{n^{0.5}} + 1}
= 1 + h_n^{-1} n^{0.5} \leq 2
\end{align}
where the last inequality follows from $h_n = \ceil[\big]{n^{(x + 1) / 2}} \geq n^{0.5}$ for all $x \geq 0$.
It follows that $-2 \leq y_i(\z) \leq 2$ for all units and assignments, so Assumption~\ref{ass:regularity-conditions} holds.

Consider the \HT\ estimator:
\begin{align}
	\eHT &= \frac{1}{n} \sum_{i=1}^{n}\frac{Z_i Y_i}{\prob{Z_i = 1}} - \frac{1}{n} \sum_{i=1}^{n} \frac{(1-Z_i) Y_i}{\prob{Z_i = 0}}
	\\
	&= \frac{1}{n} \sum_{i=1}^{n}\frac{Z_i y_i(\Z)}{\prob{Z_i = 1}} - \frac{1}{n} \sum_{i=1}^{n} \frac{(1-Z_i) y_i(\Z)}{\prob{Z_i = 0}}
	\\
	&= \frac{1}{n} \sum_{i=1}^{h_n}\frac{Z_i y_i(\Z)}{\prob{Z_i = 1}} - \frac{1}{n} \sum_{i=1}^{h_n} \frac{(1-Z_i) y_i(\Z)}{\prob{Z_i = 0}}
	\\
	&= \frac{1}{n} \sum_{i=1}^{h_n} Z_i h_n^{-1} n^{0.5} \davg^{0.5} f^*_n(\z_{\textsc{sub},n})
	 + \frac{1}{n} \sum_{i=1}^{h_n} (1-Z_i) h_n^{-1} n^{0.5} \davg^{0.5} f^*_n(\z_{\textsc{sub},n})
	\\
	&= \frac{n^{0.5} \davg^{0.5}}{n} f^*_n(\z_{\textsc{sub},n}) = n^{-0.5} \davg^{0.5} f^*_n(\z_{\textsc{sub},n})
\end{align}
It follows that:
\begin{equation}
	\E{\eHT} = n^{-0.5} \davg^{0.5} \E[\big]{f^*_n(\z_{\textsc{sub},n})} = n^{-0.5} \davg^{0.5} \bracket[\Big]{ \prob[\big]{f^*_n(\z_{\textsc{sub},n}) = 1} - \prob[\big]{f^*_n(\z_{\textsc{sub},n}) = -1} }.
\end{equation}
Taken together:
\begin{equation}
	n^{0.5} \davg^{-0.5} \paren[\big]{\eHT - \E{\eHT}} = f^*_n(\z_{\textsc{sub},n}) - \bracket[\Big]{ \prob[\big]{f^*_n(\z_{\textsc{sub},n}) = 1} - \prob[\big]{f^*_n(\z_{\textsc{sub},n}) = -1} }
\end{equation}

It remains to show that:
\begin{equation}
\lim_{n \to \infty} \prob[\Big]{n^{0.5} \davg^{-0.5} \paren[\big]{\eHT - \E{\eHT}} \leq s} = \prob[\big]{X(\alpha) \leq s},
\end{equation}
for all $s \in \Reals \setminus \braces{-1, 0, 1}$.
Note that the construction of $f^*_n$ is such that:
\begin{equation}
	\lim_{n \to \infty} \bracket[\Big]{ \prob[\big]{f^*_n(\z_{\textsc{sub},n}) = 1} - \prob[\big]{f^*_n(\z_{\textsc{sub},n}) = -1} } = 0
\end{equation}
Hence, for any fixed $s \in \Reals \setminus \braces{-1, 0, 1}$, there exists an $n'$ such that the smallest distance from $s$ to any point in $\braces{-1, 0, 1}$ for all $n \geq n'$ is less than:
\begin{equation}
	\prob[\big]{f^*_n(\z_{\textsc{sub},n}) = 1} - \prob[\big]{f^*_n(\z_{\textsc{sub},n}) = -1}.
\end{equation}
For the $s \in \Reals \setminus \braces{-1, 0, 1}$ under consideration and all such $n \geq n'$:
\begin{equation}
\prob[\Big]{n^{0.5} \davg^{-0.5} \paren[\big]{\eHT - \E{\eHT}} \leq s} = \prob[\big]{f^*_n(\z_{\textsc{sub},n}) \leq s}
\end{equation}
As noted above, by construction of $f^*_n$:
\begin{equation}
\lim_{n \to \infty} \prob[\big]{f^*_n(\z_{\textsc{sub},n}) \leq s}
= \prob[\big]{X(\alpha) \leq s}
= \begin{cases}
	0 & \text{if } x < -1,
	\\
	\alpha / 2 & \text{if } x < 0,
	\\
	1 - \alpha / 2 & \text{if } x < 1,
	\\
	1 & \text{if } x \geq 1.
\end{cases} \tag*{\qedhere}
\end{equation}
\end{proof}

\begin{appremark}
The key ingredient of the proof of Proposition~\ref{prop:Chebyshev-sharp} is that the support of $\z_{\textsc{sub},n}$ is sufficient large so there exists $f^*_n$ such that:
\begin{equation}
	\lim_{n \to \infty} \prob[\big]{f^*_n(\z_{\textsc{sub},n}) = 1} = \alpha / 2
	\qquad\text{and}\qquad
	\lim_{n \to \infty} \prob[\big]{f^*_n(\z_{\textsc{sub},n}) = -1} = \alpha / 2.
\end{equation}
This is ensured by $b_n \to \infty$ as long as the treatment assignments are sufficiently independent, which holds for complete and paired randomization.
The complication with complete randomization is that $b_n$ must grow sufficiently slowly so that $\icoutavg=\littleO[\big]{n^{0.5}}$.
This may make the convergence in distribution to $X(\alpha)$ slower, but should not pose a concern otherwise.
For paired randomization, one needs to make sure the content of $\z_{\textsc{sub},n}$ is such that Assumption~\ref{ass:pair-separation} holds.
This can be done by having at most one unit in each pair in $\z_{\textsc{sub},n}$.
The proof for arbitray designs requires similar considerations with respect to Assumptions~\ref{ass:design-mixing} and~\ref{ass:design-separation}.
\end{appremark}

\section{Proofs of Lemma~\ref{lem:ate-function-Lipschitz} and Propositions~\ref{prop:total-variation-distance-bound} and~\ref{prop:Wasserstein-bound}}

\begin{refprop}{\ref{prop:total-variation-distance-bound}}
	Given Assumption~\ref{ass:bounded-unit-effects}:
	\begin{equation}
		\abs{\eEATEp - \eEATEq} \leq 2 \teconst \delta(P, Q).
	\end{equation}
\end{refprop}

\begin{proof}
Because the sample space is discrete:
\begin{equation}
\eEATEp = \Ep{\eATE(\Z)} = \sum_{\z \in \braces{0,1}^n} \probp{\Z = \z} \eATE(\z)
\end{equation}
and similarly for $\eEATEq$.
Hence:
\begin{align}
\abs{\eEATEp - \eEATEq} &= \abs[\Bigg]{\sum_{\z \in \braces{0,1}^n} \bracket[\big]{\probp{\Z = \z} - \probq{\Z = \z}} \eATE(\z)}
\\
								&\leq \sum_{\z \in \braces{0,1}^n} \abs{\eATE(\z)} \times \abs[\big]{\probp{\Z = \z} - \probq{\Z = \z}}
\end{align}

The average treatment effect function, $\eATE(\z)$, is bounded as:
\begin{equation}
	\abs{\eATE(\z)} \leq \frac{1}{n} \sum_{i=1}^{n} \abs{\tau_i(\z_{-i})} \leq \frac{1}{n} \sum_{i=1}^{n} \teconst = \teconst,
\end{equation}
where Assumption~\ref{ass:bounded-unit-effects} gives $\abs{\tau_i(\z_{-i})} \leq \teconst$.
It follows that:
\begin{equation}
	\abs{\eEATEp - \eEATEq} \leq \teconst \sum_{\z \in \braces{0,1}^n} \abs[\big]{\probp{\Z = \z} - \probq{\Z = \z}}
\end{equation}

Let $A = \braces{\z \in \braces{0,1}^n: \probp{\Z = \z} \geq \probq{\Z = \z}}$, and $A^\complement = \braces{0,1}^n \setminus A$, so that:
\begin{equation}
	\sum_{\z \in \braces{0,1}^n} \abs[\big]{\probp{\Z = \z} - \probq{\Z = \z}}
	=
	\bracket[\big]{P(A) - Q(A)} + \bracket[\big]{Q(A^\complement) - P(A^\complement)}
\end{equation}
The two terms are equal:
\begin{equation}
	P(A) - Q(A) = Q(A^\complement) - P(A^\complement),
\end{equation}
because:
\begin{align}
	P(A) + P(A^\complement) = \sum_{\z \in A} \probp{\Z = \z} + \sum_{\z \in A^\complement} \probp{\Z = \z} = \sum_{\z \in \braces{0,1}^n} \probp{\Z = \z} = 1
\end{align}
and similarly with $Q$.
Hence:
\begin{equation}
	\abs{\eEATEp - \eEATEq} \leq 2\teconst \bracket[\big]{P(A) - Q(A)}
\end{equation}
The event under consideration is in the event space, $A \in \mathcal{F}$, so:
\begin{equation}
	\bracket[\big]{P(A) - Q(A)} \leq \sup_{x \in \mathcal{F}} \, \abs{P(x) - Q(x)} = \delta(P, Q). \tag*{\qedhere}
\end{equation}
\end{proof}

\begin{reflem}{\ref{lem:ate-function-Lipschitz}}
	Given Assumption~\ref{ass:bounded-unit-effects}, $\eATE(\z)$ is $2 \teconst n^{-1/r} \icoutmom{r / (r - 1)}$-Lipschitz continuous with respect to the $L_r$ distance over $\braces{0, 1}^n$ for any $r \geq 1$.
\end{reflem}

\begin{proof}
The lemma states that:
\begin{equation}
\abs[\big]{\eATE(\z') - \eATE(\z'')} \leq 2 \teconst n^{-1/r} \icoutmom{r / (r - 1)} \norm{\z' - \z''}_r
\end{equation}
for any $\z', \z'' \in \braces{0,1}^n$ and $r \geq 1$.

Consider a set of $n + 1$ vectors in $\braces{0,1}^n$ indexed by $v \in \braces{0, 1, 2, \dotsc, n}$, denoted $\z^{v}$.
The typical element of these vectors is:
\begin{equation}
	z^{v}_i = \begin{cases}
		z'_i & \text{if } v < i,
		\\
		z''_i & \text{if } v \geq i.
	\end{cases}
\end{equation}
In other words, as the index goes from $0$ to $n$, the vector gradually goes from $\z'$ to $\z''$, changing at most one element at each step.
For example, $\z^{2} = \paren{z''_1, z''_2, z'_3, z'_4, \dotsc, z'_n}$.

Note that $\z^{0} = \z'$ and $\z^{n} = \z''$, so:
\begin{equation}
\abs[\big]{\eATE(\z') - \eATE(\z'')} = \abs[\big]{\eATE(\z^{0}) - \eATE(\z^{n})}
\end{equation}
Add and subtract $\eATE(\z^{v})$ for $v \in \braces{1, 2, \dotsc, n - 1}$ to get:
\begin{multline}
\abs[\big]{\eATE(\z^{0}) - \eATE(\z^{n})}
\\
= \abs[\big]{\eATE(\z^{0}) + \eATE(\z^{1}) - \eATE(\z^{1}) + \dotsb + \eATE(\z^{n - 1}) - \eATE(\z^{n - 1}) - \eATE(\z^{n})},
\end{multline}
which yields the following bound:
\begin{equation}
\abs[\big]{\eATE(\z^{0}) - \eATE(\z^{n})} \leq \sum_{v = 1}^n \abs[\big]{\eATE(\z^{v}) - \eATE(\z^{v - 1})}.
\end{equation}

Recall that $\z^{v}$ and $\z^{v - 1}$ differ at most in one element.
In particular:
\begin{equation}
	\z^{v} - \z^{v - 1} = \paren{0, \dotsc, 0, z_v'' - z_v', 0, \dotsc, 0}.
\end{equation}
Hence, if $z''_v = z'_v$, then $\eATE(\z^{v}) = \eATE(\z^{v - 1})$, so:
\begin{equation}
	\abs[\big]{\eATE(\z^{v}) - \eATE(\z^{v - 1})} = \abs{z_v' - z_v''} \times \abs[\big]{\eATE(\z^{v}) - \eATE(\z^{v - 1})}.
\end{equation}

Recall Definition~\ref{def:acate}, and note:
\begin{equation}
	\abs[\big]{\eATE(\z^{v}) - \eATE(\z^{v - 1})} \leq \frac{1}{n} \sum_{i=1}^{n} \abs[\big]{\tau_i(\z^{v}_{-i}) - \tau_i(\z^{v - 1}_{-i})}
\end{equation}
where $\z^{v}_{-i}$ is $\z^{v}$ with the $i$th element deleted.
Similar to Section~\ref{sec:arbitrary-designs}, let:
\begin{equation}
	\tilde{\z}^{v}_{-i} = (I_{1i}z^{v}_1, \dotsc, I_{i - 1,i}z^{v}_{i - 1}, I_{i + 1,i}z^{v}_{i + 1}, \dotsc, I_{ni}z^{v}_n).
\end{equation}
Because of the definition of $I_{vi}$, we have $\tau_i(\z^{v}_{-i}) = \tau_i(\tilde{\z}^{v}_{-i})$.
As before, $\tilde{\z}^{v}_{-i}$ and $\tilde{\z}^{v - 1}_{-i}$ differ at most in one element:
\begin{equation}
	\tilde{\z}^{v}_{-i} - \tilde{\z}^{v - 1}_{-i} = \paren{0, \dotsc, 0, I_{vi}z_v'' - I_{vi}z_v', 0, \dotsc, 0}.
\end{equation}
except when $v = i$, in which case $\tilde{\z}^{v}_{-i} - \tilde{\z}^{v - 1}_{-i} = \mathbf{0}$.
It follows that $\abs{\tau_i(\z^{v}_{-i}) - \tau_i(\z^{v - 1}_{-i})} = 0$ when either $v = i$ or $I_{vi} = 0$, hence:
\begin{equation}
	\abs{\tau_i(\z^{v}_{-i}) - \tau_i(\z^{v - 1}_{-i})} = \indicator{v \neq i \wedge I_{vi} = 1} \times \abs{\tau_i(\z^{v}_{-i}) - \tau_i(\z^{v - 1}_{-i})}.
\end{equation}
[An alternative route is to use $\indicator{v \neq i \wedge z_v' \neq z_v'' \wedge I_{vi} = 1}$ here, which would make the lemma sharper but necessitate other changes.]

Finally, by Assumption~\ref{ass:bounded-unit-effects}, $\abs{\tau_i(\z^{v}_{-i}) - \tau_i(\z^{v - 1}_{-i})} \leq 2 \teconst$, so:
\begin{equation}
	\abs{\tau_i(\z^{v}_{-i}) - \tau_i(\z^{v - 1}_{-i})} \leq 2 \teconst\indicator{v \neq i \wedge I_{vi} = 1} \leq 2 \teconst I_{vi},
\end{equation}
where the last inequality follows from $I_{vi} \in \braces{0,1}$.

Returning to the assignment-conditional average treatment effect functions:
\begin{equation}
	\abs[\big]{\eATE(\z^{v}) - \eATE(\z^{v - 1})} \leq \frac{1}{n} \sum_{i=1}^{n} \abs[\big]{\tau_i(\z^{v}_{-i}) - \tau_i(\z^{v - 1}_{-i})} \leq \frac{2 \teconst}{n} \sum_{i=1}^{n} I_{vi} = \frac{2 \teconst \icout_v}{n},
\end{equation}
where the last equality follows from the definition $\icout_i = \sum_{j=1}^n I_{ij}$ in Section~\ref{sec:quantifying-interference}.
Taken together:
\begin{equation}
\abs[\big]{\eATE(\z') - \eATE(\z'')}
\leq \sum_{v = 1}^n \abs{z_v' - z_v''} \times \abs[\big]{\eATE(\z^{v}) - \eATE(\z^{v - 1})}
\leq \frac{2 \teconst}{n} \sum_{v = 1}^n \abs{z_v' - z_v''} \icout_v
\end{equation}
Use Hölder's inequality with conjugates $r$ and $r / (r - 1)$ to separate the factors in the sum:
\begin{equation}
\frac{2 \teconst}{n} \sum_{i = 1}^n \abs{z_i' - z_i''} \icout_i \leq 2 \teconst n^{-1/r} \bracket[\bigg]{\sum_{i = 1}^n \abs{z_i' - z_i''}^r}^{\frac{1}{r}} \bracket[\bigg]{\frac{1}{n} \sum_{i = 1}^n \icout_i^{r / (r - 1)}}^{\frac{r - 1}{r}},
\end{equation}
which is equal to $2 \teconst n^{-1/r} \icoutmom{r / (r - 1)} \norm{\z' - \z''}_r$.
\end{proof}

\begin{refprop}{\ref{prop:Wasserstein-bound}}
	Given Assumption~\ref{ass:bounded-unit-effects}:
	\begin{equation}
		\abs{\eEATEp - \eEATEq} \leq 2 \teconst n^{-1/r} \icoutmom{r / (r - 1)} W_r(P, Q).
	\end{equation}
\end{refprop}

\newcommand{\normeATE}{\tilde{\tau}_{\normalfont\ATE}}

\begin{proof}
Let $\normeATE(\z, r)$ be $\eATE(\z)$ normalized by its Lipschitz constant for the $L_r$ distance:
\begin{equation}
	\normeATE(\z, r) = \frac{\eATE(\z)}{2 \teconst n^{-1/r} \icoutmom{r / (r - 1)}}.
\end{equation}
We can now write:
\begin{multline}
	\abs{\eEATEp - \eEATEq} = \abs[\big]{\Ep{\eATE(\Z)} - \Eq{\eATE(\Z)}}
	\\
	= 2 \teconst n^{-1/r} \icoutmom{r / (r - 1)} \abs[\big]{\Ep{\normeATE(\z, r)} - \Eq{\normeATE(\z, r)}},
\end{multline}
because $2 \teconst n^{-1/r} \icoutmom{r / (r - 1)}$ is a positive constant given $n$ and $r$.

Under Assumption~\ref{ass:bounded-unit-effects}, $\normeATE(\z, r)$ is 1-Lipschitz continuous.
Let $\mathcal{L}_{1,r}$ contain all 1-Lipschitz continuous functions on $\braces{0, 1}^n$ with respect to the $L_r$ distance.
Bound the difference between the normalized estimands as:
\begin{equation}
	\abs[\big]{\Ep{\normeATE(\z, r)} - \Eq{\normeATE(\z, r)}}
	\leq
	\sup\braces[\Big]{ \abs[\big]{\Ep{f(\Z)} - \Eq{f(\Z)}} : f \in \mathcal{L}_{1,r} }.
\end{equation}
By the Kantorovich-Rubinstein duality theorem \citep[see, e.g.,][]{Edwards2011}:
\begin{equation}
	\sup\braces[\Big]{ \abs[\big]{\Ep{f(\Z)} - \Eq{f(\Z)}} : f \in \mathcal{L}_{1,r} } = W_r(P, Q). \tag*{\qedhere}
\end{equation}
\end{proof}

\clearpage

\setcounter{section}{0}%
\setcounter{table}{0}%
\setcounter{figure}{0}%
\setcounter{equation}{0}%

\renewcommand\thesection{B\arabic{section}}%
\renewcommand\thetable{B\arabic{table}}%
\renewcommand\thefigure{B\arabic{figure}}%
\renewcommand\theequation{B\arabic{equation}}%

\begin{center}
	\Huge \textbf{Supplement B: Simulation study}
\end{center}
\bigskip

	\section{Simulation study}

	We will depart slightly from the setting investigated in the main paper for practical reasons. The theoretical results apply to any sequence of samples satisfying the stated conditions. There is no need to consider sampling or any other form of random data generation apart from treatment assignment. It may, however, be hard to pick a specific sequence of samples in a representative and objective way. To avoid the issue, we specify data generating processes and resample from them in the simulations. Since the theoretical results apply to each individual draw from these processes, we know that they apply to the average of the draws as well. A consequence of introducing sampling is that the expectation operators in this supplement are taken over both the randomization and sampling distributions.

	The study includes three data generating processes which illustrate the performance under different types of interference structures. The potential outcomes in all three settings are generated as:
	\begin{equation}
	y_i(\z) = \left\{ \begin{array}{ll}
	2z_i + X_i + \varepsilon_i & \text{if } \bal{G_i} > 0,
	\\[0.5em]
	z_i + X_i + \varepsilon_i & \text{if } \bal{G_i} = 0,
	\\[0.5em]
	X_i + \varepsilon_i & \text{if } \bal{G_i} < 0,
	\end{array}\right.
	\end{equation}
	where $X_i$ represents an observed covariate that is uniformly distributed on the interval $[0, 3]$ and $\varepsilon_i$ represents an error term that is uniformly distributed on the interval $[0, 7]$. The variables are independent both within and between units. The covariate $X_i$ is only used to construct pairs when treatment is assigned with paired randomization. Interference is introduced through $\bal{G_i}$. For each unit, an \textit{interference set} $G_i$ is generated as a subset of $\bfU \setminus i$. The function $\bal{G_i}$ is then defined as the balance of treatment assignments of the units whose indices are in the interference set:
	\begin{equation}
	\bal{G_i} = \sum_{j\in G_i} \paren{2 Z_j - 1}.
	\end{equation}
	The function returns a positive value if a majority of the units in $G_i$ are treated; a negative value if a majority is in control; and zero if the same number of units are treated and in control. We set $\bal{G_i} = 0$ when $G_i$ is empty. The construction of $G_i$ decides the interference structure. The potential outcomes are such that $I_{ij}=1$ if and only if $i\in G_j$, and, consequently, $\di_{ij}=1$ if and only if a unit $\ell$ exists such that $\ell\in G_i \cap G_j$.

	The three data generating processes differ in how they generate $G_i$. In particular, we investigate when interference is contained within groups of units, when the interference structure is randomly generated and when only one unit is interfering with other units. The three types of structures are illustrated in Figure \ref{fig:simulation-settings}. For each structure, we investigate different amounts of interference dependence, ranging from when the no-interference assumption holds to when the interference is so excessive that Assumption \ref{ass:restricted-interference} does not hold.

	\begin{figure}
		\centering
		\includegraphics[width=\textwidth]{\figurepath{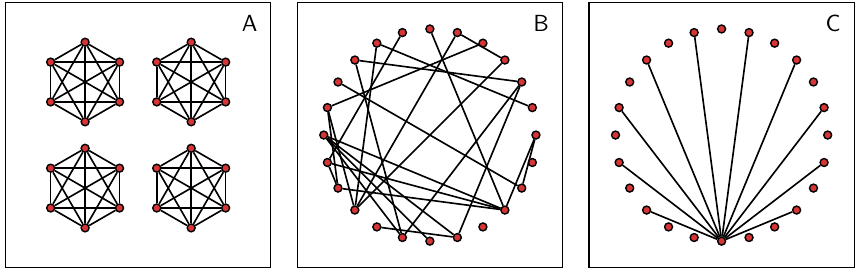}}
		\caption{The three types of interference structures investigated in the simulation study. The vertices represent units and the edges denote $I_{ij}=1$ or $I_{ji}=1$ for the connected units. The structures differ in the denseness of edges relative to the amount of interference dependence. In Panel A, the interference is restricted to groups of units so that all units in a group are interfering with each other as investigated in Section \ref{sec:interference-groups}. This is the densest possible interference structure for a given value of $\davg$. Panel B presents a randomly generated structure as investigated in Section \ref{sec:interference-random}. Panel C presents the setting where a single unit interferes with other units as investigated in Section \ref{sec:interference-oneunit}. This produces the sparsest interference structure for a given value of $\davg$.} \label{fig:simulation-settings}
	\end{figure}

	The potential outcomes are generated so that $\E[\big]{\tau_i(\Z_{-i}) \given G_i, X_i, \varepsilon_i}=1$ for all $i$ under any design such that $\bal{G_i}$ is symmetrically distributed around $0$ (which holds for all designs we consider in these simulations). We, thus, know that $\eEATE=1$ throughout the simulations. We have $\Var{Y_i \given G_i, \Z} \approx 4.8$ for all $i$, so the treatment effect is approximately half of a conditional standard deviation of the outcome variable.

	We focus on the experimental designs discussed in the main paper, namely the Bernoulli, complete and paired randomization designs. The pairs in the last design are constructed using the observed covariate. The units are ranked by $X_i$, and pairs are formed as adjacent units. That is, the two units with the lowest values of $X_i$ form one pair, the next two units form another pair, and so on. This mirrors the typical use of paired randomization to improve precision. We investigate samples ranging from 100 to 100,000 units with constant ratio of $10^{1/8}$ between steps: $n\in\braces{10^{2 + x/8}: x = 0, 1, \cdots, 24}$. The sample sizes are rounded to an even number to accommodate paired randomization.

	The primary focus in the simulations is the root mean square error (\RMSE) of the H\'ajek estimator. The results for other statistics and for the Horvitz-Thompson estimator are presented in the tables at the end of this supplement. To summarize these results, the \HT\ estimator performs worse than the \HA\ estimator, especially for the Bernoulli design, but the qualitative conclusions are the same. The bias is negligible relative to the variance for the settings and sample sizes we consider here.

	\subsection{Interference within groups} \label{sec:interference-groups}

	The first type of interference structure we consider is inspired by the commonly used partial interference assumption. Let $G_i = \braces{j : j \neq i \text{ and } \ceil{j / a_n} = \ceil{i / a_n}}$ be the interference set for $i$, where $a_n$ determines the size of $G_i$. The sets are symmetrical in the sense that $i\in G_j$ if and only if $j\in G_i$. One can interpret this setting as if the units are assigned to groups indexed by $\ceil{i / a_n}$, and all units in a group are interfering with each other.

	The size of the groups decides the amount of interference, and since the sequence $(a_n)$ controls the group size, we can use it to control $\davg$. In particular, we have $I_{ij}=\di_{ij}$, so $\icoutavg = \davg \approx a_n$ subject only to rounding error. Lemma \ref{lem:interference-measure-inequalities} gives $\icoutavg \leq \davg$, so $\icoutavg = \davg$ is the maximum value of $\icoutavg$. The data generating process thus produces the densest possible interference structure in the sense of the largest $\icoutavg$ for a given $\davg$.

	The results are presented in Figure \ref{fig:groupsum-interference}. The figure presents five different sequences of $(a_n)$. In the first case, $a_n$ is constant at one. All units are here assigned to a group of their own, and there is no interference. We observe the expected root-$n$ rate decrease in the \RMSE. The performances of Bernoulli and complete randomization are close to identical, but paired randomization provides a notable efficiency gain. The four subsequent settings introduce interference at different levels. All settings start with $a_{100}\approx 25$, after which the groups grow at different rates as described in the legend of the figure.

	In the first case with interference, we have $a_{n}=25$, so the groups are of a constant size of $25$ (disregarding rounding). We see that the \RMSE\ is considerably higher than under no interference for the smaller sample sizes, especially for the Bernoulli design. As the sample grows, however, the \RMSE\ approaches the performance under no interference. We know from Corollary \ref{coro:bernoulli-complete-root-n} that the H\'ajek estimator is root-$n$ consistent when the average interference dependence is bounded. We have $\davg=25$ when $a_{n}=25$, so we, indeed, expect the \RMSE\ to be of the same order as without interference in this case. The theoretical results are mirrored in the simulations.

	\begin{figure}
		\centering
		\includegraphics[width=\textwidth]{\simulationpath{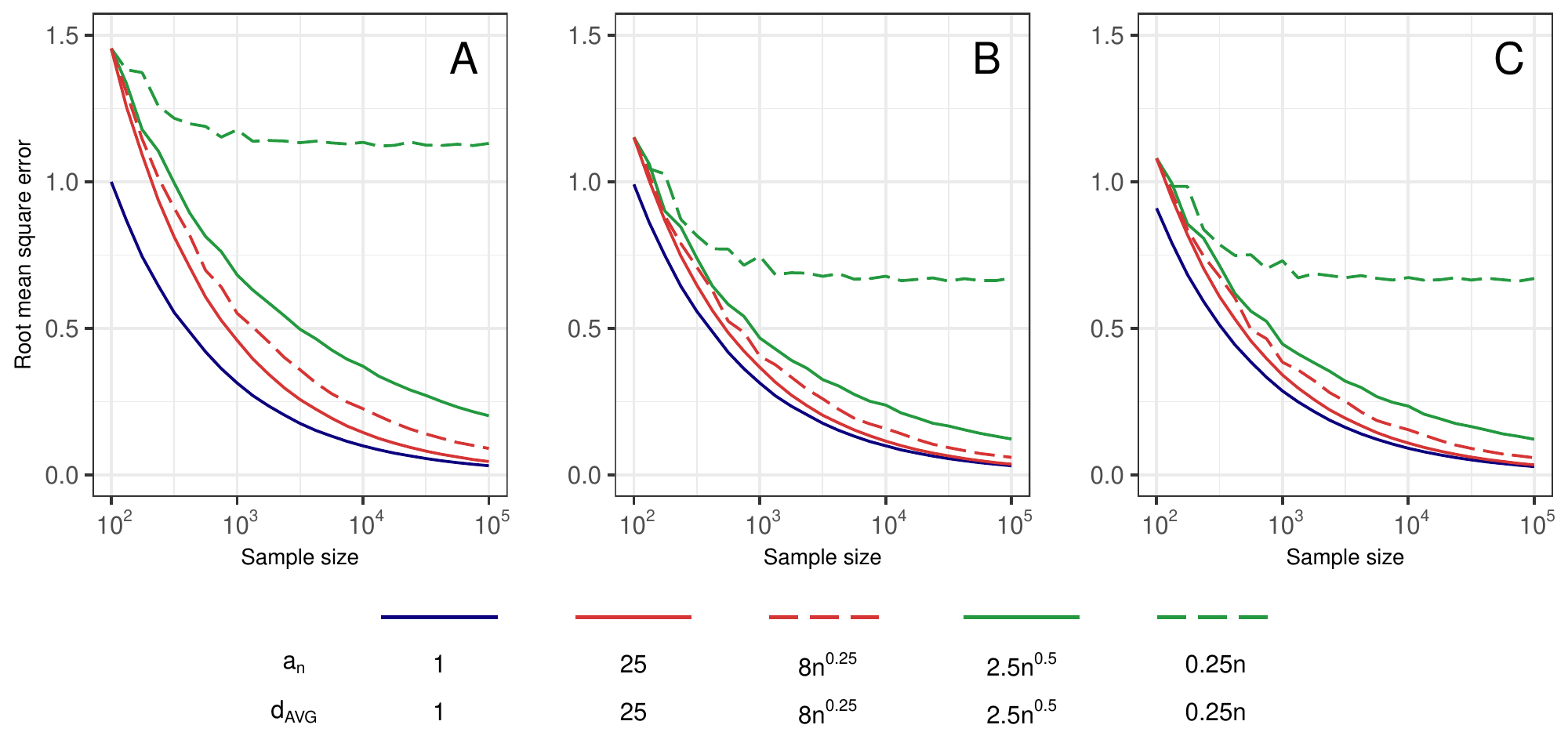}}
		\caption{Estimator performance when the interference is contained within groups of units. The figure presents the root mean square error of the H\'ajek estimator with respect to \EATE\ for the interference structure described in Section \ref{sec:interference-groups}. The panels present the results from the Bernoulli (A), complete (B) and paired (C) randomization designs. The lines correspond to different sequences $(a_n)$, as described in the legend, which govern the amount of interference dependence. Each line connects 25 data points with sample sizes evenly spaced on a logarithmic scale between 100 and 100,000. Each data point is the \RMSE\ averaged over 50,000 draws from the corresponding data generating process normalized by the average \RMSE\ from the Bernoulli design with sample size 100 under no interference. The simulation error is negligible; the non-smoothness of the \RMSE\ curve for the smaller sample sizes is due to rounding of $a_n$ when constructing the groups.} \label{fig:groupsum-interference}
	\end{figure}

	The two subsequent cases, $a_{n}=8n^{0.25}$ and $a_{n}=2.5n^{0.5}$, investigate the performance when $\davg$ grows with the sample size but at a rate so that restricted interference holds. We know from the theoretical results that the estimator is consistent, but the rate of convergence might be slower than root-$n$. This is reflected in the simulations. The \RMSE\ decreases as the sample grows, but it does so at a slower rate than under no interference. Proportionally, the \RMSE\ is considerably larger than under no interference for the larger sample sizes. This illustrates that large samples may be needed to draw inferences about \EATE\ in settings with considerable interference. Note that complete randomization performs better than the Bernoulli design despite having a worse upper bound on the rate of convergence in this case.

	The last sequence investigates the performance when restricted interference does not hold. Here, $a_{n}=0.25n$, so the data generating process partitions the units into four groups of equal size. The number of units in the groups grows proportionally to the sample size. In this case, $\davg$ is not dominated by $n$, and the theoretical results do not apply; the estimator may not be consistent. Indeed, while the \RMSE\ decreases initially, it levels out, and for samples larger than about a thousand units, there is no noticeable decrease in the \RMSE.

	\subsection{Random interference} \label{sec:interference-random}

	The second type of interference structure randomly generates $G_i$. We know of no way to draw interference structures uniformly from all structures with a specific value of $\davg$. The task is not completely unlike the problem of uniformly generating random regular graphs \citep{Janson2000Random}. To make the data generation tractable, we use a simple procedure that controls the expected value of $\davg$. We generate $G_i$ so that, independently, $j\in G_i$ with probability $(a_n - 1) / (n - 1)$. This corresponds to Gilbert's version of the Erdős–Rényi model of random graphs \citep{Gilbert1959Random}. The marginal distribution of $\icout_i  - 1$ is, thus, binomial with $n-1$ trials of probability $(a_n - 1) / (n - 1)$, so $\E{\icout_i} = \E{\icoutavg} = a_n$. In Section~\ref{sec:exp-of-davg-in-sim} of this supplement, we show that:
	\begin{equation}
	\E{\davg} = n - (n-1)\paren[\bigg]{1 - \frac{a_n - 1}{n - 1}}^{n}\paren[\bigg]{1 + \frac{a_n - 1}{n - 1}}^{n-2}.
	\end{equation}
	The expectation is of the same order as the square of $a_n$ when the sequence is dominated by $n^{0.5}$, and it is of order $n$ when $a_n$ dominates $n^{0.5}$:
	\begin{equation}
	\E{\davg} \sim \left\{ \begin{array}{ll}
	a_n^2 & \text{if } a_n = \littleO{n^{0.5}},
	\\[0.5em]
	n & \text{if } n^{0.5} = \littleO{a_n}.
	\end{array}\right.
	\end{equation}
	In other words, we can control the expected value of $\davg$ with the sequence $(a_n)$. We investigate five sequences $(a_n)$ that yield growth rates of $\davg$ roughly similar to the five settings investigated in the previous section.

	\begin{figure}
		\centering
		\includegraphics[width=\textwidth]{\simulationpath{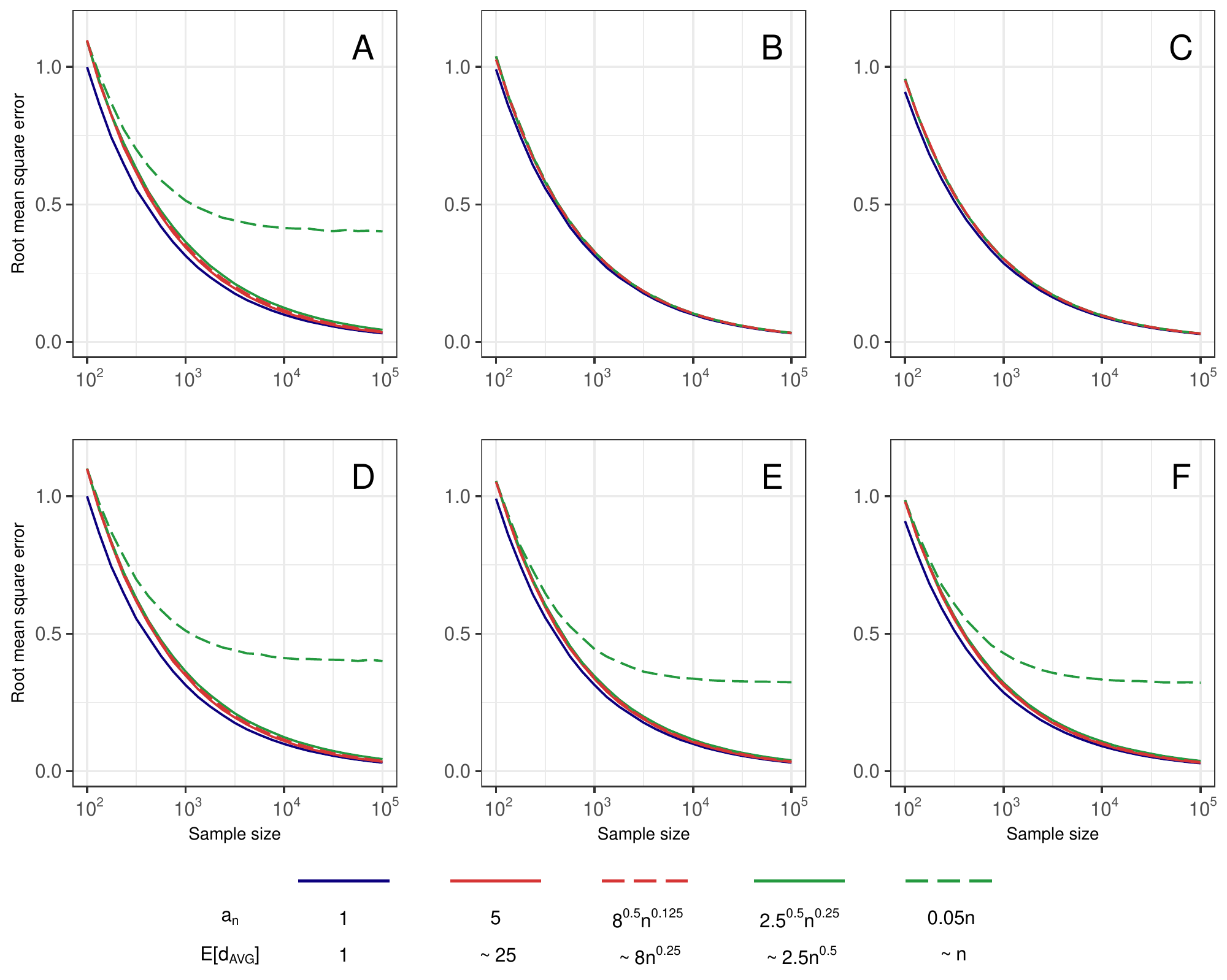}}
		\caption{Estimator performance when the interference structure is randomly generated. The figure presents the root mean square error of the H\'ajek estimator with respect to \EATE\ for the interference structure described in Section \ref{sec:interference-random}. The panels present the results from the Bernoulli (A, D), complete (B, E) and paired (C, F) randomization designs for the unweighted (A, B, C) and weighted (D, E, F) version of the balance function as described in the text. See the note of Figure \ref{fig:groupsum-interference} for additional details.} \label{fig:random-interference}
	\end{figure}

	The results are presented in the first row of panels in Figure \ref{fig:random-interference}. As in the previous setting, when the interference dependence is bounded (i.e., $a_n = 5$), the performance is similar to the case without interference, $a_n = 1$. However, unlike the previous setting, all sequences of $a_n$ exhibit this behavior. The one exception is when $a_n=0.05n$ under the Bernoulli design in Panel A. The behavior is somewhat artificial due to the stylized data generating process, but it does illustrate that the theoretical results describe the worst case; the estimators may converge at a rapid rate even when the interference is extensive.

	To understand this behavior, consider the consequences of a small perturbation of the assignments. In particular, consider changing the assignment of a single unit $i$ assigned to control under the Bernoulli design. Let $\mathcal{G}_i = \braces{j: i\in G_j}$ denote the set of units for which unit $i$ is in their respective interference sets. Of the units in $\mathcal{G}_i$, some fraction will have $\bal{G_j} = 0$. When we change $i$'s assignment, those units will change to $\bal{G_j} = 1$, and thus change the values of their outcomes. Similarly, some other fraction of $\mathcal{G}_i$ will have $\bal{G_j} = -1$, and their values will change to $\bal{G_j} = 0$. If unit $i$ is in sufficiently many units' interference sets, which tends to happen when $a_n$ is large, the perturbations will have large effects on the estimator. This explains why the \RMSE\ does not converge to zero when $a_n=0.05n$ under the Bernoulli design.

	Now consider the same situation under complete and paired randomization. With these designs, we cannot change unit $i$'s assignment in isolation. The number of treated units is fixed, so to change $i$'s assignment, we must change another unit's assignment in the opposite direction. This other unit, say unit $\ell$, will, just as $i$, be in some number of other units' interference sets, and some fraction of these units will have $\bal{G_j} \in\braces{0,1}$. The consequence is that when we change $i$'s assignment from control to treatment and add one to $\bal{G_j}$ for the units in $\mathcal{G}_i$, a counteracting effect happens because we simultaneously change $\ell$'s assignment from treatment to control and remove one from $\bal{G_j}$ for units in $\mathcal{G}_\ell$. The number of units that get one added to $\bal{G_j}$ will, with high probability and relative to the sample size, be close to the number of units that get one removed, so the overall effect is small even if units $i$ and $\ell$ interfere many other units.

	The extreme case is when $a_n=n$, so that interference is complete. The Bernoulli design would here exhibit considerable variation just as when $a_n=0.05n$. However, under complete and paired randomization, exactly half of the sample has $\bal{G_j} =-1$, and the other half has $\bal{G_j}=1$. No perturbation of the assignments can change this. In other words, the designs perfectly balance the spillover effects of any change of the assignments so the outcomes are unchanged, and the standard error does not increase relative to no interference (see Table \ref{tab:random-sd} below).

	The data generating process can be modified slightly to introduce sensitivity to assignment perturbations also under complete and paired randomization. Consider generating a random variable $\lambda_i$ for each unit that follows a heavy-tailed distribution. In our case, $\lambda_i$ is standard log-normal distributed. We can then change the definition of $\bal{G_i}$ to a weighted balance:
	\begin{equation}
	\bal{G_i} = \sum_{j\in G_i} \lambda_j\paren{2 Z_j - 1}.
	\end{equation}
	If we perturb the assignment of a unit with a large $\lambda_i$, the effect will outweigh the effect of the counteracting change of another unit's assignment, since the other unit tends to have a lower value on $\lambda_i$. If the unit with the large $\lambda_i$ is interfering with sufficiently many units (i.e., when $a_n$ is large), the change will noticeably affect the behavior of the estimator. The second row of panels in Figure \ref{fig:random-interference} presents the results when we use the weighted version of $\bal{G_i}$. We see that, indeed, the \RMSE\ for $a_n=0.05n$ does not converge under any of the designs.

	\subsection{Interference from one unit} \label{sec:interference-oneunit}

	The final type of interference structure to be considered is when a single unit interferes with other units. Specifically, let $G_i = \braces{1}$ when $2\leq i \leq a_n$ and $G_i = \emptyset$ otherwise. A simpler, but equivalent, way to generate the potential outcomes in this case is:
	\begin{equation}
	y_i(\z) = \left\{ \begin{array}{ll}
	2z_1z_i + X_i + \varepsilon_i & \text{if } 2 \leq i \leq a_n,
	\\[0.5em]
	z_i + X_i + \varepsilon_i & \text{otherwise}.
	\end{array}\right.
	\end{equation}

	We have $I_{ij}=1$ if and only if $i=1$ and $j \leq a_n$, and thus $\di_{ij}=1$ whenever $i,j \leq a_n$. It follows that $\icoutavg = 1 + \paren[\big]{\floor{a_n} - 1}  / n$ and $\davg = 1 + \floor{a_n}\paren[\big]{\floor{a_n} - 1} / n$. This provides the maximum amplification of the variation in unit 1's assignment. In all cases we investigate here, $\icoutavg< 2$ no matter the value of $\davg$. The setting is thus a candidate for the worst case since a very sparse interference structure leads to many interference dependencies (i.e., a small $\icoutavg$ but large $\davg$).

	\begin{figure}
		\centering
		\includegraphics[width=\textwidth]{\simulationpath{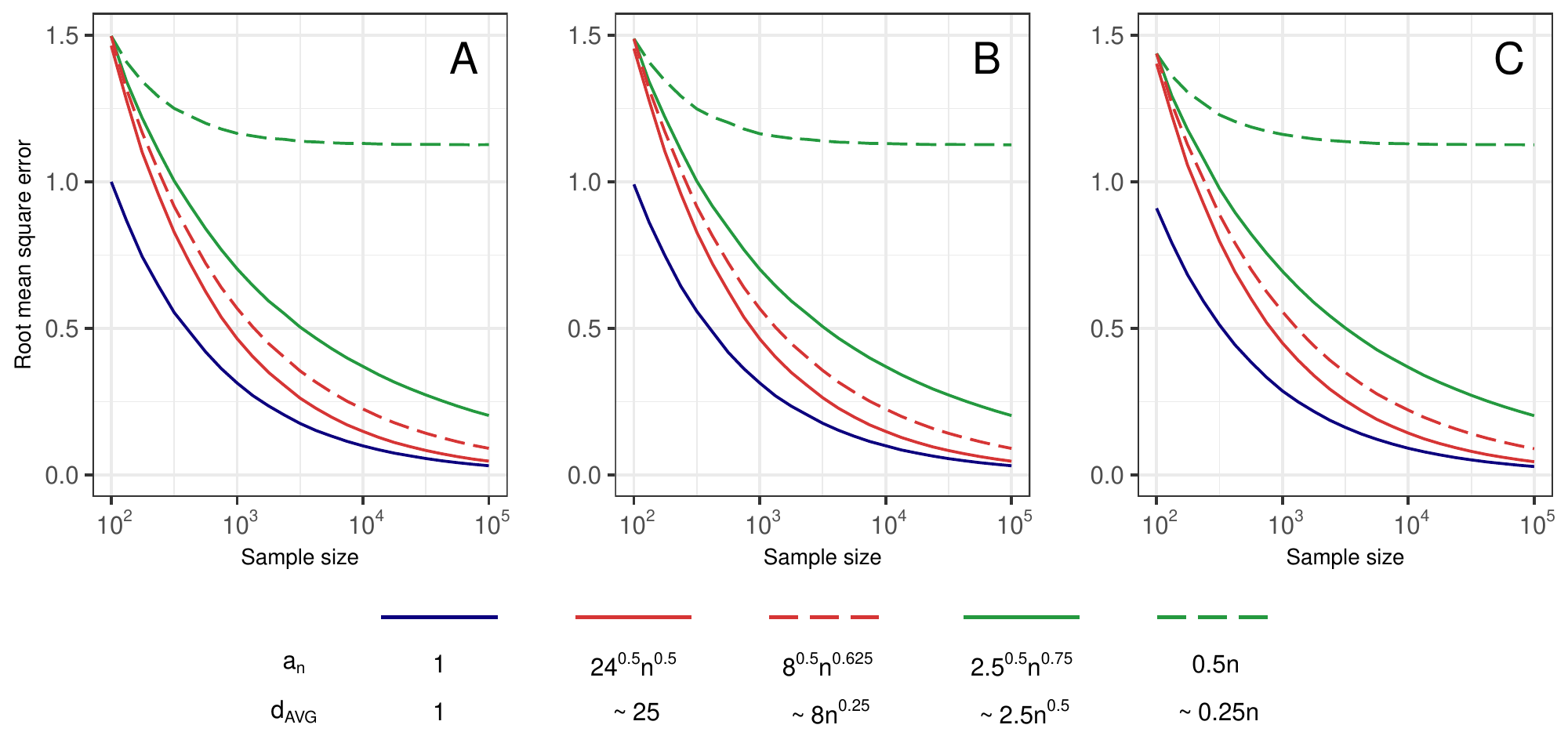}}
		\caption{Estimator performance when only one unit is interfering with other units. The figure presents the root mean square error of the H\'ajek estimator with respect to \EATE\ for the interference structure described in Section \ref{sec:interference-oneunit}. The panels present the results from the Bernoulli (A), complete (B) and paired (C) randomization designs. See the note of Figure \ref{fig:groupsum-interference} for additional details.} \label{fig:oneunit-interference}
	\end{figure}

	Figure \ref{fig:oneunit-interference} presents the results. The performance is similar across the three designs. For bounded interference (i.e., $a_n = 24^{0.5}n^{0.5}$), the \RMSE\ quickly approaches the no-interference case. However, compared to the two other types of interference structures, there is a noticeable difference also for the largest sample sizes. With this data generating process, a sample containing one hundred thousand units is not large enough to remove the performance penalty introduced by the interference. The two sequences with restricted but not bounded interference ($a_n = 8^{0.5}n^{0.625}$ and $a_n = 2.5^{0.5}n^{0.75}$) follow the same pattern as above: the \RMSE\ decreases but at a slower rate than under no interference. When the interference dependence is proportional to the sample size ($a_n = 0.5n$), the \RMSE\ levels out, and we see no notable decrease for samples larger than about one thousand units. These results confirm the theoretical finding: we must use $\davg$ rather than $\icoutavg$ to restrict the interference.

\section{The average interference dependence in Section \ref{sec:interference-random}}\label{sec:exp-of-davg-in-sim}

	\begin{nonum-proposition}
		With the data generating process in Section \ref{sec:interference-random}:
		\begin{equation}
		\E{\davg} = n - (n-1)\paren[\bigg]{1 - \frac{a_n - 1}{n - 1}}^{n}\paren[\bigg]{1 + \frac{a_n - 1}{n - 1}}^{n-2}.
		\end{equation}
	\end{nonum-proposition}

	\begin{proof}
		Recall $\davg$ from Definition \ref{def:interference-dependence}, and note:
		\begin{equation}
		\E{\davg} = \frac{1}{n}\sum_{i=1}^{n}\sum_{j=1}^n\E{\di_{ij}} = \frac{1}{n}\sum_{i=1}^{n}\sum_{j=1}^n\prob{\di_{ij} = 1}.
		\end{equation}
		Recall that $\di_{ij}=1$ if and only if $\ell \in \bfU$ exists such that $I_{\ell i}I_{\ell j} = 1$. Consider the probability:
		\begin{multline}
		\prob{\di_{ij} = 1} = \prob[\Bigg]{\sum_{\ell = 1}^n I_{\ell i}I_{\ell j} > 0} = 1 - \prob[\Bigg]{\sum_{\ell = 1}^n I_{\ell i}I_{\ell j} = 0}
		\\
		= 1 - \prod_{\ell = 1}^{n}\prob[\big]{I_{\ell i}I_{\ell j} = 0},
		\end{multline}
		where the last equality follows from that all the terms in the sum are independent under the current data generating process.

		Consider the probability that $I_{\ell i}I_{\ell j} = 0$:
		\begin{equation}
		\prob[\big]{I_{\ell i}I_{\ell j} = 0} = 1 - \prob[\big]{I_{\ell i} = 1, I_{\ell j} = 1} = 1 - \prob[\big]{I_{\ell i} = 1\given I_{\ell j} = 1}\prob[\big]{I_{\ell j} = 1}.
		\end{equation}
		The distribution of $I_{i j}$ is specified by the data generating process. Specifically, recall that $I_{ij} = 1$ with probability one when $i=j$, and $I_{ij}$ is Bernoulli distributed with probability $(a_n - 1) / (n - 1)$ when $i\neq j$. We thereby have:
		\begin{align}
		\prob[\big]{I_{\ell j} = 1} &= \left\{\begin{array}{ll}
		1 & \text{if } \ell = j,
		\\[0.2em]
		(a_n - 1) / (n - 1) & \text{otherwise},
		\end{array}\right.
		\\[0.8em]
		\prob[\big]{I_{\ell i} = 1\given I_{\ell j} = 1} &= \left\{\begin{array}{ll}
		1 & \text{if } i = j,
		\\[0.2em]
		\prob[\big]{I_{\ell i} = 1} & \text{otherwise},
		\end{array}\right.
		\\[0.8em]
		\prob[\big]{I_{\ell i} = 1\given I_{\ell j} = 1}\prob[\big]{I_{\ell j} = 1} &= \left\{\begin{array}{ll}
		1 & \text{if } i=j=\ell,
		\\[0.2em]
		(a_n - 1) / (n - 1) & \text{if } i=j \text{ or } i=\ell \text{ or } j=\ell,
		\\[0.2em]
		\bracket[\big]{(a_n - 1) / (n - 1)}^2 & \text{otherwise},
		\end{array}\right.
		\end{align}

		Note that:
		\begin{equation}
		1 - \paren[\bigg]{\frac{a_n - 1}{n - 1}}^{2} = \paren[\bigg]{1 - \frac{a_n - 1}{n - 1}}\paren[\bigg]{1 + \frac{a_n - 1}{n - 1}},
		\end{equation}
		so:
		\begin{align}
		\frac{\prob[\big]{I_{\ell i}I_{\ell j} = 0}}{1 - (a_n - 1) / (n - 1)}
		&= \left\{\begin{array}{ll}
		0 & \text{if } i=j=\ell,
		\\[0.2em]
		1 & \text{if } i=j \text{ or } i=\ell \text{ or } j=\ell,
		\\[0.2em]
		\bracket[\big]{1 + (a_n - 1) / (n - 1)} & \text{otherwise},
		\end{array}\right.
		\\[0.8em]
		\frac{\prod_{\ell = 1}^{n}\prob[\big]{I_{\ell i}I_{\ell j} = 0}}{\bracket[\big]{1 - (a_n - 1) / (n - 1)}^n}
		&= \left\{\begin{array}{ll}
		0 & \text{if } i=j,
		\\[0.2em]
		\bracket[\big]{1 + (a_n - 1) / (n - 1)}^{n-2} & \text{otherwise}.
		\end{array}\right.
		\end{align}
		Thus, when $i=j$, we have $\prob{\di_{ij} = 1} = 1$, and when $i\neq j$:
		\begin{multline}
		\prob{\di_{ij} = 1} = 1 - \prod_{\ell = 1}^{n}\prob[\big]{I_{\ell i}I_{\ell j} = 0}
		\\
		= 1 - \paren[\bigg]{1 - \frac{a_n - 1}{n - 1}}^{n}\paren[\bigg]{1 + \frac{a_n - 1}{n - 1}}^{n-2}.
		\end{multline}
		We return to the expectation of interest to complete the proof:
		\begin{align}
		\E{\davg} &= \frac{1}{n}\sum_{i=1}^{n}\sum_{j=1}^n\prob{\di_{ij} = 1}
		\\
		&= \frac{1}{n}\sum_{i=1}^{n}\prob{\di_{ii} = 1} +  \frac{1}{n}\sum_{i=1}^{n}\sum_{j\neq i}\prob{\di_{ij} = 1}
		\\
		&= \frac{1}{n}\sum_{i=1}^{n}1 +   \frac{1}{n}\sum_{i=1}^{n}\sum_{j\neq i}\bracket[\Bigg]{1 - \paren[\bigg]{1 - \frac{a_n - 1}{n - 1}}^{n}\paren[\bigg]{1 + \frac{a_n - 1}{n - 1}}^{n-2}}
		\\
		&= 1 +  (n-1)\bracket[\Bigg]{1 - \paren[\bigg]{1 - \frac{a_n - 1}{n - 1}}^{n}\paren[\bigg]{1 + \frac{a_n - 1}{n - 1}}^{n-2}}
		\\
		&= n - (n-1)\paren[\bigg]{1 - \frac{a_n - 1}{n - 1}}^{n}\paren[\bigg]{1 + \frac{a_n - 1}{n - 1}}^{n-2}. \qedhere
		\end{align}
	\end{proof}

\providecommand{\maketable}[3]{
	\begin{table}
		\caption{#1} \label{#2}
		\centering
		\renewcommand{\arraystretch}{1.2}
		\vspace{0.08in}
		\resizebox{\textwidth}{!}{ %
			\setlength{\tabcolsep}{12pt}
			\input{#3}
		}
	\end{table}
}

\providecommand{\panelsep}{0.6em}

\newcommand{\biastext}[1]{Bias of the \HT\ and \HA\ estimators with respect to \EATE\ #1}

\newcommand{\sdtext}[1]{Standard error of the \HT\ and \HA\ estimators #1}

\newcommand{\twotables}[2]{
	
	\begin{table}
		\caption{\biastext{#1}} \label{tab:#2-bias}
		\centering
		\renewcommand{\arraystretch}{1.2}
		\vspace{0.08in}
		\resizebox{\textwidth}{!}{ %
			\setlength{\tabcolsep}{12pt}
			\input{\simulationpath{tab-#2-bias.tex}}
		}
	\end{table}

	\begin{table}
		\caption{\sdtext{#1}} \label{tab:#2-sd}
		\centering
		\renewcommand{\arraystretch}{1.2}
		\vspace{0.08in}
		\resizebox{\textwidth}{!}{ %
			\setlength{\tabcolsep}{12pt}
			\input{\simulationpath{tab-#2-sd.tex}}
		}
	\end{table}

}

\twotables{in the simulation study when the interference is contained within groups of units as described in Section \ref{sec:interference-groups}.}{groupsum}

\twotables{when the interference structure is randomly generated using the unweighted balance function as described in Section \ref{sec:interference-random}.}{random}

\twotables{when the interference structure is randomly generated using the weighted balance function as described in the end of Section \ref{sec:interference-random}.}{randomweights}

\twotables{when only one unit is interfering with other units as described in Section \ref{sec:interference-oneunit}.}{oneunit}

\end{document}